\documentclass[11pt]{amsart}
\usepackage{amssymb,amsfonts,dsfont,cancel}
\usepackage{graphicx}
\usepackage{amsmath}
\usepackage{latexsym}
\usepackage{color}
\usepackage{comment}
\usepackage[mathscr]{euscript}
\newtheorem{theorem}{Theorem}[section]
\theoremstyle{plain}
\newtheorem{lemma}[theorem]{Lemma}
\newtheorem{proposition}[theorem]{Proposition}

\newtheorem{corollary}[theorem]{Corollary}

\theoremstyle{remark}
\newtheorem{definition}[theorem]{Definition}
\newtheorem{example}[theorem]{Example}

\newtheorem{problem}{Problem}
\newtheorem{remark}[theorem]{Remark}

\numberwithin{equation}{section}
\newcommand{\ext}[1]{\textup{ext}\left(#1\right)}

\newcommand{\norms}[1]{\left|\!\left|\!\left|#1 \right|\!\right|\!\right|_E}
\newcommand{\normscomm}[1]{\left|\!\left|\!\left|#1 \right|\!\right|\!\right|_{E(\mathcal{M},\tau)}}
\DeclareMathOperator{\supp}{supp}

\newcommand{\M}{\mathcal{M}}

\newcommand{\Complex}{\mathbb{C}}

\newcommand{\abs}[1]{\left\vert#1\right\vert}

\newcommand{\norm}[1]{\|#1\|}
\newcommand{\Ran}[1]{\text{Ran}\,#1}
\newcommand{\norme}[1]{\|#1\|_E}
\newcommand{\normf}[1]{\|#1\|_F}
\newcommand{\normcomm}[1]{\|#1\|_{E(\textit{M},\tau)}}
\newcommand{\Ker}[1]{\text{Ker}\,#1}
\newcommand{\Kerp}[1]{\text{Ker}^\perp#1}
\newcommand{\nonsp}{E(\mathcal{M},\tau)}
\newcommand{\Rep}[1]{\text{Re}\,#1}
\newcommand{\tr}{\textup{tr}}
\newcommand{\Imp}[1]{\text{Im}\,#1}

\def\underset#1\to#2{\mathop{#2}\limits_{#1}{ }}
\def\overset#1\to#2{\mathop{#2}\limits^{#1}{ }}

\newcommand{\one}{\textup{\textbf{1}}}
\newcommand{\tauone}{\tau(\textup{\textbf{1}})}
\newcommand{\Mtau}{\left(\mathcal{M},\tau\right)}

\begin{document}
\title[Geometry of   symmetric spaces of measurable operators]
{Geometric properties of noncommutative symmetric spaces of measurable operators and unitary matrix ideals}

\author{M.M. Czerwi\'nska}
\address{ Department of Mathematics and Statistics, University of North Florida, Jacksonville, FL 32224} \email{m.czerwinska@unf.edu}

\author
{A. Kami\'nska}
\address{Department of Mathematical Sciences, The University of
Memphis, Memphis, TN 38152} \email{kaminska@memphis.edu}

\thanks {\emph{2010 subject classification}\ {46B20, 46B28, 47L05,
47L20}}

\keywords{Symmetric spaces of measurable operators, unitary matrix spaces, rearrangement invariant spaces, $k$-extreme points, $k$-convexity, complex extreme points, complex convexity, monotonicity, (local) uniform (complex and real) convexity, $p$-convexity (concavity), (strong) smoothness, (strongly) exposed points, (uniform) Kadec-Klee properties, Banach-Saks properties, Radon-Nikod\'ym property, Krivine-Maurey stability }

\maketitle
\begin{abstract} This is a survey article  of geometric properties of  noncommutative symmetric spaces  of measurable operators $\nonsp$,  where $\M$ is a semifinite von Neumann algebra with a faithful, normal, semifinite trace $\tau$,  and $E$ is a symmetric function  space.  If $E\subset c_0$ is a symmetric sequence space then the analogous properties in the  unitary matrix ideals $C_E$ are also presented. In the preliminaries we  provide   basic  definitions and concepts illustrated by some examples and occasional proofs.  In particular we list and discuss  the properties of general singular value function, submajorization in the sense of Hardy, Littlewood and P\'olya, K\"othe duality, the spaces $L_p\Mtau$, $1\le p<\infty$,  the identification between $C_E$ and $G(B(H), \rm{tr})$ for some symmetric function space $G$, the commutative case when  $E$ is identified with $E(\mathcal{N}, \tau)$ for $\mathcal{N}$ isometric to $L_\infty$ with the standard integral trace,  trace preserving $*$-isomorphisms between $E$ and a $*$-subalgebra of $E\Mtau$, and a general method of removing  the assumption of non-atomicity of $\M$.
The main results on geometric properties are given in separate sections.  We present the results on (complex) extreme points, (complex) strict convexity, strong extreme points and midpoint local uniform convexity, $k$-extreme points and $k$-convexity,  (complex or local) uniform convexity, smoothness and strong smoothness, (strongly) exposed points, (uniform) Kadec-Klee properties, Banach-Saks properties, Radon-Nikod\'ym property and stability in the sense of Krivine-Maurey. We also state some open problems.

\end{abstract}

In 1937, John von Neumann  \cite{VN} pp. 205-218, observed that for a symmetric norm $\|\cdot\|$ in $\mathbb{R}^n$, it is possible to define a norm on the space of $n\times  n$ matrices $x$ by setting $\|x\| = \|\{s_i(x)\}_{i=1}^n\|$, where $s_i(x)$, $i=1,2\dots,n$, are eigenvalues of the matrix $|x| = (x^*x)^{1/2}$  ordered in a decreasing manner. Later on in the forties and fifties, J. von Neumann and R. Schatten developed analogous theory for infinite dimensional compact operators. They defined and studied  unitary matrix ideals $C_E$ corresponding to  a symmetric sequence Banach space $(E, \|\cdot\|_E)$. The space   consists of all compact operators $x$ on a Hilbert space such that $\{s_n(x)\}\subset E$ with the norm  $\|x\|=\|\{s_n(x)\}\|_E $, where $s_n(x)$, $n\in\mathbb{N}$, are singular numbers of $x$, that is  eigenvalues of $|x|$. For $E=\ell_1$, the space $C_E$ is called the trace class of operators or the space of nuclear operators, while if $E = \ell_2$ then it is called the class of Hilbert-Schmidt operators. The first monograph of these spaces was written by R. Schatten in 1960 \cite{Sch}, and later on in 1969 by I. C. Gohberg and M. G. Krein \cite{GK}. In 1967, C. McCarthy wrote an article on the now called the Schatten classes $C_p$, $0<p\le \infty$, that is the spaces $C_E$ when $E=\ell_p$, and showed among others that this space is uniformly convex for $1<p<\infty$ \cite{chmc}. 
The beginning of the theory of symmetric spaces of measurable operators can be traced back to the early fifties. It was then when I. Segal and J. Dixmier \cite{Seg, Dix} laid out the foundation for noncommutative $L_p\Mtau$ spaces, $0< p <\infty$,  by  introducing the concept of noncommutative integration in the  settings of semifinite von Neumann  algebras $\M$ with traces $\tau$. Inspired by their work,  V. Ov{\v{c}}innikov in 1970 studied interpolation theory in the context of measurable operators \cite{Ov1, Ov2}. In his work the emphasis was placed on the rearrangement invariant structure of the spaces. The symmetric structure of the spaces was induced by a singular value function,  the generalization of singular numbers of compact operators, and  the theory of symmetric spaces of measurable operators was initiated.  F. Yeadon continued the studies of symmetric spaces of measurable operators in articles \cite{Y1,Y2, Y3}. It is worth noting that the notion of the  singular value function of the measurable operator was introduced in a Bourbaki seminar note by Grothendieck \cite{Gro}.  In 1989, P. G.~Dodds, T. K.~Dodds and B.~de Pagter \cite{DDPnoncomm,DDPMark} presented a more general construction of symmetric spaces of measurable operaotrs $\nonsp$. In particular they used the notion of measurablility introduced by E. Nelson \cite{Nelson}, which is significantly broader  than the one applied  by V. Ov{\v{c}}innikov  and  F. Yeadon. In fact, Nelson's  notion of $\tau$-measurability of the closed operator affiliated with a semifinite von Neumann algebra with a normal, faithful, semifinite trace $\tau$ is equivalent with requiring the operator to possess an everywhere finite decreasing rearrangement.

 In the past several decades  the theory of the spaces of the measurable operators has been extensively studied  and applied. It has attracted great attention of the well known specialists in functional analysis and operator theory as J. Arazy, V. I. Chilin, P. G. Dodds, T. K. Dodds,  U. Haagerup, M. Junge, N. Kalton, F. Lust-Piquard, B. De Pagter, G. Pisier, F. Sukochev,   Q. Xu  \cite{A1981, DDPnoncomm, Hag, JP, KS, L-PS, X1989, SCH1990, DD1993}, and others. The non-commutative $L_p\Mtau$ spaces, and more general non-commutative spaces of measurable operators $E(\mathcal{M},\tau)$, share many properties with the usual $L_p$ spaces, or symmetric spaces $E$, but on the other hand they are very different. They provide interesting examples that cannot exist among the usual function or sequence spaces. They are also used as fundamental tools in some other areas of mathematics such as operator algebra theory, non-commutative geometry and non-commutative probability, as well as in mathematical physics. A very interesting survey by G. Pisier and Q. Xu \cite{PX} classifies the similarities and differences between the usual $L_p$ spaces and their non-commutative counterparts. P. Dodds, B. De Pagter and F. Sukochev are in the process of writing a monograph  on the spaces  $E(\mathcal{M},\tau)$ \cite{noncomm}. We wish to thank them for making the manuscript available to us, which has been  a great help in studies those spaces and in particular in preparation of  this survey article.   

In the early eighties J. Arazy was  the first who started to study    the geometric properties in noncommutative  matrix ideals $C_E$, making a substantial contribution in this subject. He related the properties of the symmetric sequence space $E$ to the corresponding properties of $C_E$.   His ideas influenced later V. Chilin, A. Krygin and F. Sukochev \cite{CKS1992ext,CKS1992}  and 
Q. Xu \cite{X1992}, who initiated investigation of the relation between the properties of the symmetric function space $E$ and the properties of $E\Mtau$.  
 
The purpose of the article is to collect and present  a number of results on geometric properties of the spaces $E\Mtau$ and $C_E$ which were published in various journals in the past several  decades.  Several well known and important  properties have been already  studied like different types of convexities, smoothness, $KK$-properties, Radon-Nikod\'ym property, stability. However there are still plenty  of them which have not been investigated. We hope that this article will serve not only  as a source of the known results and their references but also as a motivation for further  studies of new properties and their applications.  

 The article is divided into a number of topic sections.  Although the proofs of most statements are  not given, there are some  for which we present the proofs. In particular we give the detailed proofs in the section \ref{sec:symmfun} on symmetric function spaces, where we interpret  the spaces  $E\Mtau$ in the commutative case. It is crucial for the readers to understand this basic liaison. We also extend section \ref{sec:isom} on trace preserving $*$-isomorphisms, by some more specific results which are necessary for  detailed studies of local geometric  properties. We are trying to give  exact references  of any statement presented here in an effort to make this article clear, readable  and possible to follow by  novices in noncommutative theory of measurable operators.

The article is divided into the following sections.
\begin{itemize}
\item[(1)]  Preliminaries.
\item[(2)] Examples of symmetric spaces of measurable operators.
\item[(3)] Trace preserving isomorphisms.
\item[(4)] Non-atomic extension of $E\Mtau$.
\item[(5)] Extreme points and strict convexity.
\item[(6)] Strongly extreme points and midpoint local uniform convexity.
\item[(7)] $k$-extreme points and $k$-convexity.
\item[(8)] Complex extreme points and complex convexity.
\item[(9)] Complex local uniform convexity.
\item[(10)] $p$-convexity and $q$-concavity.
\item[(11)] Uniform and local uniform convexity.
\item[(12)] Complex uniform convexity.
\item[(13)] Smoothness.
\item[(14)] Strong smoothness.
\item[(15)]  Exposed and strongly exposed points.
\item[(16)] Kadec-Klee properties.
\item[(17)] Uniform Kadec-Klee property.
\item[(18)] Banach-Saks properties.
\item[(19)] Radon-Nikod\'ym property.
\item[(20)] Stability in the sense of Krivine-Maurey.

\end{itemize}

\section{Preliminaries}

 Let $\Complex$, $\mathbb{R}$ and $\mathbb{N}$ denote the complex, real and natural numbers, respectively. The set of non-negative real numbers will be denoted by $\mathbb{R}^+$.

Let $H$ be a complex Hilbert space, $B(H)$ the space of bounded linear operators from $H$ to $H$ and $\M\subset B(H)$ be a von Neumann algebra on a Hilbert space $H$. 

A closed and densely defined linear operator $x: D (x)\rightarrow H$, where the domain $D(x)$ is a linear subspace of $H$, is called \emph{self-adjoint}
if $x^*=x$ and \emph{normal} if $x^*x=xx^*$, meaning that the domains  of the operators on  both sides of the equations coincide.  If in addition $\langle x\xi,\xi\rangle\geq 0$ for all $\xi\in D(x)$ then $x$ is said to be a \emph{positive operator}.

Let $D$ be a non-empty subset of a partially ordered set $(X,\leq)$. If $\{x_{\alpha}\}\subset X$ is an increasing net and $x=\sup x_{\alpha}$ exists, then we write $x_{\alpha}\uparrow x$. Analogously $x_{\alpha}\downarrow x$ means that the net $\{x_{\alpha}\}\subset X$  is decreasing and    $x= \inf x_\alpha$.

 Let $\M^+$ be the space of all positive operators in $\M$.  
 The \emph{trace} $\tau$ on $\M$ is a map $\tau:\M^+\rightarrow [0,\infty]$, which satisfies the following properties.\\
(i) $\tau(x+y)=\tau(x)+\tau(y)$ for all $x,y\in \M^+$.\\
(ii) $\tau(\lambda x)=\lambda\tau(x)$ for all $x\in \M^+$ and $\lambda\in \mathbb{R}^+$.\\
(iii) $\tau(u^*xu)=\tau(x)$ whenever $x\in \M^+$ and $u$ is a unitary operator.

Moreover, the trace $\tau:\M^+\rightarrow [0,\infty]$ is called\\
(i') \emph{faithful} if $x\in\M^+$ and $\tau(x)=0$ imply that $x=0$,\\
(ii') \emph{semi-finite} if for every $x\in \M^+$ with $\tau(x)>0$ there exists $0\leq y\leq a$ such that $0<\tau(y)<\infty$,\\
(iii') \emph{normal} if $\tau(x_{\beta})\uparrow \tau(x)$ whenever $x_{\beta}\uparrow x$ in $\M^+$.

Let further $\M$ be a \emph{semifinite von Neumann algebra} that is a von Neumann algebra equipped with a semi-finite, faithful  and normal trace $\tau$ \cite{Takesaki}.

 If $x \in \M$ then $\|x\|_{\M}$ will stand for the operator norm in $B(H)$. We will denote by $\one$ the identity in $\M$ and by $P(\M)$ the complete space of all orthogonal projections in $\M$. The symbol $U(\M)$ will stand for the collection of  all unitary operators in $\M$. 
The von Neumann algebra $\M$ is called \textit{non-atomic} if it has no minimal orthogonal projections,  while  $\M$ is said to be \textit{atomic} if all minimal projections have equal positive trace. A projection $p \in P(\mathcal{M})$ is called \emph{$\sigma$-finite} (with respect to the trace $\tau$) if there exists a sequence $\{p_n\}$ in $P(\mathcal{M})$ such that $p_n\uparrow p$ and $\tau(p_n)<\infty$ for all $n\in\mathbb{N}$. If  the unit element $\one$ in $\mathcal{M}$ is $\sigma$-finite, then we say that the trace  $\tau$ on $\mathcal{M}$ is \emph{$\sigma$-finite.}

 Given a normal operator $x$, $e^x(\cdot)$ will denote its spectral measure,  that is a projection valued measure $e^x(A) \in P(\M)$ for all Borel sets $A\subset \mathbb{C}$, and such that $x=\int_{\mathbb\Complex}\lambda de^{x}(\lambda)$. If $x$ is a normal operator with the spectral measure $e^x(\cdot)$ and $f$ is a complex valued Borel function on $\Complex$,  then $f(x)$ is defined by $f(x)=\int_{\mathbb R}f(\lambda) de^{x}(\lambda)$.  For instance applying this formula we can define  a power $x^c$ for any $c\in \Complex$ of an operator $x$. The theory of the mappings $f\to f (x)$
is called the Borel functional calculus of the operator $x$.  For the theory of spectral measures and functional calculus we refer to \cite{KR, Takesaki}.   Every closed and densely defined linear operator $x$ can be written in the form $x=\Rep{x}+i\Imp{x}$, where its \textit{real part} $\Rep{x}=(x+x^*)/2$ and  \textit{imaginary part} $\Imp{x}=(x-x^*)/(2i)$ are  both self-adjoint operators. Moreover, the \textit{positive part} $x^+$ and  the  \textit{negative part} $x^-$ of a self-adjoint operator $x$ are both defined by $x^+=\int_0^\infty \lambda de^x(\lambda)$ and $x^-=\int_{-\infty}^0 \lambda de^x(\lambda)$, with $x=x^+-x^-$. Hence every closed and densely defined linear operator can be written as a linear combination of four positive operators. The range and kernel of a closed and densely defined linear operator $x$ are denoted by $\Ran{x}$ and $\Ker{x}$, respectively. The projection onto $\Ker{x}$ is called the \emph{null projection} of $x$ and is denoted by $n(x)$. The projection $s(x)=\one-n(x)$, which is the projection onto $\Kerp{x}=\overline{\Ran{x}}$, is called the \emph{support projection} of $x$. If $u\in B(H)$ satisfies $u^*u=uu^*=\one$, then $u$ is called a \emph{unitary operator}. Moreover, an operator $v\in B(H)$ is a \emph{partial isometry} if the restriction of $v$ to the orthogonal complement of its kernel is an isometry, that is $\|v(\xi)\|_{H}=\|\xi\|_{H}$ for all $\xi \in \Kerp{v}$.

If $x$ is closed and densely defined then $x^*x$ is self-adjoint and we define $|x|=\sqrt{x^*x}$.
Let us point out that in the case of operators the triangle inequality for absolute value does not hold in general. The following simple example of operators  $x$ and  $y$ given by matrices
 \[
 x=\begin{bmatrix}
 -1&0\\
 0&0
 \end{bmatrix}\quad\text{and}\quad 
  y=\frac12\begin{bmatrix}
 1&1\\
 1&1
 \end{bmatrix},
 \]
shows  that $|x + y| \nleq |x| +|y|$ \cite{AL1985}.
The analogue of the triangle inequality for operators states that for any two operators $x,y\in B(H)$ there exist unitary operators $u,v\in B(H)$ such that $|x+y| \le u|x|u^* + v|y|v^*$ \cite[Theorem 2.2]{AAP1982}.

Given a non-empty subset $S$ of $B (H)$, the \emph{commutant} $S'$ of $S$ is defined by $S'=\{x\in B(H):\,xy=yx\quad\text{for all }y\in S\}$. We say that a closed and densely defined operator $x$ is \emph{affiliated} with the von Neumann algebra $\M$, denoted  by $x\eta \M$, whenever $ux=xu$ for all unitary operators in the commutant $\M'$ of $\M$. The collection of all operators affiliated with $\M$ will be denoted by $\M^{\text{affil}}$. Since every bounded operator can be written as a linear combination of unitary operators,  $x\in \M^{\text{affil}}$ if and only if for every $y\in\M'$ and $\xi\in D(x)$ we have that $y(\xi)\in D(x)$ and $yx(\xi)=xy(\xi)$. Moreover, if  $x=u\abs{x}$ is the polar decomposition of a closed and densely defined operator $x$, then $x$ is affiliated with $\mathcal{M}$ if and only if $u\in\mathcal{M}$ and $\abs{x}$ is affiliated with $\mathcal{M}$ \cite{Takesaki}. We have then that $s(x)=u^*u=e^{\abs{x}}(0,\tauone)\in\M$ and $n(x)=\one-s(x)=e^{\abs{x}}\{0\}\in\M$.  A closed, densely defined operator $x$, affiliated with a semi-finite von Neumann algebra $\mathcal{M}$, is called $\tau$-\textit{measurable} if there exists $\lambda>0$ such that $\tau\left(e^{\abs{x}}(\lambda,\infty)\right)<\infty$.  The collection of all $\tau$-measurable operators will be denoted by $S\Mtau$. The set $S\Mtau$ is a $*$-algebra with respect to the sum and product defined as the closure of the algebraic sum and product, respectively. For every subset $X\subset S\Mtau$ we will denote  further the set of all positive elements of $X$ by $X^+$.  For $\epsilon, \delta>0$,  we define a neighborhood $V(\epsilon, \delta)$ of zero by setting
\[
V(\epsilon,\delta)=\{x\in S\Mtau:\,\tau(e^{\abs{x}}(\epsilon,\infty))\leq \delta\}.
\]
The collection of sets $V(\epsilon,\delta)$ forms a neighborhood base at zero for the metrizable Hausdorff topology $\mathscr{T}_m$ on $S\Mtau$, called the \emph{measure topology } on $S\Mtau$.  Equipped with this topology, $S\Mtau$ is a complete topological $*$-algebra. If a sequence $\{x_n\}\subset S\Mtau$ converges to $x\in S\Mtau$ with respect to $\mathscr{T}_m$, we will say that $x_n$ converges to $x$ in measure, and denote by $x_n\xrightarrow{\tau}x$. For more details and proofs  we refer readers to \cite{Nelson, Takesaki}. 

 For an operator $x\in S\Mtau  $ the distribution function $d(x)=d(\cdot, x):[0,\infty)\to[0,\infty]$ is given by
\[
d(t,x)=\tau(e^{|x|}(t,\infty)),\quad t\geq 0.
\] 
By the definition of $\tau$-measurability, $d(t,x)$ is finite for some $t\geq 0$. Moreover, $d(x)$ is decreasing, right-continuous and $\lim_{t\to\infty}d(t,x)=0$.  Note that in this paper the terms decreasing or increasing will always mean non-increasing or non-decreasing, respectively.

Given $x\in S\Mtau$,  the function $\mu(x)=\mu(\cdot,x):[0,\infty)\to[0,\infty]$ defined by
\[
\mu(t,x)=\inf\left\{s\geq 0: d(s,x)\leq t\right\},  \quad t\geq 0,
\]
 is called a \emph{decreasing rearrangement} of $x$ or a \emph{generalized singular value function} of $x$. It follows that $\mu(x)$ is  a decreasing and right-continuous function on $[0,\infty)$.
We will use the notation $\mu(\infty, x)=\lim_{t\to\infty}\mu(t,x)$. $S_0\Mtau$ will stand for the set of measurable operators $x\in S\Mtau$ for which $\mu(\infty, x)=0$.
  Observe that if $\tauone<\infty$ then $\mu(t,x)=0$ for all $t\geq \tauone$, and so $\mu(\infty, x)=0$. Using the definition of $\mu(x)$ it is easy to see that $\mu(t,x)=0$ for all $t\geq \tau(e^{|x|}(0,\infty))=\tau(s(x))$. Since $d(x)$ is right continuous, we also have that $\mu(t,x)>0$ for all $0\le t<\tau(s(x))$. Hence $\tau(s(x))=m(\supp \mu(x))$.  If $x$ is bounded, then $\mu(0,x)=\|x\|_{\M}$, and if $x$ is unbounded then $\mu(0,x)=\infty$ \cite[Lemma 2.5 (i)]{Fack-Kos1986}.
 
The trace $\tau$ on $\mathcal{M}^{+}$ extends uniquely to the functional $\tilde{\tau}: S\Mtau^{+}\to[0,\infty]$ given by $\tilde{\tau}(x)=\int_0^{\infty}\mu(x)$, $x\in S\Mtau^+$ \cite[Proposition 3.9]{DDP4}.  This extension satisfies all conditions (i) - (iii)  stated in the definition of the trace as well as all properties (i') - (iii'). It will be also denoted by $\tau$. 

It is worth to note that the sets $V(\epsilon,\delta)$  take the form $V(\epsilon,\delta)=\{x\in S\Mtau:\, \mu(\delta,x)<\epsilon\}$. Hence $x_n\xrightarrow{\tau} x$ is equivalent to $\mu(\delta, x_n-x)\to 0$ for every $\delta > 0$ \cite[Lemma 3.1]{Fack-Kos1986}.
 
Below there is a list of some basic properties of the singular value function.
\begin{lemma} For $x,y \in S\Mtau$ the following is satisfied.
\label{lm:singfun}
\begin{itemize}
\item[(1)]  If $u,v\in \M$ then $\mu( uxv)\leq \|u\|_{\M}\|v\|_{\M}\mu(x)$.
\item[(2)] $\mu(\abs{x})=\mu(x)=\mu(x^*)$ and $\mu( \alpha x)=\abs{\alpha}\mu(x)$, $\alpha \in \Complex$. 
\item[(3)]  For $0\leq x\leq y$, $\mu(t,x)\leq \mu(t,y)$ for every $t\ge 0$.
\item[(4)]    $\mu(t_1+t_2,x+y)\leq \mu(t_1,x)+\mu(t_2,y)$, $t_1, t_2\ge 0$.
\item[(5)] $\mu(f(|x|)=f(\mu(x))$ for any continuous increasing function $f$ on $[0,\infty)$ with $f(0)\geq 0$.
\item[(6)] \cite[Proposition 1.1]{czer-kam2010} If $x\in S\Mtau $ and $\abs{x}\ge\mu(\infty,x)s(x)$ then $\mu(\abs{x}-\mu(\infty,x)s(x))=\mu(x)-\mu(\infty,x)$.
\item[(7)]\cite{noncomm} If $s\geq 0$ and $p=e^{\abs{x}}(s,\infty)$ then $\mu(\abs{x}p)=\mu(x)\chi_{[0,\tau(p))}$.
\item[(8)]\cite[Corollary 1.6]{czer-kam2015} Let $x\in S\Mtau$ and $p\in \mathcal{P}(\M)$. If $px=xp=0$ and $0\leq C\leq \mu(\infty,x)$ then $\mu(x+Cp)=\mu(x)$.
\end{itemize}
\end{lemma}

The proof of items (1)-(5) can be found in  \cite[Lemma 2.5]{Fack-Kos1986} or \cite{LSZ2013}. Property (7) follows by the fact that $\abs{x}p=f(|x|)$, where $f(t)=\chi_{(s,\infty)}(t)$, and so $d(\lambda, \abs{x}p)=\tau(e^{f(\abs x)}(\lambda,\infty))=\tau(f^{-1}(e^{\abs x}(\lambda,\infty))$ for every $\lambda\geq 0$.

Let $I=[0,\alpha)$,  $0<\alpha\leq\infty$  or $I=\mathbb{N}$. Let $L^0=L^0[0,\alpha)$ stand for the space of all complex-valued Lebesgue measurable functions
on $[0, \alpha)$ with identification a.e. with respect to the Lebesgue measure $m$.  
Given $f\in L^0$, the \emph{distribution function} $d(f)$ of $f$ is defined by $d(\lambda,f)=m\{t > 0:\,\abs{f (t)} > \lambda\}$, for all $\lambda\geq0$. The \emph{decreasing rearrangement} of $f$ is given by $\mu(t,f) = \inf\{s > 0:\,d(s,f)\leq t\}$,  $t \geq 0$. We set $\mu(\infty, f)=\lim_{t\to\infty}\mu(t,f)$. Observe that $d(f)=d(\cdot,f)$ and $\mu(f)=\mu(\cdot,f)$ are right-continuous, decreasing functions on $[0,\infty)$. In the case of the discrete measure,
$\ell^0 = L^0(\mathbb N)$ denotes the collection of all complex valued sequences. Then  for $f=\{f(n)\} =  \{f(n)\}_{n=1}^\infty \in \ell^0$  with $\lim_n f(n) =0$,  $\mu(t,f)$ is a  finite and countably valued function on $[0,\infty)$. In this case we will identify its decreasing rearrangement $\mu(f)$ with the sequence $\left\{\mu(n-1, f)\right\}_{n=1}^\infty$.

A \emph{support} of $f\in L^0(I)$, that is the set $\{t\in I:\,f(t)\neq 0\}$ will be denoted by $\supp f$.  Moreover for $f,g\in L^0(I)$, we say that $f$  is \emph{submajorized} by $g$, in the sense of Hardy, Littlewood and P\'olya, and we write $f\prec g$ if $\int_0^t \mu(f)\leq \int_0^t\mu(g)$ for all $t\ge 0$. Observe that if $I=\mathbb N$ then $f\prec g$ means that $\sum_{i=1}^n \mu(i-1,f)\leq \sum_{i=1}^n \mu(i-1,g)$ for every $n\in\mathbb{N}$.  For operators $x,y\in S\Mtau$, $x\prec y$  denotes $\mu(x)\prec \mu(y)$.  We have that $\mu(x+y)\prec \mu(x)+\mu(y)$ \cite[Theorem 4.3 (iii)]{Fack-Kos1986} and $\mu(xy)\prec \mu(x)\mu(y)$ \cite[Theorem 4.2 (iii)]{Fack-Kos1986}.

Any Banach space $F=F(I)\subset L^0(I)$, where either  $I= [0,\alpha)$, $0< \alpha \le \infty$, or $I=\mathbb{N}$, with the norm $\normf{\cdot}$ satisfying the condition that $f\in F$ and $\normf{f}\leq\normf{g}$ whenever $0\leq f\leq g$, $f\in L^0(I)$ and $g\in F$, is a \emph{Banach function, or sequence space}, respectively. An element $f\in F$ is called \emph{order continuous} if for every $0\leq f_n\leq\abs{f}$ such that $f_n\downarrow 0$ a.e. it holds $\|f_n\|_F\downarrow 0$. By $F_a$ we will denote the set of all order continuous elements of $F$. We say that   $F$ is \emph{order continuous} if $F=F_a$.
The space $F$ is said to have the \emph{Fatou property} if for any non-negative sequence $\{f_n\}\subset F$ with $\sup_n\|f_n\|_F<\infty$,  $f\in L^0$ and $f_n\uparrow f$ a.e. we have that $f\in F$ and $\|f_n\|_F\uparrow\|f\|_F$. The space $F^\times = F^\times (I)$ is called a  K\"{o}the dual of $F$ and is defined as
\[
F^\times = \left\{f\in L^0(I): \int_I f g < \infty \ \ \text{for all} \ g\in F\right\}.
\]
The space $F^\times$ equipped with the norm 
\[
\|g\|_{F^\times} = \sup\left\{\int_I fg: \ \|g\|_F \le 1\right\}, \ \ \ g\in F^\times,
\]
is a Banach (function or sequence) space satisfying the Fatou property. It is well known that $F=F^{\times\times}$ if and only if $F$ has the Fatou property \cite{BS, Z}.

 \begin{proposition}\label{prop:infinity} \cite[Theorem 14.9]{AB}

Let $F$ be a Banach (function or sequence) space. Then the following statements are equivalent.
\begin{itemize}
\item[(i)] $F$ is order continuous.
\item[(ii)] There is no subspace of $F$ isomorphic to $\ell_\infty$.
\item[(iii)] There is no subspace of $F$ order isomorphic to $\ell_\infty$.
\item[(iv)] $F$ is separable. 
\end{itemize}
\end{proposition}

The conditions (i) - (iii) are equivalent by \cite[Theorem 14.9]{AB}. Moreover  every separable Banach function or sequence space must be order continuous since otherwise it contains an isomorphic copy of $\ell_\infty$ which is not separable. Here $F$ is a subspace of $L^0(I)$ with its support contained in $I$, where $I$ is either $[0,\alpha)$, $0 <\alpha\le \infty$,  equipped with the Lebesgue measure or $I=\mathbb{N}$ with the counting measure. In both cases the measure is separable. Moreover $F$ contains simple functions on the supports contained in some sequence of sets $A_n\subset I$ with finite measure and such that $\cup_n A_n = \supp F$. Thus by Theorem 5.5 on p. 27 in \cite{BS},  (i) implies (iv).

A Banach function or sequence space $F$ is called a $KB$-space whenever it is order continuous and has the Fatou property \cite{AB, KA}. We have the following result. 

\begin{proposition}\label{prop:czero} \cite[Theorem 14.13]{AB}

Let $F$ be a Banach (function or sequence) space. Then the following statements are equivalent.
\begin{itemize}
\item[(i)] $F$ is not a $KB$-space that is $F$ is either not order continuous or $F$ does not posses Fatou property. 
\item[(ii)] $c_0$ is embeddable in $F$, that is $F$ contains a subspace isomorphic to $c_0$. 
\item[(iii)] $c_0$ is lattice embeddable in $F$, that is $F$ contains a subspace order isomorphic to $c_0$. 
\end{itemize}

\end{proposition}

A Banach function or sequence space $E\subset L^0$ is called \emph{a symmetric space} (also called \emph{rearrangement invariant space}) if it follows from $f\in L^0$, $g\in E$ and $\mu(f)\leq \mu(g)$ that $f\in E$ and $\norme{f}\leq\norme{g}$. Therefore $\|f\|_E = \|g\|_E$ whenever $f,g\in E$ and  $d(f) = d(g)$ \cite{BS, KPS}. If from $f,g\in E$ and $f\prec g$ we have that $\norme{f}\leq\norme{g}$ then $E$  is called \emph{strongly symmetric}.  Moreover, $E$ is called \emph{fully symmetric} if  for any $f\in L^0,\,g\in E$ and $f\prec g$ it follows that $f\in E$ and $\norme{f}\leq \norme{g}$. For any symmetric space $E$ we will use the notation $E_0=\{f\in E:\, \mu(\infty, f)=0\}$.
 Any symmetric space which is order continuous or satisfies the Fatou property is strongly symmetric \cite{BS,KPS}. 
For every symmetric space $E$ we have  \cite{BS},
\[
L_1(I)\cap L_{\infty}(I)\hookrightarrow E\hookrightarrow L_1(I)+L_{\infty}(I) \ \ \ \text{if} \ \ I=[0,\alpha),  \ \ \text{and} \ \ \ell_1\hookrightarrow E\hookrightarrow \ell_\infty \ \ \ \text{if} \ \ I = \mathbb{N}.
\]
 If $E$ is a symmetric a symmetric space  then  $E^\times$ is also a symmetric space  and
\[
\|g\|_{E^\times} = \sup\left\{\int_I \mu(f) \mu(g): \|g\|_E \le 1 \right\}, \ \ \   g\in E^\times.
\]
A symmetric space over  $I=[0,\alpha)$ will be called a \emph{symmetric function space}, and over $I=\mathbb{N}$, a \emph{symmetric sequence space}.   

Given a semifinite  von Neumann algebra $\M$ with a fixed semifinite, normal faithful trace $\tau$ and a symmetric Banach function space $E$ on $[0,\alpha)$, $\alpha=\tauone$, the corresponding \emph{noncommutative space of measurable operators} $E\Mtau$ is defined by setting
\[
E\Mtau=\{x\in S\Mtau:\quad \mu(x)\in E\}, 
\]
and it is equipped with the norm 
\[
\|x\|_{\nonsp}=\norme{\mu(x)}, \  \ \ \ x\in E\Mtau.
\]
For long period of time it was only known that $\nonsp$ is complete if $E$ is strongly symmetric. This has been proved in  papers \cite{S1988, DDPMark, DDPnoncomm, SC1990}.  In 2008, N. Kalton and F. Sukochev  \cite{KS}   solved this problem in full generality showing that $E\Mtau$ is a Banach space, without requiring any additional assumptions on a symmetric Banach space  $E$. A nice exposition of their non-trivial proof can also be found in \cite[Theorem 3.5.5]{LSZ2013}. It is worth to observe that Kalton-Sukochev’s proof holds for any quasi-Banach symmetric space which is in addition $p$-convex for some $0 < p < \infty$ and that this restriction was shown to be redundant in \cite{S}.

If $E=L_p$, $1\leq p\leq \infty$, then $E\Mtau=L_p\Mtau$ with the norm $\|x\|_{L_p\Mtau}=\|\mu(x)\|_{L_p}$, is called a \emph{noncommutative $L_p$ space}. 
As shown in \cite{DDP4}, the restriction of $\tau$ from $S\Mtau^+$ to $L_1\Mtau^+$ is an additive positively homogeneous real
valued functional, for which  $\tau(x)=\int_0^\infty \mu(x)$ for all $x\in L_1\Mtau^+$.  This functional extends uniquely to a linear functional $\dot{\tau}:L_1\Mtau\to \Complex$, denoted again by $\tau$.

An element $x\in\nonsp$ is called \emph{order continuous} if for every sequence $0\leq x_n\leq |x|$ with $x_n\downarrow 0=\inf{x_n}$ it follows that $\|x_n\|_{\nonsp}\downarrow 0$.  The set of all order continuous elements in $E\Mtau$ is denoted by $(E\Mtau)_a$. If $E\Mtau= (E\Mtau)_a$ then the space $E\Mtau$ is called order continuous. It is known that if $E$ is order continuous and strongly symmetric, then so is $E\Mtau $ \cite[Proposition 2.3]{CSweak}. On the other hand if $E\Mtau$ is order continuous and $\M$ is non-atomic then $E$ must be order continuous by order  isometric embedding of $E$ into $E\Mtau$ (see Corollary \ref{cor:isomglobal}).   Moreover, if $E$ is a symmetric space on $[0,\alpha)$, which is order continuous, then it is fully symmetric \cite[Chapter II, Theorem 4.10]{KPS}, and therefore $E\Mtau $ is fully symmetric.

Let $\M$ be a semifinite von Neumann algebra acting on  a separable Hilbert space $H$. If $E$ is  separable  then $E$ is order continuous by Proposition \ref{prop:infinity}. If in addition $E$ is strongly symmetric then $E\Mtau$ is order continuous  \cite[Proposition 2.3]{CSweak}.  Thus by Corollary 6.10 in \cite{DDP2011}, if $H$ is separable and $E$ is separable strongly symmetric then $E\Mtau$ is separable (see also  \cite[Proposition 1, Theorem 2]{Medz}.  On the other hand, by isometric embedding of $E$  into $E\Mtau$ in the case of non-atomic $\M$ (see Corollary \ref{cor:isomglobal}),  if $\nonsp$ is separable, then $E$ is separable. 
Separability of $L_p\Mtau$ spaces was considered in \cite{S1}. 

If $E$ is order continuous then  the dual $E\Mtau^{*}$ can be identified with the K\"{o}the dual  $E\Mtau^{\times}$ \cite{DDP4}, where 
\[
\nonsp^{\times}=\{x\in S(\mathcal{M},\tau): xy\in L_{1}\Mtau \text{ for all }y\in \nonsp\},
\]
 and it is equipped with the norm
\[
 \|x\|_{\nonsp^{\times}}=\sup\{\tau(|xy|):y\in \nonsp\text{,}\, \normcomm{y}\leqslant 1 \}\text{,}\quad x\in \nonsp^{\times}.
\]
Therefore if $E$ is order continuous then  every functional $\Phi\in E\Mtau^{*}$ is of the form $\Phi(x)=\tau(xy)$, $x\in \nonsp$, for some $y\in E\Mtau^{\times}$ and $\|\Phi\| = \|y\|_{E^\times\Mtau}$. Observe that $\tau(xy)$ is well defined since $xy\in L_1\Mtau$.

If $E$ a strongly symmetric Banach function space on $[0,\tauone)$ then $E\Mtau ^{\times}=E^{\times}\Mtau$ and $E^{\times}$ is also a fully symmetric Banach function space  \cite[Propositions 5.4, 5.6]{DDP4}. Therefore if $E$ is an order continuous symmetric function space, and hence it is a fully symmetric function space, then  $E\Mtau^{*}$ is identified with a fully symmetric K\"{o}the dual $E^{\times}\Mtau$. In particular, $L_1\Mtau^{\times}=L_\infty\Mtau=\M$. We wish to note that we also have $\M^{\times}=L_\infty\Mtau^{\times}=L_1\Mtau$ \cite[Proposition 5.2 (viii)]{DDP4}.

For the theory of operator algebras we refer to
\cite{KR,Takesaki}, and for noncommutative Banach  spaces of measurable operators
to \cite{DDPnoncomm,LSZ2013, noncomm,P}.

\section{Examples of symmetric spaces of measurable operators}
We  discuss below how $E\Mtau$ can be identified with many known spaces, like noncommutative $L_p$ spaces, unitary matrix spaces including Schatten classes,  or  symmetric function spaces.

\subsection{Noncommutative $L_p$ spaces}
If $E=L_p[0,\tauone)$, $1\leq p<\infty$, then for $x\in L_p\Mtau$ we have 
\[
\|x\|_{L_p\Mtau}=\|\mu(x)\|_{L_p}=\left(\int_0^{\tauone}\mu(\abs{x}^p) \right)^{1/p}=\left(\tau(\abs{x}^p)\right)^{1/p}.
\]
We have that  $x\in L_\infty\Mtau$ if and only if $x\in S\Mtau$ and $\mu(x)\in L_\infty[0,\tauone)$, which is equivalent with $x\in \M$. Moreover by \cite[Lemma 2.5 (i)]{Fack-Kos1986},  
\[
\|x\|_{L_\infty\Mtau}=\|\mu(x)\|_{L_\infty}=\sup_{t\in[0,\tauone)}\mu(t,x)=\mu(0,x)=\|x\|_{\M}.
\]
  Hence $L_\infty\Mtau=\M$ with equality of norms.
 The spaces 
 \[
 L_1\Mtau+\M=\left\{x\in S\Mtau: \int_0^1\mu(x)<\infty\right\},
 \]
 \[
 L_1\Mtau\cap \M=\{x\in S\Mtau:\, \mu(x)\in L_1[0,\tauone)\cap L_\infty[0,\tauone)\}
 \]
are equipped with the norms 
\[
\|x\|_{ L_1\Mtau+\M} = \int_0^1\mu(x),  \  \ \ \|x\|_{L_1\Mtau\cap \M }=\max\{\|x\|_{L_1\Mtau},\|x\|_{\M}\},
\]
 respectively. 
If $\M$ is non-atomic we have that 
\[
L_1\Mtau\cap \M\hookrightarrow E\Mtau\hookrightarrow L_1\Mtau+\M
\]
 with the continuous embeddings \cite[Example 2.6.7]{LSZ2013}.

\subsection{Unitary matrix spaces and Schatten classes}
\label{sec:unitary}

Recall that given a maximal orthonormal system $\{e_{\alpha}\}$ in the Hilbert space $H$ the \textit{canonical trace} $\tr: B(H)^+\to [0,\infty]$ is defined by
\[
\tr(x)=\sum_{\alpha}\langle xe_{\alpha}, e_{\alpha}\rangle,\quad x\in B(H)^+.
\]
The value of $\tr(x)$ does not depend on the choice of the maximal orthonormal system in $H$. The canonical trace  $\tr$ is semi-finite, faithful and  normal. 

Given a symmetric sequence space $E\neq\ell_{\infty}$, the \emph{unitary matrix space}  $C_E$ is a subspace of a Banach space of compact operators $K(H)\subset B(H)$ for which the sequence of
 singular numbers $S(x)=\left\{s_n(x)\right\}\in E$, and it is equipped with the norm $\|x\|_{C_E}=\|S(x)\|_E$.
  Note that if $E$ is a symmetric sequence space, then  $E\neq \ell_\infty$ is equivalent with  $E\subset c_0$.
  
If $H$ is separable and $E$ is a separable sequence space then $C_E$ is separable \cite[Proposition 1, Theorem 2]{Medz}. Moreover, if $E$ is order continuous then $C_E$ is order continuous  \cite[Corollary 6.1]{CSweak}. On the other hand if $C_E$ is separable (respectively, order continuous) then the separability (respectively, order continuity) of $E$ follows by the order isometric embedding of $E$ into $C_E$ (see Corollary \ref{cor:isomglobal}).

If a symmetric sequence space $E\neq \ell_1$ then $E^\times\subset c_0$ and $C_{E^\times}$ is well defined. If $E\neq \ell_1$ is separable then $(C_E)^*$ is isometrically isomorphic to  $(C_E)^\times$ and  $(C_E)^\times=C_{E^\times}$. In this case the functionals $\Phi\in (C_E)^*$ are of the form
\[
\Phi(x)=\tr(xy),\quad x\in C_{E},\, y\in C_{E^\times},
\]
and $\|\Phi\|_{(C_E)^*}=\|y\|_{(C_E)^\times}$ \cite[Theorem 12.2]{GK}.

The unitary matrix space $C_E$ can be  identified with a symmetric  space of measurable operators $G\Mtau$ for some symmetric function space $G$ on $[0,\infty)$, and $\M = B(H)$ with canonical trace $\tr$. Using this identification, many lifting-type results from the symmetric sequence space $E$ into the space $C_E$ can be deduced from the corresponding results for the symmetric function space $E$ and the  space $\nonsp$.

  Indeed  let $G$ be the set of all real functions $f\in L_1(0,\infty)+L_\infty(0,\infty)$  such that
\[
\pi(f)=\{\pi_n(f)\}=\left\{\int_{n-1}^n\mu(f) \right\}\in E,
\]
and set $\|f\|_G=\|\pi(f)\|_E$. 
As shown in \cite[Theorem 3.6.6.]{LSZ2013}, $G$ equipped with this norm is a symmetric function space on $[0,\infty)$. 
Moreover, if $E$ is fully symmetric or order continuous then so is
$G$ \cite[Proposition 6.1]{CSweak}. It is well known
that $S\left(B(H),\tr\right)=B(H)$, where $\tr$ is the canonical trace
on $B(H)$, and the convergence $x_n\xrightarrow{\tr} x$ is equivalent
to the norm convergence $\|x-x_n\|_{B(H)}\to0$, for $x,x_n\in B(H)$ \cite[Example 2.3.2.]{LSZ2013}.
Since $E\neq\ell_\infty$, the symmetric
space of measurable operators $G\left(B(H),\tr\right)$ is a proper two-sided $*$-ideal in
$B(H)$ and therefore it is contained in $K(H)$ \cite{GK}. Thus for any $x\in
G\left(B(H),\tr\right)$ the singular value function $\mu(x)$ is of
the form $\mu(t,x)=\sum_{n=1}^\infty s_n(x)\chi_{[n-1,n)}(t)$,
$t\ge 0$, where $s_n(x) \to 0$. Therefore the spaces $C_E$  and $G\left(B(H),\tr\right)$
coincide as sets and they have identical norms $\|x\|_{C_E}=\|S(x)\|_E=\|\pi(\mu(x))\|_E=\|\mu(x)\|_G=\|x\|_{G(B(H),\tr)}$.

In particular  when $E=\ell_p$, $1\le p<\infty$, we have that $G=L_p(0,\infty)$ and $L_p\left(B(H),\tr\right)=C_p$, where $C_p$ is the space of $p$-Schatten class of operators. 
We have that $C_1\hookrightarrow C_E\hookrightarrow K(H)$ \cite[Example 2.6.7 c]{LSZ2013} with the continuous embeddings.
\subsection{Symmetric function spaces}
\label{sec:symmfun}
For the reader's convenience we include in this part the detailed explanation how the noncommutative symmetric spaces can be identified with their commutative counterparts. 
Thanks to this representation many of the results for noncommutative spaces can be interpreted for symmetric function spaces, especially in the context of relating properties of functions and their decreasing rearrangements.

Let $0< \alpha \le \infty$.
Consider the commutative von Neumann algebra 
\[\mathcal{N}=\{N_f: L_2[0,\alpha)\to
L_2[0,\alpha):\quad f\in L_\infty[0,\alpha)\},
\]
 where $N_f$ acts via pointwise multiplication on $L_2[0,\alpha)$ and the trace $\eta$ is given by integration, that is  
 \[
 N_f(g)=f\cdot g,\quad  g\in L_2[0,\alpha),\quad\text{ and  }\quad \eta(N_f)=\int_0^{\alpha} f.
 \] 
 It is straightforward to check that the map $f\mapsto N_f$ is a $*$-isomorphism from $L_\infty[0,\alpha)$ into $B(L_2[0,\alpha))$, which is also an isometry since  $\|f\|_{L_\infty}=\|N_f\|_{B(L_2[0,\alpha))}$. Therefore the von Neumann algebra $\mathcal{N}$ is commonly identified with $L_{\infty}[0,\alpha)$. 

If $N_f$ is a projection in $\mathcal{N}$ then $fg=N_f(g)=N_f(N_f(g))=f^2g$ for all $g\in L_2[0,\alpha)$. Hence for any $t\in[0,\alpha)$, $f(t)=0$ or $f(t)=1$. Consequently, the projections in $\mathcal{N}$ are given by  
\[
P(\mathcal{N})=\{N_{\chi_A}:\,A \text{ is a measurable subset of } [0,\alpha)\}.
\]
Furthermore, if $N_f$ is a unitary operator in $\mathcal{N}$ then  $N_{\chi_{[0,\alpha)}}=N_f(N_f)^*=N_f N_{\overline{f}}=N_{f\overline{f}}=N_{\abs{f}^2}$
and the unitary operators in $\mathcal{N}$ are given by
\[
U(\mathcal{N})=\{N_{f}:\,f\in L_\infty[0,\alpha),\abs{f}=\chi_{[0,\alpha)}\}.
\]

\textbf{Fact 1.} $\mathcal{N}'=\mathcal{N}$

\begin{proof}
 Clearly $\mathcal{N}\subset \mathcal{N}'$, since $\mathcal{N}$ is commutative. Let $F\in \mathcal{N}'$ that is $F$ is a bounded operator on $L_2[0,\alpha)$ and  
\begin{equation}
F(\xi\cdot g)=F(N_\xi(g))=N_\xi(F(g))=\xi\cdot F(g)\text{ for every }\xi\in L_\infty[0,\alpha), g\in L_2[0,\alpha).
\label{eq:rm:commutative1}
\end{equation} 
Hence for any measurable set $A\subset [0,\alpha)$ with $m(A)<\infty$, we have that $F(\chi_A)=F(\chi_A)\chi_A$.
In particular, $F(\chi_{[i-1,i)})=F(\chi_{[i-1,i)})\chi_{[i-1,i)}$ for every $i\in\mathbb{N}$, and so $\{F(\chi_{[i-1,i)})\}$ is a sequence of functions with disjoint supports included in $[i-1,i)$. We claim that 
\[
\sup_{i\in\mathbb{N}}\, \text{esssup}_{t\in[i-1,i)}|F(\chi_{[i-1,i)})|(t)<\infty.
\] 
In fact supposing the above is not satisfied, that is for every $n\in\mathbb{N}$ there exist $i_n\in\mathbb{N}$, a set $A_{i_n}\subset [i_n-1,i_n)$  with $m(A_{i_n}) > 0$, and such that $\abs{F(\chi_{[i_n-1,i_n)})}(t)\geq n$ for all $t\in A_{i_n}$. Taking $g_{i_n}=\frac{1}{m(A_{i_n})^{1/2}}\chi_{A_{i_n}}$ we have $\|g_{i_n}\|_{L_2}=1$, and for every $n\in\mathbb{N}$,
\begin{align*}
\|F(g_{i_n})\|_{L_2}^2&=\int_0^{\alpha} \frac{1}{m(A_{i_n})}\abs{F(\chi_{A_{i_n}})}^2=\int_0^{\alpha} \frac{1}{m(A_{i_n})}\abs{F(\chi_{A_{i_n}}\chi_{[i_n-1,i_n)})}^2\\
&=\int_0^{\alpha} \frac{1}{m(A_{i_n})}\abs{F(\chi_{[i_n-1,i_n)})}^2\chi_{A_{i_n}}\geq n^2,
\end{align*}
contradicting the fact that $F$ is a bounded operator on $L_2[0,\alpha)$.

 Hence $\sup_{i\in\mathbb{N}}\, \text{esssup}_{t\in[i-1,i)}|F(\chi_{[i-1,i)})|(t)<\infty$ and $\sum_{i=1}^\infty F(\chi_{[i-1,i)})\in L_\infty[0,\alpha)$. 

By (\ref{eq:rm:commutative1}), $F(g\chi_{[i-1,i)})=F(\chi_{[i-1,i)})g$ for every simple function $g\in L_2[0,\alpha)$. For an arbitrary $g\in L_2[0,\alpha)$, we can take a sequence of simple functions $\{g_n\}\subset L_2[0,\alpha)$ with $g_n\to g$ in $L_2[0,\alpha)$. Since $F$ is a bounded operator on $L_2[0,\alpha)$  we have that $F(g_n)\to F(g)$. Moreover, for any $i\in\mathbb N$, $g_n\chi_{[i-1,i)}\to g\chi_{[i-1,i)}$ as $n\to \infty$ in $L_2[0,\alpha)$ and $F(\chi_{[i-1,i)})\in L_\infty[0,\alpha)$. Hence
\[
F(g_n\chi_{[i-1,i)})\to F(g\chi_{[i-1,i)})\text{ and }F(g_n\chi_{[i-1,i)})=F(\chi_{[i-1,i)})g_n\to F(\chi_{[i-1,i)})g
\]
in $L_2[0,\alpha)$ for each $i\in \mathbb{N}$ as $n\to \infty$.
Thus $F(g\chi_{[i-1,i)})=F(\chi_{[i-1,i)})g$ for all $g \in L_2[0,\alpha)$.

Take next $h\in L_2[0,\alpha)$ and set $h_n=h\chi_{[0,n)}$. Then $h_n\to h$ in $L_2[0,\alpha)$ and
\begin{align*}
F(h)&=\lim_n F(h_n)=\lim_n F\left(\sum_{i=1}^nh\chi_{[i-1,i)}\right)=\lim_n\sum_{i=1}^n F(h\chi_{[i-1,i)})\\&=\lim_n\sum_{i=1}^n F(\chi_{[i-1,i)})h=\left(\sum_{i=1}^\infty F(\chi_{[i-1,i)})\right)h.
\end{align*}
Since it was shown earlier that $\sum_{i=1}^\infty F(\chi_{[i-1,i)})\in L_\infty[0,\alpha)$, we have that $F=N_{\sum_{i=1}^\infty F(\chi_{[i-1,i)})}$ and $F\in \mathcal{N}$.
\end{proof}

In the next fact we extend the operator $N_f$ from $f\in L_\infty[0,\alpha)$ to $f\in L^0[0,\alpha)$.

\textbf{Fact 2.}
Given $f\in L^0[0,\alpha)$ define the operator $N_f$ by setting 
\[
D(N_f)=\{\xi \in L_2[0,\alpha):\, f\xi\in   L_2[0,\alpha)\}
\]
and for $\xi\in D(N_f)$,
\[
N_f\xi=f\xi.
\]
The operator $N_f$ is closed and densely defined.
\begin{proof}
Observe first that the operator $N_f$ is well defined. Let $N_{f_1}=N_{f_2}$ for $f_1,f_2\in L^0[0,\alpha)$.
Setting $A_{in} = \{t\in [0,\alpha): 1/n \le |f_{i}(t)| \le n\}\cap [0,n]$, $i=1,2, n\in \mathbb{N}$, we get  $\cup_n (A_{1n}\cap A_{2n}) = [0,\alpha)$. Hence $f_i\chi_{A_{1n}\cap A_{2n}}\in L_2[0,\alpha)$ for $i=1,2$,
and $f_1\chi_{A_{1n}\cap A_{2n}}=f_2\chi_{A_{1n}\cap A_{2n}}$ 
for all $n\in\mathbb{N}$. Thus $f_1=f_2$ a.e..

Let $\xi\in L_2[0,\alpha)$, $f\in L^0[0,\alpha)$, and consider the sequence of measurable sets $A_n=\{t\in[0,\alpha):\, \abs{f(t)}\leq n\}\cup\left[\frac1n,\infty\right)$ for $n\in\mathbb{N}$. We will show that $\xi\chi_{A_n}\in D(N_f)$ and $\xi\chi_{A_n} \to \xi$ in $L_2[0,\alpha)$, which establishes that $N_f$ is densely defined. Indeed, we have 
\begin{align*}
\|\xi-\xi\chi_{A_n}\|_{L_2}&\leq \|\xi\|_{L_2}\|\|\chi_{[0,\alpha)}-\chi_{A_n}\|_{L_2}=\|\xi\|_{L_2}m(A_n^c)\\
&\leq \frac1n\|\xi\|_{L_2}\to 0\text{ as }n\to\infty.
\end{align*}
Moreover,
\[
\|f\xi\chi_{A_n}\|_{L_2}^2=\int_0^{\alpha} \abs{f\xi\chi_{A_n}}^2=\int_{A_n}\abs{f}^2\abs{\xi}^2\leq  n^2\|\xi\|_{L_2}^2.
\]

It is not difficult to see that $N_f$ is also closed. Indeed let $\xi_n\rightarrow\xi$ in $L_2[0,\alpha)$, where $\{\xi_n\}\subset D(N_f)$, and $N_f\xi_n=f\cdot \xi_n\rightarrow \beta$ in $L_2[0,\alpha)$. Then  there is a subsequence $\{\xi_{n_k}\}$ of $\xi_n$, such that  $\xi_{n_k} \to \xi$ and $f\cdot \xi_{n_k}\to \beta$   a.e. on $[0,\alpha)$.  We have then that $f\cdot \xi_{n_k}\to f\cdot \xi$ a.e. and so $\beta=f\cdot \xi$. Consequently, $N_f$ is closed.
\end{proof}

\textbf{Fact 3.}\label{fact3} $\mathcal{N}^{\text{affil}}=\{ N_f: \,\,f\in L^0[0,\alpha) \}$.
\begin{proof}
Observe first that $N_f$, $f\in L^0[0,\alpha)$, is affiliated with $\mathcal{N}$. Indeed, let $N_g\in U(\mathcal{N'})=U(\mathcal{N})$, where  $|g|=\chi_{[0,\alpha)}$. For $\xi\in D(N_f)$  we have that $N_f N_g(\xi)=fg\xi \in L_2[0,\alpha)$, and so $N_g(\xi)\in D(N_f)$. Since pointwise multiplication is a commutative operation, we get $N_f N_g(\xi)=fg\xi=gf\xi=N_gN_f(\xi)$.

It remains to show that every closed and densely defined operator $x$ on $L_2[0,\alpha)$ which is affiliated with $\mathcal{N}$, is of the form $N_f$ for $f\in L^0[0,\alpha)$.  Let $x=u\abs{x}$ be the polar decomposition of $x$.  Recall that  $x$ is affiliated with $\mathcal{N}$ if and only if $\abs{x}$ is affiliated with $\mathcal{N}$ and $u\in\mathcal{N}$. Moreover, $\abs{x}\in \mathcal{N}^{\text{affil}}$ if and only if $e^{\abs{x}}(B)\in\mathcal{N}$ for every Borel set $B$ in $[0,\alpha)$. Set $p_n=e^{\abs{x}}[n-1,n)$ and  $x_n=\abs{x}p_n$, $n\in\mathbb N$. Then $x_n$ is bounded and affiliated with $\mathcal{N}$, and therefore $x_n\in \mathcal{N''}=\mathcal{N}$. Hence there are  sequences of measurable sets $A_n$ and non-negative functions $g_n\in L_{\infty}[0,\alpha)$, such that $p_n=N_{\chi_{A_n}}$ and $x_n=N_{g_n}$. Since $\{p_n\}$ is a sequence of mutually orthogonal projections, $\{A_n\}$ is a sequence of pairwise disjoint sets. Furthermore, $g_n \xi=N_{g_n}(\xi)=x_n(\xi)=x_n p_n(\xi)=N_{g_n}N_{\chi_{A_n}}(\xi) =g_n\chi_{A_n}\xi$, for every $\xi \in L_2[0,\alpha)$. In particular the equality holds for every $\xi=\chi_F$, where $F$ is a set of finite measure. Hence  $\supp g_n \subset A_n$. Finally, by $\one=N_{\chi_{[0,\alpha)}}=\sum_{n=1}^\infty p_n=\sum_{n=1}^\infty N_{\chi_{A_n}}$, it follows that  $\cup_{n=1}^\infty A_n=[0,\alpha)$. Consider now $g=\sum_{n=1}^\infty g_n$ with the sum taken pointwise.   Let $\xi\in D(x)=D(\abs{x})$, that is $\xi\in L_2[0,\alpha)$ and $|x|(\xi)\in L_2[0,\alpha)$. Then $|x|(\xi)=\left(\sum_{n=1}^\infty x_n\right)(\xi)$ converges pointwise. By the dominated convergence theorem, $\left(\sum_{n=1}^N g_n\right) \xi=\sum_{n=1}^N g_n \xi=\sum_{n=1}^N x_n (\xi)=\left(\sum_{n=1}^N x_n\right) (\xi)$ converges to $|x|(\xi)$ in $L_2[0,\alpha)$. Consequently,
\[
N_g(\xi)=g\xi=\left(\sum_{n=1}^\infty g_n\right)\xi=\left(\sum_{n=1}^\infty x_n\right)(\xi)=\abs{x}(\xi),
\]
and $N_g(\xi)\in L_2[0,\alpha)$.
 Therefore $\abs{x}\subset N_g$, that is $D(\abs{x})\subset D(N_g)$ and for all $\xi \in D(\abs{x})$, $|x|(\xi)=N_g(\xi)$. 
 
For the converse, suppose that $\xi\in D(N_g)$, that is $\xi\in L_2[0,\alpha)$ and $ g\xi = \left(\sum_{n=1}^\infty g_n\right)\xi=\sum_{n=1}^\infty g_n\xi\in L_2[0,\alpha)$. Again by the dominated convergence theorem we have  $\sum_{n=1}^N g_n\xi\rightarrow g\xi$ in $L_2[0,\alpha)$, and $\sum_{n=1}^\infty g_n\xi$ is a norm convergent series in $L_2[0,\alpha)$. Hence $\abs{x}(\xi)=\sum_{n=1}^\infty x_n(\xi)\in L_2[0,\alpha)$ and $\xi\in D(\abs{x})$. Since also $N_g(\xi)=|x|(\xi)$ we have that $N_g\subset \abs{x}$ and consequently $N_g=\abs{x}$. 

Finally, since $x$ is affiliated with $\mathcal{N}$, $u\in\mathcal{N}$ and $u=N_h$ for some $h\in L_\infty[0,\alpha)$. Setting $f=gh$ we have that $x=u\abs{x}=N_hN_g=N_{gh}=N_f$, and $x$ is of the desired form. 
\end{proof}

\textbf{Fact 4.} The algebra of all $\eta$-measurable operators on $\mathcal{N}$ is of the form  
\[
S(\mathcal{N},\eta)=\{N_f: \, f\in L^0[0,\alpha)\text{ and } \exists A, m(A^c)<\infty, f\chi_A\in L_{\infty}[0,\alpha)\},
\]
 and is identified with 
 \[
 S([0,\alpha),m)=\{f\in L^0[0,\alpha):\, \exists A, m(A^c)<\infty, f\chi_A\in L_{\infty}[0,\alpha)\}.
 \]
\begin{proof}
It is naturally to expect that $N_f\geq 0$ if and only if $f\geq 0$ a.e.. Indeed, $N_f\geq 0$ is equivalent to $\langle N_f\xi,\xi\rangle=\langle f\xi,\xi\rangle=\int_0^{\alpha} f\abs{\xi}^2\geq 0$ for every $\xi \in L_2[0,\alpha)$. So for any $A\subset [0,\alpha)$ with finite measure, taking $\xi=\chi_A$, we get $\int_A f\geq 0$, which  is equivalent to $f\geq 0$ a.e..

 Let $x = N_f\in \mathcal{N}^{\text{affil}}$. Then by Fact 3 above, $f\in L^0[0,\alpha)$ and  $\abs{x}=N_{|f|}$.  Given $s>0$ we have that $e^{\abs{x}}(s,\infty)=N_{\chi_B}$ for some measurable set $B$, and $e^{\abs{x}}[0,s]=\one -e^{\abs{x}}(s,\infty)=N_{\chi_{[0,\alpha)}}-N_{\chi_B}=N_{\chi_{B^c}}$. Moreover,
 \[
 N_{|f|\chi_B}=N_{|f|}N_{\chi_B}=|x| e^{\abs{x}}(s,\infty)=\int_{(s,\infty)}\lambda de^{|x|}(\lambda)\geq se^{\abs{x}}(s,\infty)=N_{s\chi_B},
 \]
  and 
 \[
 N_{|f|\chi_{B^c}}=N_{|f|}N_{\chi_{B^c}}=|x| e^{\abs{x}}[0,s]=\int_{[0,s]}\lambda de^{|x|}(\lambda)\leq se^{\abs{x}}[0,s]=N_{s\chi_{B^c}}.
 \]
  Hence $|f|\chi_B\geq s\chi_B$ and $|f|\chi_{B^c}\leq s\chi_{B^c}$. 
  
  We will claim next that $B=\{t\in[0,\alpha):\, \abs{f(t)}>s\}$.
Suppose first that $e^{\abs{x}}(s,\infty)=N_{\chi_B}=0$, equivalently $B=\emptyset$. Then
\[
|f|=|f|\chi_{B^c}\leq s\chi_{B^c}=s\chi_{[0,\alpha)},
\]
and so $B=\emptyset=\{t\in [0,\alpha): |f(t)|>s\}$.
Assume now that $e^{\abs{x}}(s,\infty)\neq 0$.  Then for all $\xi\in L_2[0,\alpha)$ either $e^{|x|}(s,\infty)(\xi)=0$ or  $|x|e^{\abs{x}}(s,\infty)(\xi)\neq se^{\abs{x}}(s,\infty)(\xi)$. Indeed, suppose to the contrary that there exists $\xi\in L_2[0,\alpha)$ such that $e^{|x|}(s,\infty)(\xi)\neq 0$ and $|x|e^{\abs{x}}(s,\infty)(\xi)=se^{\abs{x}}(s,\infty)(\xi)$. Let  $\lambda>s$. Then in view of $|x|e^{|x|}(\lambda,\infty)\geq \lambda e^{|x|}(\lambda,\infty)$ we have
\begin{align*}
\lambda \langle e^{\abs{x}}(\lambda,\infty)(\xi),\xi\rangle
&\leq\langle |x|e^{\abs{x}}(\lambda,\infty)(\xi),\xi\rangle=\langle e^{\abs{x}}(\lambda,\infty)|x|e^{\abs{x}}(s,\infty)(\xi),\xi\rangle\\
&=\langle e^{\abs{x}}(\lambda,\infty)se^{\abs{x}}(s,\infty)(\xi),\xi\rangle=s\langle e^{\abs{x}}(\lambda,\infty)(\xi),\xi\rangle.
\end{align*}
 Since $\lambda>s$, it follows that $\langle e^{\abs{x}}(\lambda,\infty)(\xi),\xi\rangle=0$ for all $\lambda >s$.
 By $\langle e^{\abs{x}}(\lambda,\infty)(\xi),\xi\rangle\uparrow \langle e^{\abs{x}}(s,\infty)(\xi),\xi\rangle$ as $\lambda\downarrow s$, we have that 
\[
\|e^{\abs{x}}(s,\infty)(\xi)\|_{L_2}^2=  \langle e^{\abs{x}}(s,\infty)(\xi),e^{\abs{x}}(s,\infty)(\xi)\rangle=\langle e^{\abs{x}}(s,\infty)(\xi),\xi\rangle=0,
 \]
 which leads to a contradiction. 

Hence if $e^{|x|}(s,\infty)(\xi)\neq 0$,  $\xi\in L_2[0,\alpha)$, then  $|x|e^{\abs{x}}(s,\infty)(\xi)\neq se^{\abs{x}}(s,\infty)(\xi)$. Let $A\subset B$ with  $0<m(A)<\infty$, and choose $\xi =\chi_{A}$. Then $\xi\in L_2[0,\alpha)$ and $e^{|x|}(s,\infty)(\xi)=N_{\chi_B}(\chi_{A})=\chi_B\chi_A=\chi_A\neq 0$ a.e.. Hence 
\[
f\chi_A=N_f(\chi_A)=|x|e^{|x|}(s,\infty)(\xi)\neq s e^{|x|}(s,\infty)(\xi)=s\chi_A.
\]
Since $A$ was an arbitrary  subset of $B$ with $0<m(A)<\infty$, we have that $f(t)\neq s$ for all $t\in B$.  Consequently, $|f|\chi_B>s\chi_B$ and $|f|\chi_{B^c}\leq s\chi_{B^c}$. Hence also in this  case, $B=\{t\in[0,\alpha):\, \abs{f(t)}>s\}$.

Suppose next that $x\in S(\mathcal{N},\eta)$, that is $x\in \mathcal{N}^{\text{affil}}$ and $\eta(e^{\abs{x}}(\lambda,\infty))<\infty$ for $\lambda$ large enough.  Hence $x=N_f$ for some $f\in L^0[0,\alpha)$ and $m\{t:\abs{f(t)}>\lambda\}<\infty$.   Equivalently $x=N_f\in S(\mathcal{N},\eta)$ if and only if there exists a measurable set $A$, with $m([0,\alpha)\setminus A)<\infty$ and $f\chi_A\in L_{\infty}[0,\alpha)$.   Thus $S(\mathcal{N},\eta)$ can be identifies with the set 
\begin{align*}
S([0,\alpha),m)&=\{f\in L^0[0,\alpha):\, d(f,s)<\infty,\text{ for some } s\geq 0\}\\
&=\{f\in L^0[0,\alpha):\, \exists A, m(A^c)<\infty, f\chi_A\in L_{\infty}[0,\alpha)\}.
\end{align*}
The map $f\mapsto N_f$ is a $*$-isomorphism from $S([0,\alpha),m)$ onto $ S(\mathcal{N},\eta)$.
\end{proof}
\textbf{Fact 5.} $\mu(N_f)=\mu(f)$ and for any symmetric function space $E$ we have that $E(\mathcal{N},\eta)$ is isometrically isomorphic to the function space $E$.
\begin{proof}
Note that  $d(N_f,s)=\eta(e^{\abs{x}}(s,\infty))=m\{t:\abs{f(t)}>s\}=d(f,s)$. Hence, for $N_f\in S(\mathcal{N},\eta)$, the generalized singular value function $\mu( N_f )$ is precisely the decreasing rearrangement $\mu(f)$ of the function $f\in S([0,\alpha),m)$. 
\end{proof}

The characterizations of many local geometric properties of an operator $x$ in noncommutative spaces will include some conditions on $n(x)$ and $s(x)$, the null and range projections of $x$.
We will see frequently the  two conditions (i) and (ii) stated below  for $x\in S\Mtau$.  Those conditions  can be easily translated to the commutative settings as follows.\\

\textbf{Fact 6.}  If $\M=\mathcal{N}$,  $\tau=\eta$ and $x= N_f$ for some $f\in L^0[0,\alpha)$, then the conditions \begin{itemize}
\item[](i) $\mu(\infty,x)=0$ or (ii) $n(x)\mathcal{M}n(x^*)=0$ and $\abs{x}\geq\mu(\infty,x)s(x)$, 
\end{itemize}
where   $n(x)\M n(x^*)=0$  means that for any $y\in \M$, $n(x)yn(x^*)=0$,
are equivalent to 
\begin{itemize}
\item[](i') $\abs{f}\geq\mu(\infty,f)\chi_{[0,\alpha)}$.
\end{itemize}
\begin{proof}
By Fact 5,  if $x=N_f\in S(\mathcal{N},\eta)$, then $\mu(x)=\mu(f)$ and (i) gives $\mu(\infty,f)=0$.  It is not difficult to check that  $s(x)=N_{\chi_{\supp f}} $ and $n(x)=N_{\chi_{(\supp f)^c}}$. Similarly, $s(x^*)=N_{\chi_{\supp \overline{f}}} =N_{\chi_{\supp f}}$ and $n(x^*)=N_{\chi_{(\supp \overline{f})^c}}=N_{\chi_{(\supp f)^c}}$. Hence in view of the condition $n(x)\mathcal{N}n(x^*)=0$,  taking  $N_{\chi_{[0,\alpha)}}\in\mathcal{N}$  we get  $0=N_{\chi_{(\supp f)^c}}N_{\chi_{[0,\alpha)}}N_{\chi_{(\supp f)^c}}=N_{\chi_{(\supp f)^c}}$. Therefore $\chi_{(\supp f)^c}=0$ a.e., and so $s(x)=N_{\chi_{\supp f}}=N_{\chi_{[0,\alpha)}}$. If additionally $\abs{x}\geq \mu(\infty,x)s(x)$, then we have $N_{|f|}\geq \mu(\infty, f)N_{\chi_{[0,\alpha)}}$ and $|f|\geq \mu(\infty,f)\chi_{[0,\alpha)}$. Thus (i) and (ii) imply (i'). Suppose now that (i') holds, that is $|f|\geq \mu(\infty, f)\chi_{[0,\alpha)}$, where $x=N_f\in S(\mathcal{N},\eta)$. Then either $\mu(\infty,f)=\mu(\infty, x)=0$ or $(\supp f)^c=0$ a.e. and $n(x)=N_{\chi_{(\supp f)^c}}=0$.  Hence in either case $|x|\geq \mu(\infty,x) \one$ and either (i) or (ii) is satisfied.
\end{proof}

\section{Trace preserving isomorphisms} 
\label{sec:isom}

Recall that given two $*$-algebras $\mathcal A$ and $\mathcal B$, the mapping $\Phi:\mathcal{A}\to\mathcal{B}$ is  called a \textit{$*$-homomorphism} if $\Phi$ is an algebra homomorphism and $\Phi(x^*)=(\Phi(x))^*$ for all $x\in\mathcal{A}$. If, in addition, $\mathcal A$ and $\mathcal B$ are unital and $\Phi(\one_A)=\one_B$,  where $\one_A$ and $\one_B$ are units in $\mathcal{A}$ and $\mathcal{B}$ respectively,  then $\Phi$ is called \textit{unital $*$-homomorphism}. The term $*$-\textit{isomorphism} stands for an injective $*$-homomorphism. Observe that every  $*$-homomorphism $\Phi:\mathcal A\to\mathcal B$ is positive, that is for any $x\in \mathcal{A}$, if $x\ge 0$ then $\Phi(x)\ge 0$. Indeed, since $\Phi(\sqrt{x})=\Phi((\sqrt{x})^*)=\left(\Phi(\sqrt{x})\right)^*$, if follows that
\[
\Phi(x)=\Phi(\sqrt{x}\sqrt{x})=\Phi(\sqrt{x})\Phi(\sqrt{x})=(\Phi(\sqrt{x}))^*\Phi(\sqrt{x})=\abs{\Phi(\sqrt{x})}^2\ge 0.
\]

J. Arazy in \cite{A1} observed  that $E$ is isometric to a $1$-complemented subspace of $C_E$, and therefore many geometric properties of $C_E$ are inherited by $E$. Moreover, for each $x\in C_E$ the above isometry can be found with additional property that  it maps the singular sequence  $S(x)$ into $x$. Hence also locally,  a geometric property of $x$  can be passed along into the sequence $S(x)$.

The J. Arazy's result relies on the Schmidt representation of a compact operator. The symmetric sequence space is embedded in the subspace of diagonal operators in $B(H)$. We include below the result with an outline of a proof. 

\begin{proposition}\cite[Proposition 1.1]{A1}
\label{prop:isomarazy}
Let $E\neq \ell_\infty$ be a symmetric sequence space and $x\in C_E$. Then there exists a linear isometry  $V:E\to C_E$ such that  $V(S(x))=x$. If $x\geq 0$ then $V$ is  in addition a $*$-isomorphism. Moreover, there is a contractive projection from $C_E$ onto $V(E)$. 
\end{proposition}
 \begin{proof}
 Fix $x\in C_E$ and let $x=\sum_{n=1}^\infty s_n(x)\langle\cdot,e_n\rangle f_n$ be its Schmidt representation, where $\{e_n\}$ and $\{f_n\}$ are orthonormal sequences in $H$. Define $V:E\to C_E$ by 
 \[
 V(\lambda)=\sum_{n=1}^\infty \lambda_n\langle\cdot,e_n\rangle f_n,\quad\text{where}\quad \lambda=\{\lambda_n\}\in E.
 \]
 Clearly $V(S(x))=x$.  Note that $|V(\lambda)|^2=V(\lambda)^*V(\lambda)=\sum_{n=1}^\infty \overline{\lambda_n}\langle\sum_{k=1}^\infty \lambda_k\langle\cdot,e_k\rangle f_k,f_n\rangle e_n=\sum_{n=1}^\infty |\lambda_n|^2\langle\cdot,e_n\rangle e_n $. Hence the eigenvalues of $|V(\lambda)|$ are $|\lambda_n|$.  In view of $E\subset c_0$, for every $\lambda\in E$, the sequence of singular numbers $s_n(V(\lambda))=\sqrt{s_n(|V(\lambda)|^2)}$ is a decreasing permutation of $|\lambda|=\{|\lambda_n|\}$ approaching zero. Hence $V(\lambda)$ is a compact operator and $\|V(\lambda)\|_{C_E}=\|\lambda\|_E$.  If $x\geq 0$ then $x=\sum_{n=1}^\infty s_n(x)\langle\cdot,e_n\rangle e_n$ and  $V$ is also a $*$-isomorphism.
 
Define next $P:C_E\to C_E$ by
\[
Py=\sum_{n=1}^\infty\langle ye_n,f_n\rangle\langle\cdot,e_n\rangle f_n,\quad y\in C_E.
\] 
By \cite[Proposition 2.6]{Sim}, for any $y\in C_E$ we have 
\[
\|y\|_{C_E}=\sup\{\|\{\langle y\phi_n,\psi_n\rangle\}\|_E : \ \text{all orthonormal sets $\{\phi_n\},\,\{\psi_n\}$ in $H$}\}.
\] 
 Hence if $y\in C_E$, then $\{\langle ye_n,f_n\rangle\}\in E$  and $P(C_E)\subseteq V(E)$. Let $z\in V(E)$ and $\lambda=\{\lambda_n\}\in E$ be such that $V(\lambda)=z$. Then for all $n\in\mathbb{N}$, $\langle ze_n, f_n\rangle =\lambda_n$ and therefore $Pz=z$. Thus $P(C_E)= V(E)$. Moreover, $\|Py\|_{C_E}=\|\{\langle ye_n,f_n\rangle\}\|_E\leq \|y\|_{C_E}$ for every $y\in C_E$. Hence $\|P\|\leq 1$.  Finally, it is easy to verify that $P^2=P$ and so $P$ is a contractive projection from $C_E$ onto $V(E)$.
 \end{proof}

It turns out that J. Arazy's result can be extended to noncommutative symmetric function spaces $\nonsp$, but only under certain conditions  imposed on the operator $x$ itself, the trace $\tau$ and the von Neumann algebra $\M$.

\begin{proposition}\cite{noncomm}, \cite[Lemma 1.3]{CKS1992}
\label{isom2}
Let $\mathcal{M}$ be a non-atomic von Neumann algebra with a faithful, normal,  $\sigma$-finite trace $\tau$ and $x\in
S_0^+\left(\mathcal{M},\tau\right)$.  Then there exists a non-atomic commutative von Neumann subalgebra $\mathcal{N}$ in $ \M$  and a $*$-isomorphism $U$ from the $*$-algebra $S(\mathcal{N},\tau)$ onto the $*$-algebra $S\left(\left[0,\tauone\right),m\right)$ such that $x\in S(\mathcal{N},\tau)$ and $\mu(y)=\mu(Uy)$ for every $y\in S(\mathcal{N},\tau)$.
\end{proposition}

Given an operator $x\in S\Mtau$ and a projection $p\in \mathcal{P}(\M)$ we define the von Neumann algebra $\M_{p}=\{py|_{p(H)}:\quad y\in\M\}$. It is known that there is a unital  $*$-isomorphism from $S\left(\M_p,\tau_p\right)$ onto $pS\Mtau p$. Moreover, the decreasing rearrangement $\mu^{\tau_{p}}$ computed with respect to the von Neumann algebra $\left(\M_{p},\tau_{p}\right)$ is given by $\mu^{\tau_{p}}(y)=\mu(pyp)$, $y\in S\left(\M_p,\tau_p\right)$.
See \cite{CKK2012, DDPnoncomm} for details.

Using measure preserving transformations which retrieve functions from their decreasing rearrangements and the inverse operator $U^{-1}$ from  Proposition \ref{isom2} the following can be shown. 

\begin{proposition} \cite{noncomm}
\label{isom1}
Suppose that $\mathcal{M}$ is a non-atomic von Neumann algebra with a faithful, normal  trace $\tau$. 
Let $x\in\left( L_1(\mathcal{M},\tau)+\mathcal{M}\right)\cap S_0^+\left(\mathcal{M},\tau\right)$. Then there exist a non-atomic
commutative von Neumann subalgebra 
$\mathcal{N}\subset s(x)\M s(x)$ and a unital $*$-isomorphism
$V$ acting from the $*$-algebra  $S\left(\left[0,\tau(s(x))\right),m\right)$ 
into the $*$-algebra $S(\mathcal{N},\tau)$, such that 
\[ V\mu(x)=x\ \ \ \text{ and }\ \ \ \mu(V(f))=\mu(f)\ \ \text{ for all } f\in S\left(\left[0,\tau(s(x))\right),m\right).
\]
\end{proposition}
\begin{proof}
Observe first that since $\mu(\infty, x)=0$, $\tau(e^{\abs{x}}(\frac 1n, \infty))<\infty$ for every $n\in\mathbb{N}$. Since $e^{\abs{x}}(\frac 1n, \infty)\uparrow e^{\abs{x}}(0,\infty)=s(x)$ the restriction $\tau|_{s(x)\M s(x)}$ is $\sigma$-finite. By Proposition \ref{isom2} there is a non-atomic commutative subalgebra $\mathcal{N}$ of $s(x)\M s(x)$ and a $*$-isomorphism $U$ from $S(\mathcal{N},\tau)$ onto $S([0,\tau(s(x))),m)$ such that  $x\in S(\mathcal{N},\tau)$ and $\mu(y)=\mu(Uy)$ for every $y\in S(\mathcal{N},\tau)$. Set $f=Ux$. Since $x\geq 0$ and every $*$-homomorphism preserves the order, $f\geq 0$. We also have $\mu(f)=\mu(Ux)=\mu(x)$. In particular $\mu(\infty, f)=\mu(\infty, x)=0$ and $m(\supp f)=m(\supp \mu(f))=m(\supp \mu(x))=\tau(s(x))$.  By \cite[ChII, Corollary 7.6]{BS}, there is a measure preserving transformation $\sigma:\supp f\to [0,\tau(s(x)))$ such that $f(t)=\mu(\sigma(t), f)=\mu(\sigma(t), x)$ for every $t\in \supp f$. The term measure preserving means that $m(\sigma^{-1}(E))=m(E)$ for every measurable subset $E\subseteq [0,\tau(s(x)))$.

Define a $*$-homomorphism $V$ from $S([0,\tau(s(x))),m)$ into $S(\mathcal{N},\tau)$ by setting
\[
V(g)=U^{-1}(g\circ \sigma),\quad g\in S([0,\tau(s(x))),m).
\]
We have, 
\[
V(\mu(x))=U^{-1}(\mu(x)\circ \sigma)=U^{-1}(\mu(f)\circ \sigma)=U^{-1}(f)=x.\]
 Moreover, for any $g\in S([0,\tau(s(x))),m)$,
 \[
\mu(V(g))=\mu(U^{-1}(g\circ \sigma))=\mu(g\circ \sigma)=\mu(g).
 \]

\end{proof}
\begin{proposition} \cite{noncomm} 
\label{isom3}
Suppose that $\mathcal{M}$ is a non-atomic von Neumann algebra with a faithful, normal,  $\sigma$-finite trace $\tau$.
Let $x\in\left( L_1(\mathcal{M},\tau)+\mathcal{M}\right)\cap S_0^+\left(\mathcal{M},\tau\right)$ and $\tau(s(x))<\infty$. Then there exist a non-atomic commutative von Neumann subalgebra $\mathcal{N}\subset\M$ and a unital $*$-isomorphism $V$ acting from the $*$-algebra $S\left(\left[0,\tauone\right),m\right)$ into the $*$-algebra $S(\mathcal{N},\tau)$, such that 
\[ V\mu(x)=x\ \ \ \text{ and }\ \ \ \mu(V(f))=\mu(f)\ \ \text{ for all } f\in S\left([0,\tauone),m\right).
\]
\end{proposition}
\begin{proof}
The proof of this proposition is analogous to the proof above. The only difference is the lack of the restriction of $\M$ to $s(x)\M s(x)$, and the extension of the measure preserving transformation $\sigma$ to the whole interval $[0,\tauone)$. Assume that $f=U(x)$ as above, and so $m(\supp f)= \tau(s(x))<\infty$. Let $\sigma_1$ be a measure preserving transformation from $\supp f$ to $[0,\tau(s(x)))$ such that $f(t)=\mu(\sigma_1(t), f)$ for every $t\in \supp f$. Since $m(\supp f)<\infty$, we have that $(\supp f)^c$ and the interval $[m(\supp f),\tauone)$ have the same measure. Indeed, if $\tauone=\infty$ then both measures are infinite as well. If $\tauone<\infty$, then both measures are equal to $\tauone -m(\supp f)$. It is not difficult to find a measure preserving transformation $\sigma_2: (\supp f)^c\to [m(\supp f),\tauone)$. Since $f(t)=\mu(\sigma_2(t), f)=0$ for all $t\in (\supp f)^c$, setting $\sigma=\sigma_1\chi_{\supp f}+\sigma_2\chi_{(\supp f)^c}$, we get a measure preserving transformation from $[0,\tauone)$ to $[0,\tauone)$ such that $f=\mu(f)\circ \sigma$.  

Finally define a $*$-homomorphism $V$ from $S([0,\tauone),m)$ into $S(\mathcal{N},\tau)$ by setting $V(g)=U^{-1}(g\circ \sigma)$, $g\in S([0,\tauone),m)$.
\end{proof}

\begin{corollary}
\label{cor:isomglobal}
If $\M$ is non-atomic then the symmetric function space $E$ is isometrically embedded into $\nonsp$. Similarly, the symmetric sequence space $E \ne \ell_\infty$ is isometrically embedded into $C_E$. Furthermore, those embeddings are order preserving.
\end{corollary}
\begin{proof}
Let $x\in S^+\Mtau$ be such that $\tau(s(x))=\tauone$. In fact there is a projection $p\in P(\M)$ such that $\tau(p)=\tauone$. Then by Proposition \ref{isom1} we can choose a $*$-isomorphism $V:S([0,\tauone),m)\to S(\mathcal{N},\tau)$, where $\mathcal{N}\subset s(x)\M s(x)$, such that $\mu(V(f))=\mu(f)$ for every $f\in S([0,\tauone),m)$. Hence if $f\in E$  then $V(f)\in \nonsp$, and $\|V(f)\|_{\nonsp}=\|\mu(V(f))\|_E=\|\mu(f)\|_E=\|f\|_E$.  As explained at the beginning of this section every $*$-homomorphism is positive. Hence $V(f)\leq V(g)$ whenever $f\leq g$ and $V$ preserves the order of $E$. For sequence spaces the claim follows  by Proposition \ref{prop:isomarazy}.
\end{proof}
\begin{corollary}
\label{cor:isom}
Let $\mathcal{M}$ be a non-atomic von Neumann algebra with a faithful, normal,  $\sigma$-finite trace $\tau$, $x\in S\Mtau$, and $\abs{x}\geq \mu(\infty,x)s(x)$. Denote by $p=s(\abs{x}-\mu(\infty,x)s(x))$ and define projection $q\in \mathcal{P}(\M)$ in the following way.
\begin{itemize}
\item[{(i)}] If $\tau(s(x))<\infty$ set $q=\one$.
\item[{(ii)}] If $\tau(s(x))=\infty$ and $\tau(p)<\infty$, set $q=s(x)$.
\item[{(iii)}] If $\tau(p)=\infty$, set $q=p$.
\end{itemize}
 Then there exist a non-atomic commutative von Neumann subalgebra $\mathcal{N}\subset q\M q$ and a unital $*$-isomorphism
$V$ acting from the $*$-algebra $S\left(\left[0,\tauone\right),m\right)$ into the $*$-algebra $S(\mathcal{N},\tau)$, such that 
\[ V\mu(x)=\abs{x}q\ \ \ \text{ and }\ \ \ \mu(V(f))=\mu(f).
\]
for all  $f\in S\left([0,\tauone),m\right)$.
\end{corollary}
\begin{proof}
Observe that $p=s(\abs{x}-\mu(\infty,x)s(x))=e^{\abs{x}}(\mu(\infty,x),\infty)\leq s(x)$.  Hence if $\tau(p)=\infty$ then also $\tau(s(x))=\infty$, and therefore conditions (i), (ii), and (iii) give all possible cases. Furthermore, by \cite[Proposition 1.1]{czer-kam2010}, $\mu(\abs{x}-\mu(\infty,x)s(x))=\mu(x)-\mu(\infty,x)$, and so $\abs{x}-\mu(\infty,x)s(x)\in S_0^+\Mtau$. 

Note that in either case $\tau(q)=\tauone$. Hence in view of Lemma \ref{lm:singfun} (7),  it follows that $\mu(\abs{x}q)=\mu(x)\chi_{[0,\tau(q))}=\mu(x)$.

Case (i). Since $\tau(s(x))<\infty$, we have that $\mu(\infty,x)=0$. Therefore the claim is an immediate consequence of  Proposition \ref{isom3} applied to $\abs{x}$.

Case (ii). Let $\tau(s(x))=\infty$, $\tau(p)<\infty$ and $q=s(x)$. 
Applying Proposition \ref{isom3} to the operator $\abs{x}-\mu(\infty,x)s(x)=s(x)(\abs{x}-\mu(\infty,x)s(x))s(x)\in s(x)S\Mtau s(x)$ and to the von Neumann algebra $s(x)\M s(x)$, there exist a non-atomic commutative von Neumann algebra $\mathcal{N}\subset s(x)\M s(x)$ and a $*$-isomorphism $V$ from  $S\left(\left[0,\tau(s(x))\right),m\right)= S\left(\left[0,\infty\right),m\right)$ into $S(\mathcal{N},\tau)$ such that 
\[ 
V\mu(\abs{x}-\mu(\infty,x)s(x))
=\abs{x}-\mu(\infty,x)s(x)\,\text{ and }\, \mu(V(f))=\mu(f)
\]
for all  $f\in S\left([0,\infty),m\right)$.
Since $V(\chi_{[0,\infty)})=s(x)$,
\begin{align*}
\abs{x}-\mu(\infty,x)s(x)&=V\mu(\abs{x}-\mu(\infty,x)s(x))=V(\mu(x)-\mu(\infty,x))\\
&=V\mu(x)-\mu(\infty,x)V(\chi_{[0,\infty)})=V\mu(x)-\mu(\infty,x)s(x),
\end{align*}
and consequently $V\mu(x)=\abs{x}=\abs{x}s(x)$.

Case (iii). Assume that $\tau(p)=\infty$ and $q=p$. By Proposition \ref{isom1} applied to the operator $\abs{x}-\mu(\infty,x)s(x)$ and von Neumann algebra $\M$, there exist a non-atomic commutative von Neumann algebra $\mathcal{N}\subset p\M p$ and a $*$-isomorphism $V$ from  $S\left(\left[0,\tau(p)\right),m\right)= S\left(\left[0,\infty\right),m\right)$ into $S(\mathcal{N},\tau)$ such that 
\[ 
V\mu(\abs{x}-\mu(\infty,x)s(x))=\abs{x}-
\mu(\infty,x)s(x)\,\text{ and }\,\mu(V(f))=\mu(f)
\]
for all $f\in S\left([0,\infty),m\right)$.
Since  $p\leq s(x)$, \[\abs{x}-\mu(\infty,x)s(x)=(\abs{x}-\mu(\infty,x)s(x))p
=\abs{x}p-\mu(\infty,x)p\] and $V(\chi_{[0,\infty)})=p$. Thus  again we have
\begin{align*}
\abs{x}p-\mu(\infty,x)p&=\abs{x}-\mu(\infty,x)s(x)
=V\mu(\abs{x}-\mu(\infty,x)s(x))\\
&=V\left(\mu(x)-\mu(\infty,x)\right)
=V\mu(x)-\mu(\infty,x)V(\chi_{[0,\infty)})\\
&=V\mu(x)-\mu(\infty,x)p,
\end{align*}
and $V\mu(x)=\abs{x}p$.
\end{proof}

\begin{corollary}
\label{cor:isom1}
Let $\mathcal{M}$ be a non-atomic von Neumann algebra with a faithful, normal,  $\sigma$-finite trace $\tau$, 
and $x\in S\Mtau$ with $r=e^{\abs{x}}(\mu(\infty,x),\infty)$.  
Set $q=\one$ whenever $\tau(r)<\infty$, and $q=r$ if $\tau(r)=\infty$. 

Then there exist a non-atomic commutative von Neumann subalgebra $\mathcal{N}\subseteq q\M q$ and a unital $*$-isomorphism
$V$ acting from the $*$-algebra $S\left(\left[0,\tauone\right),m\right)$ into the $*$-algebra $S(\mathcal{N},\tau)$, such that 
\[ V\mu(x)=\abs{x}r+\mu(\infty, x)V\chi_{[\tau(r),\infty)}\ \ \ \text{ and }\ \ \ \mu(Vf)=\mu(f)\ \ \text{ for all } f\in S\left([0,\tauone),m\right).
\]
\end{corollary}
\begin{proof}
Consider the operator $x_0=\abs{x}r$, where we have $s(x_0)=r$ and $x_0\geq \mu(\infty, x) s(x_0)$. Moreover, $\mu(x_0)=\mu(x)\chi_{[0,\tau(r))}$ by Lemma \ref{lm:singfun} (7). If $\tau(r)<\infty$, then $\mu(\infty, x_0)=0$. Otherwise $\mu(x_0)=\mu(x)$. In either case $x_0\geq \mu(\infty, x_0)s(x_0)$.  Moreover, $p=s(x_0-\mu(\infty, x_0) s(x_0))=e^{x_0}(\mu(\infty, x_0),\infty)=e^{x}(\mu(\infty, x),\infty)=r$. If $\tau(r)=\infty$ set $q=p=r$, and if $\tau(r)<\infty$, $q=\one$. By Corollary \ref{cor:isom} (i) and (iii) applied to $x_0$ there exist a non-atomic commutative von Neumann subalgebra $\mathcal{N}\subseteq q\M q$ and a unital $*$-isomorphism
$V$ acting from the $*$-algebra $S\left(\left[0,\tauone\right),m\right)$ into the $*$-algebra $S(\mathcal{N},\tau)$, such that 
\[ V\mu(x_0)=x_0q\ \ \ \text{ and }\ \ \ \mu(Vf)=\mu(f)\ \ \text{ for all } f\in S\left([0,\tauone),m\right).
\]

In case of $\tau(r)=\infty$, $\mu(x)=\mu(x_0)$ and $q=r$, and therefore  $V\mu(x)=x_0r=\abs{x}r$.

Consider now the case when $\tau(r)=\tau(e^{\abs{x}}(\mu(\infty,x),\infty))<\infty$ with $q=\one$. Since $\mu(\infty, x)=\inf\{s\geq 0:\, \tau(e^{\abs{x}}(s,\infty))<\infty\}$, we have that $\tau(e^{\abs{x}}(s,\infty))=\infty$ for all $s\in[0, \mu(\infty,x))$. Recalling the definition of $\mu(t,x)=\inf\{s\geq 0:\, \tau(e^{\abs{x}}(s,\infty))\leq t\}$, it is easy to observe that $\mu(t,x)=\mu(\infty, x)$ for all $t\geq \tau(e^{\abs{x}}(\mu(\infty,x),\infty))=\tau(r)$. Hence 
\[
\mu(x)=\mu(x)\chi_{[0,\tau(r))}+\mu(\infty,x)\chi_{[\tau(r),\infty)}=\mu(x_0)+\mu(\infty,x)\chi_{[\tau(r),\infty)},
\]
and 
\[
V\mu(x)=V\mu(x_0)+\mu(\infty,x)V\chi_{[\tau(r),\infty)}=x_0+\mu(\infty,x)V\chi_{[\tau(r),\infty)}=\abs{x}r+\mu(\infty,x)V\chi_{[\tau(r),\infty)}.
\]

\end{proof}

\section{Non-atomic extension of $\nonsp$}
We will describe below the construction of a non-atomic von Neumann algebra $\mathcal{A}$ with the trace $\kappa$, such that $E\Mtau$  embeds isometrically into $E(\mathcal{A},\kappa)$, for any symmetric function space $E$.

  Let
$\mathcal{A}=\mathcal{N}\overline\otimes \mathcal{M}$ be a tensor
product of von Neumann algebras $\mathcal{N}$ and $\mathcal{M}$, where $\mathcal{N}$ is a commutative von Neumann algebra identified with $L_{\infty}[0,1]$ with the trace $\eta$ (see section \ref{sec:symmfun}).  Let 
$\kappa= \eta\otimes \tau$ be a tensor product of the traces $\eta$ and $\tau$, that is $\kappa(N_f\otimes x)=\eta(N_f)\tau(x)$  \cite{KR,Takesaki}. It is well known that
$\left(\mathcal{A},\kappa\right)$ has no atoms \cite[Lemma 2.3.18]{LSZ2013}.

Let $\mathds{1}$ be the identity operator on $L^2[0,1]$ and denote by $\Complex\mathds{1}=\{\lambda \mathds{1}:\lambda\in\Complex\}$.
 Let $x\in
S\left(\mathcal{M},\tau\right)$ and consider a linear subspace $D$
in $ L_2[0,1]\otimes H$ generated by the vectors of the form
$\zeta\otimes\xi$, where  $\zeta\in
L_2[0,1]$ and $\xi\in D(x)\subset H$. For every $\displaystyle \alpha=\sum_{i=1}^n\zeta_i\otimes\xi_i
\in D$ define $\displaystyle
(\mathds{1}\otimes x)(\alpha)=\sum_{i=1}^n\zeta_i\otimes x(\xi_i)$.
The linear operator $\mathds{1}\otimes x:D\rightarrow L_2[0,1]\otimes
H $ with  domain $D$ is preclosed, and by Lemma 1.2 in \cite{CKS1992} its closure
$\mathds{1}\overline\otimes x$ is contained in
$S(\Complex\mathds{1}\otimes \M,\kappa)$.

The map $\pi:x\rightarrow \mathds{1}\otimes x$, $x\in\M$, is a unital trace preserving $*$-isomorphism from $\M$ onto the von Neumann subalgebra $\Complex\mathds{1}\otimes\M$. Consequently, $\pi$ extends uniquely to a $*$-isomorphism $\tilde{\pi}$ from $S\Mtau$ onto $S(\Complex\mathds{1}\otimes \M,\kappa)$ \cite{noncomm}.  In fact one can show that $\tilde{\pi}(x)=\mathds{1}\overline\otimes x$.

Since every $*$-homomorphism is an order preserving map, $x\geq 0$ if and only if $\mathds{1}\overline\otimes x\geq 0$, where $x\in S\Mtau$.
 The spectral measure $e^{\tilde{\pi}(x)}$ of $\tilde{\pi}(x)$ is given by $e^{\tilde{\pi}(x)}(B)=\pi (e^x(B))$, that is  $e^{\mathds{1}\overline\otimes x}(B)=\mathds{1}\otimes e^{x}(B)$ for any Borel set $B\subset \mathbb{R}$. Hence $\kappa\left(e^{\mathds{1}\overline\otimes x}(s,\infty)\right)=\kappa\left( \mathds{1}\otimes e^{x}(s,\infty)\right)=\tau(e^x(s,\infty))$ for any $s>0$. Consequently $\tilde{\pi}$ preserves the singular value function in the sense that $\tilde{\mu}(\mathds{1}\overline\otimes x)=\mu(x)$, where $\tilde{\mu}(\mathds{1}\overline\otimes x)$ is the singular value function of $\mathds{1}\overline\otimes x$ computed with respect to the von Neumann algebra $\Complex \mathds{1}\otimes\M$ and the trace $\kappa$ \cite[Lemma 2.3.18]{LSZ2013}.
Thus
\[
\|\tilde{\pi}(x)\|_{E(\Complex \mathds{1}\otimes\M,\kappa)}=\norme{\tilde{\mu}(\mathds{1}\overline\otimes x)}=\norme{\mu(x)}=\normcomm{x},
\]
where 
\begin{align*}
E(\Complex \mathds{1}\otimes\M,\kappa)&=\{\mathds{1}\overline\otimes x\in S(\Complex \mathds{1}\otimes\M,\kappa): \tilde{\mu}(\mathds{1}\overline\otimes x)\in E\}
\\&=\{\mathds{1}\overline\otimes x: x\in S\Mtau\text{ and }\mu(x)\in E\}.
\end{align*}

Hence $\tilde{\pi}$ is a $*$-isomorphism which is also an isometry from $\nonsp$ onto $E(\Complex \mathds{1}\otimes\M,\kappa)$.
 We refer reader to  \cite{CKS1992,noncomm,LSZ2013,  St} for details.
\subsection{ Removing the non-atomicity assumption} \label{non-atom}
Many authors investigating geometric properties of $\nonsp$ aspire to show that $\nonsp$ has the property $P$ if and only if $E$  has it.  Very often for the property $P$ to carry from $\nonsp$ into $E$ it is necessary to assume  non-atomicity of $\M$.

On the other hand, suppose we showed that if $E$ has the property $P$ then so does  $\nonsp$ for any non-atomic von Neumann algebra $\M$.  
Then this result can be extended to an arbitrary von Neumann algebra provided that the property $P$ is preserved by linear isometries and passes to subspaces. Indeed,  since $\mathcal{A}$ is non-atomic, so  $E(\mathcal{A},\kappa)$ has property $P$.  As explained in the section above, the $*$-isomorphism $\tilde{\pi} :\nonsp \to E(\Complex \mathds{1}\otimes\M,\kappa) \subset  E(\mathcal{A},\kappa)$  embeds isometrically $\nonsp$  into $E(\mathcal{A},\kappa)$. Hence $\nonsp$ must possess the property $P$, where $\M$ is an arbitrary von Neumann algebra.

\textbf{Convention.}\quad Unless stated otherwise, $\mathcal{M}$  will denote   a semifinite von Neumann algebra with a fixed semifinite, faithful, normal trace $\tau$. The symbol $E$ will stand for a symmetric function  space on $[0,\alpha)$. If $E$ is a sequence symmetric space then it is always assumed that $E\neq \ell_\infty$. Given a normed space $(X,\norm{\cdot})$, let $B_X$ and $S_X$  be the unit ball and the unit sphere in $X$, respectively. 

\section{Extreme points and strict convexity}

Let $C$ be a convex subset in a linear space. We call $x\in C$  an \textit{extreme point} of $C$ if $x\pm y\in C$ implies $y=0$. Equivalently, we can say that $x$ is an extreme point of $C$ if it does not lie in any open line segment joining two different points in $C$. That is $x$ is an extreme point of $C$ if $x=\lambda y+(1-\lambda)z$, for some $y,z\in C$ and $\lambda\in \mathbb{R}$,  implies that $x=y=z$. We say that a normed space $(X, \|\cdot\|)$ is strictly convex whenever every element of its unit sphere is an extreme point.

The Krein-Milman theorem  states that every compact and convex subset $K$ of a locally convex linear space is the closed convex hull of its extreme points.

In this section we will present the work on extreme points of the unit balls in symmetric noncommutative spaces. J. Holub in \cite{Holub} was first to characterize extreme points in the trace class $C_1$. J. Arazy extended the result to all unitary matrix spaces $C_E$. More precisely,  J. Arazy  showed the following. 

\begin{theorem} \cite[Theorem 2.1]{A1} 
\label{thm:arazyextreme}
Let $E$ be a symmetric sequence space, $x\in C_E$, $\|x\|_{C_E}=1$. Then $x$ is extreme point of $B_{C_E}$ if and only if $S(x)$ is an extreme point of $B_E$. 
\end{theorem}

  Holub's characterization differed from  Arazy's, as he did not relate extreme operators with their sequences of singular numbers. However, we will demonstrate  below  that their descriptions are equivalent.

\begin{theorem}
 Let $x\in C_1$, $\|x\|_{C_1}=1$. 
The two results are equivalent.
\begin{itemize} 
\item[{(i)}] \cite[Theorem 3.1]{Holub} Let $x\in C_1$, $\|x\|_{C_1}=1$. Then $x$ is extreme of $B_{C_1}$ if and only if $x$ is a one-dimensional operator.
\item[{(ii)}] \cite[Theorem 2.1]{A1} Let $x\in C_1$, $\|x\|_{C_1}=1$. Then $x$ is extreme of $B_{C_1}$ if and only if $S(x)$ is extreme of $B_{\ell_1}$.
\end{itemize}
\end{theorem}
\begin{proof}
(i)$\implies$ (ii)
If $x$ is extreme of $B_{C_1}$ then by (i), $x$ is one dimensional. Then the Schmidt representation of $x$ is $x(\cdot)=s_1(x)\langle \cdot,e_1\rangle f_1$, where $e_1, f_1$ are normalized vectors in $H$.  Since $\|x\|_{C_1}=\|S(x)\|_{\ell_1}=1$ it follows that $s_1(x)=1$ and $s_i(x)=0$, $i=2,3,\dots$.  So $S(x)=\phi_1=(1,0,\dots, 0)$ is extreme of the unit ball of $\ell_1$.

Now suppose that $S(x)$ is extreme of $B_{\ell_1}$. But the only extreme points  of the unit ball in $\ell_1$ are the unit vectors $\pm\phi_n=\{\pm\phi_n(i)\}\in \ell_1$,   where $\phi_n(i)=0$ for $i\neq n$ and $\phi_n(n)=1$.  So $S(x)=\phi_1$. It means that $x(\cdot)=\sum_{n=1}^\infty s_n(x)\langle \cdot, e_n\rangle f_n=\langle \cdot,e_1\rangle f_1$   and $x$ is one-dimensional.
Hence (i) implies that $x$ is extreme.

(ii)$\implies$ (i)
Suppose that $x$ is extreme of $B_{C_1}$ then by  (ii)  $S(x)$ is extreme of $B_{\ell_1}$ and so $S(x)=\phi_1$. Then  $x(\cdot)=\langle \cdot,e_1\rangle f_1$ and $x$ is one-dimensional. Assume next that $x$ is one-dimensional.  Then $x(\cdot)=s_1(x)\langle \cdot,e_1\rangle f_1$ and $s_1(x)=1$. Hence $S(x)$ is extreme of $B_{\ell_1}$ and  by (ii),  $x$ is extreme of $B_{C_1}$.

\end{proof}

V. Chilin, A. Krygin and F. Sukochev in \cite{CKS1992ext}  extended J. Arazy's result to  symmetric spaces of measurable operators $\nonsp$. Here the relations between extreme operators and their singular value functions become more complex. If $x$ is an extreme point of the unit ball in $\nonsp$ then $\mu(x)$ is an extreme point of the unit ball in $E$. However, $x$ does not always inherits extreme property from its singular value function $\mu(x)$. For it to happen, $\mu(x)$ has to satisfy one of the conditions listed below.

\begin{theorem}
\label{thm:ext}
\cite[Theorem 1.1]{CKS1992ext} Let $\M$ be  non-atomic. Then $x\in S_{\nonsp}$ is an extreme point of $B_{\nonsp}$ if and only if $\mu(x)\in S_E$ is an extreme point of $B_E$ and one of the following conditions hold.
\begin{itemize}
\item[{(i)}] $\mu(\infty, x)=0$,
\item[{(ii)}] $n(x)\M n(x^*)=0$  and $\abs{x}\geq \mu(\infty,x)s(x)$.
\end{itemize}

\end{theorem}

\begin{remark} \label{rm:removingnonatom}
 Observe that if $\mu(x)$ is extreme in $B_E$ and (i) is satisfied then $x$ is extreme in $\nonsp$, regardless whether $\M$ is non-atomic. Indeed, since $\tilde{\mu}(\mathds{1}\overline\otimes x)=\mu(x)$, we have that $\tilde{\mu}(\mathds{1}\overline\otimes x)$ is an extreme point of $B_E$ and $\tilde{\mu}(\infty,\mathds{1}\overline\otimes x)=0$. Therefore by Theorem \ref{thm:ext}, $\mathds{1}\overline\otimes x$ is an extreme point of the unit ball in $E(\mathcal{A},\kappa)$. It follows that $x$ is an extreme point of $B_{\nonsp}$. In fact letting $\|x\pm y\|_{\nonsp}\leq 1$,  where $y\in \nonsp$ we have $\|\mathds{1}\overline\otimes x\pm \mathds{1}\overline\otimes y\|_{E(\mathcal{A},\kappa)}=\|\tilde{\mu}(\mathds{1}\overline\otimes (x\pm y))\|_E=\|\mu(x\pm y)\|_E=\|x\pm y\|_{\nonsp}\leq 1$. Now since $\mathds{1}\overline\otimes x$ is extreme, $\mathds{1}\overline\otimes y=0$ and so $y=0$.

However, the extension of Theorem \ref{thm:ext} to an arbitrary von Neumann algebra can not be concluded if $\mu(x)$ is extreme and (ii) holds. The problem lies in the condition $n(x)\M n(x^*)=0$ which only implies that $n(\mathds{1}\overline\otimes x)(\Complex \mathds{1}\otimes\M)n(\mathds{1}\overline\otimes x^*)=0$ but not $n(\mathds{1}\overline\otimes x)\mathcal{A}n(\mathds{1}\overline\otimes x^*)=0$.
\end{remark}

Let us mention below other equivalent conditions to $n(x)\M n(x^*)=0$.  The \textit{center} $Z(\M)$ of the von Neumann algebra $\M$ is defined as
\[
Z(\M)=\{x\in \M:\, xy=yx \ \text{for all}\ \  y\in \M\},
\]
and for $x\in \M$ the central support projection is $z(x)=\inf\{p\in P(Z(\M)):\,x=xp\}$, where $P(Z(\M))$ is a family of orthogonal projections on $Z(\M)$.

 The projections $p$ and $q$ are said to be \emph{equivalent} (relative to the von Neumann algebra $\mathcal{M}$) denoted by $p\sim q$, if there exists a partial isometry $v\in\mathcal{M}$ such that $p=v^*v$ and $q=vv^*$. 

\begin{lemma}\cite[Volume I, Chapter V, Lemma 1.7]{Takesaki}
For two projections $e_1$ and $e_2$ in $\M$, the following statements are  equivalent.
\begin{itemize}
\item[{(i)}] $z(e_1)$ and $z(e_2)$ are not orthogonal.
\item[{(ii)}] $e_1\M e_2\neq 0$.
\item[{(iii)}] There exist nonzero projections $p_1\leq e_1$ and $p_2\leq e_2$ in $\M$ such that $p_1\sim p_2$. 
\end{itemize}
\end{lemma}

Therefore the following conditions are equivalent.
\begin{itemize}
\item[{(i)}] $z(n(x))$ and $z(n(x^*))$ are  orthogonal.
\item[{(ii)}] $n(x)\M n(x^*)= 0$.
\item[{(iii)}] There do not exist nonzero projections $p_1\leq n(x)$ and $p_2\leq n(x^*)$ in $\M$ such that $p_1\sim p_2$. 
\end{itemize}

It is well known that $E=E_0$ whenever $E$ is  strictly convex  \cite[Lemma 3.16]{czer-kam2015}. Thus  Theorem \ref{thm:ext} implies the following global characterization of strict convexity.
\begin{corollary}
\label{cor:extglobal}
If  $E$ is  strictly convex then $E(\mathcal{M},\tau)$ is strictly convex. If in addition $\M$ is non-atomic, then   strict convexity of $\nonsp$ implies strict convexity of $E$.
\end{corollary}

By Theorem \ref{thm:ext} applied to the commutative von Neumann algebra $\M=L_\infty[0,\tauone)$, we get a characterization of extreme functions of $B_E$  in terms of their decreasing rearrangements (see Section  \ref{sec:symmfun}).

\begin{corollary}
\label{cor:extce}
The following conditions are equivalent.
\begin{itemize}
\item [{(i)}]  $f$ is an extreme point of $B_E$.
\item[{(ii)}] $\mu(f)$ is an extreme point of $B_E$ and $\abs{f}\geq \mu(\infty,f)$.
\end{itemize}
\end{corollary}

\section{Strongly extreme points and midpoint local uniform convexity}
 Given a normed space $(X, \|\cdot \|)$ we say that $x\in S_X$ is a \emph{strongly extreme point} of the unit ball  $B_X$, or $MLUR$ point of $B_X$ \cite{Lin}, if for any $\left\{y_n\right\}$, $\left\{z_n\right\}\subset B_X$, $\norm{2x-y_n-z_n}\rightarrow 0$ implies that $\norm{y_n-z_n}\rightarrow0$. Equivalently, $x\in S_X$ is a strongly extreme point if for any $\left\{y_n\right\}\subset X$, $\norm{x\pm y_n}\rightarrow 1$ implies $\norm{y_n}\rightarrow 0$. A Banach space $X$ is called \emph{midpoint locally uniformly convex} ($MLUR$) space, if every element from the unit sphere $S_X$ is a strongly extreme point. $MLUR$ spaces have characterizations in terms of approximate compactness. A normed space $X$ is a $MLUR$ space if and only if every closed ball in $X$ is an approximatively compact Chebyshev set \cite[Theorem 5.3.28]{Megg}.

\begin{proposition}\cite[Proposition 2.3]{czer-kam2010}, \cite[Proposition 56]{DDP2014}
\label{prop:oc}
 An operator $x\in \nonsp$ is order continuous element of $\nonsp$ whenever $\mu(x)$ is order continuous element of $E$.  If in addition  $\M$ is non-atomic, then if $x$ is order continuous element then so is $\mu(x)$. Therefore if $\M$ is non-atomic, $(\nonsp)_{\text{a}}=E_{\text{a}}\Mtau$.
\end{proposition}

In fact, using similar techniques as in \cite[Proposition 2.3]{czer-kam2010} the analogous result can be shown for a symmetric sequence space $E$ and a unitary matrix space $C_E$.
\begin{proposition}
\label{prop:occe}
Let $E$ be a symmetric sequence space. Then $S(x)\in E$ is order continuous if and only if $x\in C_E$ is order continuous.  Consequently $(C_E)_a = C_{E_a}$.
\end{proposition}
\begin{proof}
Let $S(x)$ be order continuous in $E$ and $0\downarrow x_n\leq |x|$, $\{x_n\}\subset C_E$. Then $\{s_k(x_n)\} = S(x_n)\leq S(x) = \{s_k(x)\}$ and by \cite[Lemma 3.5]{DDP4}, $s_k(x_n)\downarrow_n 0$ for all $k\in\mathbb N$. Hence $\|x_n\|_{C_E}=\|S(x_n)\|_E\to 0$, proving that $x$ is order continuous.

Conversely, suppose that $x \ge 0$ is an order continuous element in $C_E$.  Let  $0\downarrow a_n\leq S(x)$, where $\{a_n\}\subset E$. By Proposition \ref{prop:isomarazy}, there is a $*$-isomorphism $V:E\to C_E$ such that  $V(S(x))=x$. Since $*$-isomorphism also preserves  the order, $0\downarrow V(a_n)\leq V(S(x))= x$. In view of $x$ being order continuous, $\|a_n\|_E=\|V(a_n)\|_{C_E}\to 0$ and $S(x)$ is order continuous.

\end{proof}

\begin{theorem}\cite[Theorem 2.5]{czer-kam2010}
\label{thm:1mlur}
Let $E$ be fully symmetric, and $x$ be an order continuous element of $\nonsp$. If the singular value function $\mu(x)$ is a $MLUR$ point of $B_{E_0}$ then $x$ is a $MLUR$ point of $B_{E_0\Mtau}$.
\end{theorem}

If $E$ is a symmetric sequence space then we always assume that $E\subset c_0$, which means that $E=E_0$. Therefore as shown in the proof of \cite[Theorem 2.9]{czer-kam2010},  Proposition \ref{prop:occe} and Theorem \ref{thm:1mlur} imply the following.  

\begin{corollary}
\label{cor:1mlurce}
Let $E$ be a fully symmetric sequence space and $x$ be an order continuous element of $C_E$. If $S(x)$ is a $MLUR$ point of $B_{E}$ then $x$ is a $MLUR$ point of $B_{C_{E}}$.
\end{corollary}

Moreover, Theorem \ref{thm:1mlur} can be translated for the commutative von Neumann algebra $\M=L_\infty[0,\tauone)$ (see Section  \ref{sec:symmfun}).

\begin{corollary}
\label{cor:1mlurcomm} Let $E$ be a fully symmetric function space and $f$ be an order continuous element of $E$. If $\mu(f)$ is a $MLUR$ point of $B_{E_0}$ then $f$ is a $MLUR$ point of $B_{E_0}$.
\end{corollary}

\begin{theorem}\cite[Theorem 2.7]{czer-kam2010}
\label{thm:3}
Suppose that $\M$ is non-atomic with a $\sigma$-finite trace $\tau$. If $x$ is a $MLUR$ point of $B_{\nonsp}$ then $\mu(x)$ is a $MLUR$-point of $B_E$ and either
\begin{itemize}
\item[{(i)}] $\mu(\infty,x)=0$, or
\item[{(ii)}] $n(x)\mathcal{M}n(x^*)=0$ and $\abs{x}\geq\mu(\infty,x)s(x)$.
\end{itemize}
\end{theorem}

The above result can be easily translated to unitary matrix spaces. 

\begin{theorem}
\label{thm:3ce}
Let $E$ be a symmetric sequence space. If $x$ is a $MLUR$ point of $B_{C_E}$ then $S(x)$ is a $MLUR$ point in $E$.
\end{theorem}
\begin{proof}
Suppose $x$ is a $MLUR$ point of $B_{C_E}$ and $\|S(x)\pm a_n\|_E\to 1$ for  $\{a_n\}\subset E$.  By Proposition \ref{prop:isomarazy}, there is a linear isometry $V:E\to C_E$ such that  $V(S(x))=x$.  Hence $\|x\pm V(a_n)\|_{C_E}=\|V(S(x))\pm V(a_n)\|_{C_E} =\|V(S(x)\pm a_n)\|_{C_E} = \|S(x)\pm a_n\|_E\to 1$. Since $x$ is $MLUR$, $\|a_n\|_E=\|V(a_n)\|_{C_E}\to 0$, proving that $S(x)$ is $MLUR$. 
\end{proof}

For $\M=L_\infty[0,\tauone)$  by Theorem \ref{thm:3}  and \cite[Corollary 2.8]{czer-kam2010} we  conclude with the following result.
\begin{corollary}
Let $E$ be a fully symmetric function space and $f$ be an order continuous element  in $E$. If $f$ is a $MLUR$ point of $B_E$ then $\mu(f)$ is a $MLUR$ point of $B_E$ and $\abs{f}\geq \mu(\infty,f)$.
\end{corollary}

\begin{remark}\label{rem:MLUR} (1)  Let $F$  be a Banach function or sequence space.  Then every $MLUR$ space  $F$ is order continuous. Indeed, if $F$ is not order continuous then  by Theorem \ref{prop:infinity}, $F$ contains an isomorphic copy of $\ell_\infty$. However $\ell_\infty$ does not admit an equivalent $MLUR$ norm \cite[Thm 2.1.5]{Lin}, so $F$ can not be $MLUR$.

(2) If $E$ is an order continuous symmetric function  space then $E=E_0$. Indeed, if $E\ne E_0$ then we can construct $f\in E$ and a sequence $f_n$ such that $\mu(\infty, f_n) = \mu(\infty, f) > 0$,  $0\le f_n \le f$ and $f_n\downarrow 0$ a.e..  It follows that  $\|f_n\|_E = \|f\|_E > 0$ for all $n\in \mathbb{N}$, which contradicts order continuity of $E$. 

\end{remark}

By Remark \ref{rem:MLUR}, 
any  $MLUR$ space $E$ is order continuous and  $E=E_0$,  thus the following corollary summarizes Theorem \ref{thm:1mlur} and \ref{thm:3}.

\begin{corollary}
\label{cor:2mlur}
Suppose $\M$ has a $\sigma$-finite trace $\tau$.
\begin{itemize}
\item[(1)] Let $\M$ be a non-atomic, $E$ be fully symmetric and $x$ be an order continuous element of $\nonsp$. Then $\mu(x)$ is a $MLUR$ point of $B_{E_0}$ if and only if $x$ is a $MLUR$ point of $B_{E_0\Mtau}$.
\item[(2)] If the space $E$ is $MLUR$ then $\nonsp$ is a $MLUR$ space. If in addition $\M$ is non-atomic, then if $\nonsp$ is $MLUR$ then $E$ is $MLUR$ as well. 
\end{itemize}
\end{corollary}

Similarly, by Corollary \ref{cor:1mlurce} and Theorem \ref{thm:3ce} we have the following.
\begin{corollary}
\label{cor:3mlur}
Let $E$ be a symmetric sequence space.  
\begin{itemize}
\item[(1)] Let $E$ be fully symmetric and $x$ be an order continuous element of $C_E$. Then $S(x)$ is a $MLUR$ point of $B_{E}$ if and only if $x$ is a $MLUR$ point of $B_{C_{E}}$.
\item[(2)] The space $E$ is $MLUR$ if and only if $C_E$ is a $MLUR$ space.
\end{itemize}
\end{corollary}

\begin{problem} (i)  Generalize Theorem \ref{thm:1mlur} to the whole space $E$ instead of $E_0$.

(ii) Remove the assumption that $x$ is order continuous in Corollaries \ref{cor:2mlur} and \ref{cor:3mlur}.

(iii) Generalize Theorem 1 in \cite{S1992} to noncommutative spaces $E\Mtau$ and $C_E$. It presents equivalent conditions for strongly symmetric spaces to be $MLUR$.

\end{problem}

\section{$k$-extreme points and $k$-convexity}

 If $(X, \|\cdot \|)$ is a normed space then a point $x\in S_X$ is called \textit{$k$-extreme} of the unit ball $B_X$ if $x$ cannot be represented as an average of $k+1$,  $k\in \mathbb{N}$, linearly independent elements from the unit sphere $S_X$. Equivalently, $x$ is $k$-extreme whenever the condition $x=\frac{1}{(k+1)}\sum_{i=1}^{k+1}x_i$,  $x_i\in S_X$ for $i=1,2,\dots,k+1$, implies that $x_1, x_2,...,x_{k+1}$ are linearly dependent.
Moreover, if every element of the unit sphere $S_X$ is $k$-extreme, then $X$ is called \textit{$k$-convex}.  If $k=1$ then $1$-extreme point is an extreme point of the unit ball in $X$.

The notion of $k$-extreme points was explicitly  introduced in \cite{ZYD} and applied to theorem on uniqueness of Hahn-Banach extensions. More precisely, L. Zheng and Z. Ya-Dong  showed there that given at least $k+1$-dimensional normed linear space over the complex field, all bounded linear functionals defined on subspaces of $X$ have at most $k$-linearly
independent norm-preserving linear extensions to $X$ if and only if the conjugate space $X^*$ is $k$-convex.
In the paper \cite{BFLM} $k$-convexity and $k$-extreme points found interesting application in studying the structure of nested sequences of balls in Banach spaces.

Clearly, if $X$ is a normed space of  dimension at least $l$, where $l\geq k$, and $x\in S_X$ is a $k$-extreme point of $B_X$, then $x$ is $l$-extreme. Moreover, $1$-extreme points are just extreme points of $B_X$, and so $1$-convexity of $X$ means  strict convexity of $X$. 

The simple example below differentiates between $k$-extreme and $k+1$-extreme points.
\begin{example}
Given $k\in\mathbb{N}$, consider the $k+2$ dimensional space $\ell_1^{k+2}$, equipped with $\ell_1$ norm. The element $x=\left(\frac1{k+1},\frac 1{k+1},\dots,\frac 1{k+1},0\right)$ is a $k+1$-extreme point of $B_{\ell_1^{k+2}}$, but not $k$-extreme. 
\end{example}
 We wish to mention here that also the family of Orlicz sequence spaces exposes the difference between $k$-extreme and $k+1$-extreme points \cite{Chen}.

 We have shown in \cite{czer-kam2015} the following equivalent characterization of $k$-extreme points.
 \begin{proposition}\cite[Proposition 2.2]{czer-kam2015}
\label{lm:2}
Given a normed space $X$, an element $x\in S_X$ is $k$-extreme of $B_X$ if and only if whenever for the elements $u_i\in X$, $i=1,2,\dots, k$, the conditions $x+u_i\in B_X$ and $x-\sum_{i=1}^{k} u_i\in B_X$ imply that $u_1,u_2,\dots, u_k$ are linearly dependent.
\end{proposition}

The next two results  extend the J. Ryff's theorem on extreme points \cite{Ryff} to $k$-extreme points. 

\begin{theorem}\cite[Theorem 2.6]{czer-kam2015}
\label{thm:orbitmain}
Let $E$ be a symmetric Banach function space and $f\in S_E$. Suppose there exists a function $g\in S_E$ such that  $f\prec g$ and $\mu(f)\neq \mu(g)$. Then $\mu(f)$ cannot be a $k$-extreme point of $B_E$ for any $k=1,2,\dots$.
\end{theorem}

\begin{corollary}\cite[Corollary 2.7]{czer-kam2015}
\label{cor:orbitmain}
 Let $E$ be a symmetric Banach function space and $f\in S_E$. If $\mu(f)$ is a $k$-extreme point of $B_E$ then for all functions $g\in S_E$ with $f\prec g$, it holds that $\mu(f)=\mu(g)$.
\end{corollary}

It is important to observe that the same characterization of the $k$-extreme points is not valid for symmetric sequence spaces. 
Consider the points $x=(\frac12,\frac12,0)$ and $y=(1,0,0)$ in $\ell_1$. It is easy to verify that $x$ is a $2$-extreme point in $\ell_1$ with $x\prec y$. However $x=\mu(x)\neq \mu(y)=y$.

It is usually easier to show that certain geometric property of $x$ translates into $\mu(x)$, rather than the other way around. The proofs of those statements will rely on some versions of the isomorphism results included in Section \ref{sec:isom}. However, it is still a challenging task. Not for every operator $x$ we have  that the isomorphism $V$ maps $\mu(x)$ into $x$, as it is for unitary matrix spaces $C_E$. We will include a full proof of the next theorem to demonstrate possible techniques one has to apply to prove that $\mu(x)$ inherits the geometric property of $x$.

We need first the following preliminary result.

\begin{lemma}\cite[Lemma 3.2 and 3.3]{czer-kam2015}
\label{lm:noncom2}
Let $\M$ be non-atomic. If $x$ is a $k$-extreme point of the unit ball $B_{\nonsp}$ then $\abs{x}\geq \mu(\infty,x)s(x)$, and either  $\mu(\infty,x)=0$ or $n(x)\mathcal{M}n(x^*)=0$.
\end{lemma}

\begin{theorem}\cite[Theorem 3.5]{czer-kam2015}
\label{thm:noncom4}
Suppose that $\M$ is  non-atomic with a $\sigma$-finite trace $\tau$.  If $x$ is a $k$-extreme point of $B_{\nonsp}$ then $\mu(x)$ is a $k$-extreme point of $B_E$ and either
\par {\rm(i)} $\mu(\infty,x)=0$, or
\par {\rm(ii)} $n(x)\mathcal{M}n(x^*)=0$ and $\abs{x}\geq\mu(\infty,x)s(x)$.
\end{theorem}
\begin{proof}
Suppose that $x$ is a $k$-extreme point of the unit ball in $E(\mathcal{M},\tau)$. By Lemma \ref{lm:noncom2} conditions (i) or (ii) are satisfied. 

Let 
\begin{equation}
\label{eq:noncom4}
\mu(x)=\frac{1}{k+1}\sum_{i=1}^{k+1} f_i, \text{ where }f_i\in S_E, \,i=1,2,\dots, k+1.
\end{equation}
To prove that $\mu(x)$ is $k$-extreme we need to show that $f_1,f_2,\dots,f_{k+1}$ are linearly dependent.
Let 
\[
p=s(\abs{x}-\mu(\infty,x)s(x))=e^{\abs{x}}(\mu(\infty,x),\infty).
\]
By Corollary \ref{cor:isom}, there exist a projection $q\in \mathcal{P}(\M)$, a non-atomic commutative von Neumann subalgebra $\mathcal{N}\subset q\M q$ and a $*$-isomorphism $V$ acting from the $*$-algebra $S\left(\left[0,\tauone\right),m\right)$ into the $*$-algebra $S(\mathcal{N},\tau)$, such that 
\[ V\mu(x)=\abs{x}q\ \ \ \text{ and }\ \ \ \mu(V(f))=\mu(f)\ \ \text{ for all } f\in S\left([0,\tauone),m\right).
\]
Moreover, there are three choices of $q$: (1) $q=\one$ whenever $\tau(s(x))<\infty$, (2) $q=s(x)$ if $\tau(s(x))=\infty$ and $\tau(p)<\infty$, or (3) $q=p$ if $\tau(p)=\infty$.

Applying now isomorphism $V$ to the equation (\ref{eq:noncom4})  we obtain
\begin{equation}
\label{eq:noncom4.1}
\abs{x}q=V\mu(x)=\frac1{k+1}\sum_{i=1}^{k+1}V(f_i).
\end{equation}

Case (1). Let $\tau(s(x))<\infty$ and $q=\one$. Since $s(x)\sim s(x^*)$ and $\tau(s(x))<\infty$,  by \cite[Chapter 5, Proposition 1.38]{Takesaki} $n(x)\sim n(x^*)$. Then by \cite[Lemma 2.6]{czer-kam2010} there exists an isometry $w$ such that $x=w\abs{x}$. 
Therefore by (\ref{eq:noncom4.1}) we have 
\[
x=\frac1{k+1}\sum_{i=1}^{k+1}wV(f_i),
\]
and $wV(f_1),wV(f_2),\dots, wV(f_{k+1})$ are linearly dependent by the assumption that $x$ is $k$-extreme. Since $w$ and $V$ are isometries $f_1, f_2, \dots, f_{k+1}$ are linearly dependent.

Case (2). Suppose that $\tau(s(x))=\infty$, $\tau(p)<\infty$, and $q=s(x)$.
Let $x=u\abs{x}$ be the polar decomposition of $x$.  By (\ref{eq:noncom4.1})
\[
 x=\frac1{k+1}\sum_{i=1}^{k+1}uV(f_i),
\]
where $uV(f_i)\in B_{\nonsp}$, $i=1,2,\dots, k+1$.
Since $x$ is $k$-extreme there exist constants $C_1,C_2,\dots, C_{k+1}$, such that $\sum_{i=1}^{k+1}C_i\neq 0$ and $\sum_{i=1}^{k+1}C_i uV(f_i)=0$. However $q=s(x)$ is an identity in the von Neumann algebra $\mathcal{N}\subset s(x)\M s(x)$ and so $u^*uV(f_i)=s(x)V(f_i)=V(f_i)$. Consequently, 
\[
\sum_{i=1}^{k+1}C_i V(f_i)=0
\]
and since $V$ is injective $f_1,f_2,\dots, f_{k+1}$ are linearly dependent.

Case (3). Consider now the case when $q=p=e^{\abs{x}}(\mu(\infty,x),\infty)$ and $\tau(p)=\infty$. By \cite[Lemma 3.4]{czer-kam2015}, if $\mu(\infty, x)>0$ then $|x|\geq \mu(\infty,x)s(x)$ is equivalent with $e^{\abs{x}}(0,\mu(\infty, x))=0$. Hence $q^{\perp}=e^{\abs{x}}\{0\}+e^{\abs{x}}\{\mu(\infty,x)\}\geq e^{\abs{x}}\{\mu(\infty,x)\}$.

For each $i=1,2,\dots, k+1$, choose $0\leq \alpha_i\leq \mu(\infty,f_i)$ such that $\frac1{k+1}\sum_{i=1}^{k+1}\alpha_i=\mu(\infty,x)$. Such constants $\alpha_i$ exist, since by (\ref{eq:noncom4}) and by Lemma \ref{lm:singfun} (4) for all $t>0$,
\[
\mu(t, x)=\mu\left(t,\frac1{k+1}\sum_{i=1}^{k+1}fi\right)\leq \frac1{k+1}\sum_{i=1}^{k+1}\mu\left(\frac t{k+1}, f_i\right), 
\]
and so $\mu(\infty,x)\leq\frac1{k+1}\sum_{i=1}^{k+1}\mu(\infty, f_i)$.

Define operators $x_i=V(f_i)+\alpha_i e^{\abs{x}}\{\mu(\infty,x)\}$. Observe that since $q$ is an identity in $\mathcal{N}$, $q^{\perp}V(f_i)=V(f_i)q^{\perp}=0$, and so $e^{\abs{x}}\{\mu(\infty,x)\}V(f_i)=
V(f_i)e^{\abs{x}}\{\mu(\infty,x)\}=0$. Furthermore $\alpha_i\leq \mu(\infty,f_i)=\mu(\infty,V(f_i))$, and hence by Lemma \ref{lm:singfun} (8), $\mu(x_i)=\mu(V(f_i))=\mu(f_i)$. Hence $x_i\in B_{\nonsp}$ for all $i=1,2,\dots, k+1$. We have now by (\ref{eq:noncom4.1}) that 
\begin{align*}
\abs{x}&=\abs{x}q+\abs{x}e^{\abs{x}}\{\mu(\infty,x)\}
=\abs{x}q+\mu(\infty,x)e^{\abs{x}}\{\mu(\infty,x)\}\\
&=\frac1{k+1}\sum_{i=1}^{k+1}V(f_i)
+\frac1{k+1}\sum_{i=1}^{k+1}\alpha_ie^{\abs{x}}\{\mu(\infty,x)\}=\frac1{k+1}\sum_{i=1}^{k+1}x_i.
\end{align*}
Using the polar decomposition $x=u\abs{x}$, 
\[
x=\frac1{k+1}\sum_{i=1}^{k+1}ux_i=\frac1{k+1}\sum_{i=1}^{k+1}(uV(f_i)+\alpha_iue^{\abs{x}}\{\mu(\infty,x)\}),
\]
and $ux_1,ux_2,\dots, ux_{k+1}$ are linearly dependent. Since two components of $x_i$, $uV(f_i)$ and $\alpha_iue^{\abs{x}}\{\mu(\infty,x)\}$ have disjoint supports, $uV(f_1), uV(f_2), \dots, uV(f_{k+1})$ are linearly dependent. Moreover $q\leq s(x)$, and so $u^*uV(f_i)=s(x)V(f_i)=s(x)qV(f_i)=qV(f_i)=V(f_i)$.  Since  $V$ is an isometry, $f_1,f_2,\dots, f_{k+1}$ are linearly dependent.   
\end{proof}

The converse statement of Theorem \ref{thm:noncom4} is as follows.

\begin{theorem}\cite[Theorem 3.13]{czer-kam2015}
\label{thm:noncom5}
Suppose $\M$ is non-atomic with a $\sigma$-finite trace $\tau$ and $E$ is a strongly symmetric function space. An element  $x\in S_{\nonsp}$ is a $k$-extreme point of $B_{\nonsp}$ whenever $\mu(x)$ is a $k$-extreme point of $B_E$ and one of the following conditions holds. 
\begin{itemize}
\item[(i)] $\mu(\infty,x)=0$,
\item[(ii)] $n(x)\mathcal{M}n(x^*)=0$ and $\abs{x}\geq\mu(\infty,x)s(x)$.
\end{itemize}
\end{theorem}

Combining now the results of Theorems \ref{thm:noncom4} and \ref{thm:noncom5}, we give a complete characterization of $k$-extreme points in terms of their singular value functions, when $\M$ is a  non-atomic von Neumann algebra. For $k=1$ we obtain the well-known theorem on extreme points proved in \cite{CKS1992ext}.
\begin{theorem}\cite[Theorem 3.14]{czer-kam2015}
\label{thm:main}
Let $E$ be a strongly symmetric space on $[0,\tauone)$ and $\M$ be  non-atomic with a $\sigma$-finite trace $\tau$. An operator $x$ is a $k$-extreme point of $B_{\nonsp}$ if and only if $\mu(x)$ is a $k$-extreme point of $B_E$ and one of the following, not mutually exclusive, conditions holds.
\begin{itemize}
\item [{(i)}] $\mu(\infty,x)=0$,
\item[{(ii)}] $n(x)\mathcal{M}n(x^*)=0$ and $\abs{x}\geq\mu(\infty,x)s(x)$.
\end{itemize}
\end{theorem}
As explained in Section \ref{sec:symmfun}, by applying the above theorem to the commutative von Neumann algebra $\M=L_\infty[0,\tauone)$ we get a characterization of $k$ -extreme functions in terms of their decreasing rearrangement.

\begin{corollary}\cite[Corollary 3.15]{czer-kam2015}
\label{cor:2}
Let $E$ be a strongly symmetric function space  and $k\in\mathbb{N}$. The following conditions are equivalent.
\begin{itemize}
\item [{(i)}]  $f$ is a $k$-extreme point of $B_E$,
\item[{(ii)}] $\mu(f)$ is a $k$-extreme point of $B_E$ and $\abs{f}\geq \mu(\infty,f)$.
\end{itemize}
\end{corollary}

The next observation allows to relate $k$-convexity of $E$ and $\nonsp$. It shows that if $E$ is $k$-convex, then $\mu(\infty, f)=0$ for all $f\in E$ and so the condition $|f|\geq \mu(\infty, f)$ is satisfied trivially.

\begin{lemma}\cite[Lemma 3.16]{czer-kam2015}
\label{lm:ezero}
If $E$ is a $k$-convex symmetric function space then $E=E_0$.  
\end{lemma}

By similar reasoning as in Remark \ref{rm:removingnonatom}, if $\mu(x)$ is a $k$-extreme point of $B_E$ and $\mu(\infty,x)=0$ then by Theorem \ref{thm:noncom5}, $x$ is a $k$-extreme  point of $\nonsp$  for an arbitrary von Neumann algebra. Hence the following holds.

\begin{corollary}\cite[Corollary 3.17]{czer-kam2015}
\label{cor:global}
 If a symmetric space $E$ is $k$-convex then $E(\mathcal{M},\tau)$ is $k$-convex. If in addition $\M$ is non-atomic, then $k$-convexity of $\nonsp$ implies $k$-convexity of $E$.
\end{corollary}

As a consequence of Corollary \ref{cor:global} we  could characterize $k$-extreme points in the orbits of functions and in Marcinkiewicz spaces.

Letting $g\in L_1[0,\alpha)+L_{\infty}[0,\alpha)$, $0<\alpha\leq \infty$,  the \textit{orbit} \cite{Ryff}  of $g$ is the set 
\[
\Omega(g)=\{f\in L_1[0,\alpha)+L_{\infty}[0,\alpha):\, f\prec g\}.
\]
 Clearly the inequality $f\prec g$ is equivalent to 
\[
\|f\|_{M_G}:=\sup_{t>0}\frac{\int_0^t \mu(f)}{\int_0^t \mu(g)}\leq 1.
\]
Setting $G(t)=\int_0^t \mu(g)$, the \textit{Marcinkiewicz} space $M_G$ is the set of all $f\in L^0$ such that $\|f\|_{M_G}<\infty$ \cite{KP, KPS}.   The space $M_G$ equipped with the norm $\|\cdot \|_{M_G}$ is a strongly symmetric function space. Therefore the orbit $\Omega(g)$ is the unit ball $B_{M_G}$ in the space $M_G$.
\begin{theorem}\cite[Theorem 4.1]{czer-kam2015}
\label{thm:orbits}
Let $g\in L_1[0,\alpha)+L_{\infty}[0,\alpha)$ and $k\in\mathbb{N}$. Then the following are equivalent.
\begin{itemize}
\item[{(i)}] $f$ is an extreme point of $\Omega(g)$.
\item[{(ii)}] $f$ is a $k$-extreme point of $\Omega(g)$.
\item[{(iii)}] $\mu(f)$ is a $k$-extreme point of $\Omega(g)$ and $\abs{f}\geq \mu(\infty,f)$.
\item[{(iv)}] $\mu(f)=\mu(g)$ and $\abs{f}\geq \mu(\infty,f)$.
\end{itemize}

\end{theorem}

As an immediate consequence we get the following result, which generalizes the characterization of extreme points in Corollary \ref{cor:extce}. 

\begin{corollary}\cite[Corollary 4.2]{czer-kam2015}
Let $M_G$ be the Marcinkiewicz space and $k$ be any natural number. The function  $f$ is a $k$-extreme point of $B_{M_G}$ if and only if  $\mu(f)=\mu(g)$ and $|f|\geq \mu(\infty,f)$. Consequently $f$ is a $k$-extreme point of $B_{M_G}$ if and only if $f$ is an extreme point of $B_{M_G}$.
\end{corollary}

\begin{problem} (i) Characterize $k$-extreme points of a unit ball in a sequence symmetric space $E$. Compare to Theorem \ref{thm:orbitmain}, Corollary \ref{cor:orbitmain} and the afterward comments.

(ii) Characterize $k$-extreme points of the ball in $C_E$.

(iii) Characterize $k$-extreme points for $k=2,3,\dots$ in a sequence Marcinkiewicz space and in the corresponding orbits. Extreme points in that space has been characterized in \cite{KLL}.

\end{problem}

\section{Complex extreme points and convex convexity}

 For a normed complex space $(X, \|\cdot\|)$ a point $x$ of $S_X$ is said to be a \emph{complex extreme} point of the unit ball $B_X$
 if  for every $\lambda\in \Complex$, $|\lambda|\leq
1$ and $y$ in $X$, whenever $x+\lambda y\in B_X$ then $y=0$ \cite{TW}.
Equivalently, $x$ is a complex extreme point for $B_X$ whenever
$x\pm y,  x\pm i y\in B_X$, $y\in B_X$, then $y=0$. The space $X$ is
said to be \emph{complex strictly convex}  space,
if every element from the unit sphere $S_X$ is a complex extreme
point. Clearly an extreme point is a complex extreme point, and a strictly convex space is complex strictly convex.

 The concepts of complex extreme points and complex strictly convex spaces have been introduced by Thorp and Whitley in \cite{TW} in connection with the strong maximum modulus theorem of vector-valued analytic functions. Its liaison to holomorphic spaces has been further confirmed by   Globevnik's work in \cite{Glob} who investigated complex uniformly convex spaces and showed among others that peak points of the ball algebra over a Banach space $X$ are complex extreme points of its unit ball $B_X$. 

It was shown in \cite{HN, HanJu} that monotone properties of normed lattices are closely related to their complex convexity properties. Recall that an  ordered normed linear space $(X,\|\cdot\|)$ is \textit{strictly monotone}   if for every $x,y\in X$ with $0\leq x\leq y$ and $x\neq y$ it follows that $\|x\|< \|y\|$.  An element $x\in X$ is called \textit{upper monotone}, if for any $y\in X$ with $x\leq y$ and $x\neq y$ we have that $\|x\|<\|y\|$. 
  For instance  complex strict convexity of $E$ is equivalent  to strict monotonicity of $E$ \cite[Corollary 1]{HN}.  Moreover, an element $f$ of $E$ is complex strictly convex  if and only if $|f|$ is an upper monotone point in $E$  \cite[Theorem 1]{HN}.

 Complex extreme points of noncommutative symmetric spaces were only studied in \cite{czer-kam2010}. The characterization of the complex extreme points is analogous to the results on extreme points in \cite{CKS1992ext}.  The relation  between complex extreme and upper monotone points played an important role in proving that $x$ inherits complex convexity from $\mu(x)$. We  observed in \cite{CKS1992ext} that if $\mu(x)$ is a complex extreme point of $B_E$  then the  functions from $B_E$ whose decreasing rearrangements majorize $\mu(x)$ must be equimeasurable with $\mu(x)$. 
\begin{lemma}
\label{lm:cextum}
Let $x,y\in B_{\nonsp}$ and let $\mu(t,x)\leq\mu(t,y)$ for all $t\in[0,\infty)$. If there exists $t_0>0$ such that $\displaystyle \mu(t_0,x)<\mu(t_0,y)$ then $\mu(x)$ is not complex extreme point of ${B}_E$.
\end{lemma}
As a consequence of the above lemma we have the following.
\begin{lemma}
\label{lm:3} Let $x\in S(\mathcal{M},\tau)$ and
$x\geq\mu(\infty,x)\one$. If $\mu(x)$ is a  complex extreme point of $B_E$ then $x$ is a complex extreme point of $B_{E(\mathcal{M},\tau)}$.
\end{lemma}
\begin{proof}
Let $x\in S(\mathcal{M},\tau)$, $x\geq\mu(\infty,x)\one$ and
$\mu(x)$ be a complex extreme point of $B_E$. Suppose that $x\pm y$, $x\pm
iy$ belong to ${B}_{E(\mathcal{M},\tau)}$, for some $y\in
{B}_{E(\mathcal{M},\tau)}$. Without loss of generality it can be assumed  that $y$ is a self-adjoint
operator \cite[Lemma 3.2]{czer-kam2010}. Now by \cite[Proposition 3]{MR979385}, for all $t>0$,
\begin{center}
$\mu(t,x)\leq\mu(t,x+iy)$.
\end{center}
Since $\mu(x)$ is a complex extreme point of $B_E$ and $\mu(x+iy)\in{B}_{E}$, by Lemma \ref{lm:cextum} it follows that for all $t>0$,
\begin{center}
$\mu(t,x)=\mu(t,x+iy)$.
\end{center}
Then \cite[Proposition 3.5]{CKS1992ext}  implies that $y=0$, and the claim follows.
\end{proof}
 A substantial effort was still required to expand this result to the broader class of operators satisfying conditions (i) and (ii) below.

 \begin{theorem}\cite[Theorem 3.7]{czer-kam2010}
\label{prop:1}
An element  $x\in S_{\nonsp}$ is a complex extreme point of $B_{\nonsp}$ whenever $\mu(x)$ is a complex extreme point of $B_E$ and one of the following conditions holds.
\begin{itemize}
\item[(i)] $\mu(\infty,x)=0$,
\item[(ii)] $n(x)\mathcal{M}n(x^*)=0$ and $\abs{x}\geq\mu(\infty,x)s(x)$.
\end{itemize}
\end{theorem}

\begin{theorem}\cite[Theorem 3.10]{czer-kam2010}
\label{thm:xtomux}
Suppose that $\M$ is non-atomic with a  $\sigma$-finite trace $\tau$.  If $x$ is a complex extreme point of $B_{\nonsp}$ then $\mu(x)$ is a  complex extreme point of ${B}_E$ and either
\begin{itemize}
\item[{(i)}] $\mu(\infty,x)=0$, or
\item[{(ii)}] $n(x)\mathcal{M}n(x^*)=0$ and $\abs{x}\geq\mu(\infty,x)s(x)$.
\end{itemize}
\end{theorem}

We summarize this chapter with complete characterization of complex extreme points in $B_{\nonsp}$. The first result is an immediate consequence of Theorems \ref{prop:1} and \ref{thm:xtomux}.
\begin{theorem}\cite[Theorem 3.11]{czer-kam2010}
\label{thmst:complexext}
Let $\M$ be  non-atomic with a $\sigma$-finite trace $\tau$. An operator $x$ is a complex extreme point of  ${B}_{\nonsp}$ if and only if $\mu(x)$ is a complex extreme point of ${B}_E$ and one of the following, not mutually exclusive, conditions holds.
\par {\rm(i)} $\mu(\infty,x)=0$,
\par {\rm(ii)} $n(x)\mathcal{M}n(x^*)=0$ and $\abs{x}\geq\mu(\infty,x)s(x)$.
\end{theorem}

Although Theorem \ref{thmst:complexext} requires $\M$ to be non-atomic, in fact it relies on the existence of the isomorphism $V$ for which $V(\mu(x))=x$. Since such isometry exists also for unitary matrix spaces by Proposition \ref{prop:isomarazy}, the following can be observed.
\begin{theorem}\cite[Theorem 3.13]{czer-kam2010}
 Let $E$ be a symmetric sequence space. Then $x$ is a complex extreme point  of $B_{C_E}$ if and only if $S(x)$ is a complex extreme point of $B_E$.
\end{theorem}

By Theorem \ref{thmst:complexext} applied to the commutative von Neumann algebra $\M=L_\infty[0,\tauone)$ we get a characterization of complex extreme points of $B_E$ in terms of their decreasing rearrangements.

\begin{corollary}
\label{cor:4}
 The following conditions are equivalent.
\begin{itemize}
\item [{(i)}]  $f$ is a complex extreme point of $B_E$.
\item[{(ii)}] $\mu(f)$ is a  complex extreme point of $B_E$ and $\abs{f}\geq \mu(\infty,f)$.
\end{itemize}
\end{corollary}

The next lemma will be useful in relating complex convexity of $E$ and $\nonsp$.
\begin{lemma}
\label{lmst:sm0}
If  $E$ is strictly monotone then $E=E_0$.
\end{lemma}
\begin{proof}
Suppose that $E\neq E_0$. Hence there exists a function $f\in E$ such that $\mu(\infty,f)>0$ and $m((\supp{f})^c)=m\{t:f(t)=0\}> 0$. Then 
\[
\abs{f}+\mu(\infty,f)\chi_{(\supp{f})^c}\geq \abs{f}\text{ and }\abs{f}+\mu(\infty,f)\chi_{(\supp{f})^c}\neq \abs{f}.
\]
Since $\mu(\abs{f}+\mu(\infty,f)\chi_{(\supp{f})^c})=\mu(f)$, we have that 
\[
\norme{\abs{f}+\mu(\infty,f)\chi_{(\supp{f})^c}}=\norme{f},
\]
and so  $E$ is not strictly monotone.
\end{proof}
\begin{corollary}
\label{corst:cr}
Let $\M$ be non-atomic.  A symmetric space $E$ is complex strictly convex if and only if $E(\mathcal{M},\tau)$ is complex strictly convex.
\end{corollary}
\begin{proof}
If $E$ is complex strictly convex, then $E$ is strictly monotone \cite[Corollary 1]{HN}. Therefore by Lemma \ref{lmst:sm0}, $E=E_0$ and consequently Theorem \ref{prop:1} implies that  $E\Mtau$ is complex strictly convex.

 Suppose now that $E\Mtau$ is  complex strictly convex. It follows that $E(\M_p,\tau_p)$ is complex strictly convex for any projection $p\in P(\M)$. Let $p\in P(\M)$ be a $\sigma$-finite projection with $\tau(p)=\tauone$. By Proposition \ref{isom2}, $E$ is isometrically embedded into $E(\M_{p},\tau_{p})$, and therefore $E$ inherits from it the complex  strict convexity.
\end{proof}
The analogous result follows for unitary matrix spaces $C_E$.
\begin{theorem}\cite[Theorem 3.13]{czer-kam2010}
\label{thm:cextremece}
 Let $E$ be a symmetric sequence space.  Then $C_E$ is complex strictly convex if and only if $E$ is complex  strictly convex.
\end{theorem}

The next theorem relates strict monotonicity of  $E$ and $\nonsp$.
\begin{theorem}\cite[Theorem 3.15]{doctth}
\label{lmst:sm}
Let  $\M$ be non-atomic.  Then  $E$ is strictly monotone if and only if   $\nonsp$ is strictly monotone.
\end{theorem}

As a consequence, we get a noncommutative version of Corollary 1 in \cite{HN}.
\begin{corollary}\cite[Corollary 3.16]{doctth}
 Let $\M$ be non-atomic.
$\nonsp$ is complex  strictly convex if and only if  $\nonsp$ is strictly monotone.
\end{corollary}

\section{Complex local uniform convexity}

In 2000, T. Wang and Y. Teng \cite{WT} defined $\Complex-LUR$ points and $\Complex-LUR$ spaces and obtained criteria for this property in the class of Musielak-Orlicz spaces of vector-valued functions.
A point $x\in S_X$,  where $(X,\|\cdot\|)$ is a complex normed space, is  a \emph{point of complex local uniform convexity} ($\Complex-LUR$ point) \cite{WT} if for every $\epsilon>0$ there exists $\delta(x,\epsilon)>0$ such that
 \[
\sup_{\lambda=\pm1,\pm  i}\norm{x+\lambda y}\geq 1+\delta(x,\epsilon)
\]
 for every $y\in X$ satisfying $\norm{y}\geq\epsilon$. Equivalently, $x$ is a $\Complex-LUR$ point whenever from $\|x+\lambda y_n\|\to 1$, $\{y_n\}\subset X$, $\lambda=\pm1, \pm i$ it follows that $\|y_n\|\to 0$. If every point of the unit sphere of $X$ is a $\Complex-LUR$ point, then $X$ is called a \emph{complex locally uniformly convex} ($\Complex-LUR$) space.

 It is clear that the real geometric properties such as uniform convexity, local uniform convexity and strict convexity imply their complex analogies, that is complex uniform convexity, complex local uniform convexity and complex strict convexity, respectively.

The next two theorems relate complex local uniform convexity of $\mu(x)\in E$ and $x\in\nonsp$.
 
 \begin{theorem}\cite[Theorem 4.1]{czer-kam2010}
\label{thm:2clur}
Let $E$ be strongly symmetric and $x$ be an order continuous element of $\nonsp$. If $\mu(x)$ is a $\Complex-LUR$ point of  $B_{E_0}$ then $x$ is a $\Complex-LUR$ point of ${B}_{E_0\Mtau}$.
\end{theorem}

\begin{theorem}\cite[Theorem 4.2]{czer-kam2010}
\label{thm:1clur}
Suppose that $\M$ is non-atomic and $\tau$ is $\sigma$-finite.
If $x$ is a $\Complex-LUR$ point in $B_{\nonsp}$ then $\mu(x)$ is a $\Complex-LUR$ point in $B_E$ and either
\begin{itemize}
\item[{(i)}] $\mu(\infty,x)=0$, or
\item[{(ii)}] $n(x)\mathcal{M}n(x^*)=0$ and $\abs{x}\geq\mu(\infty,x)s(x)$.
\end{itemize}
\end{theorem}

For the commutative von Neumann algebra $\M=L_\infty[0,\tauone)$ we get the following.
\begin{corollary}
Let $E$ be a strongly symmetric function space. The following conditions are equivalent.
\begin{itemize}
\item [{(i)}]  $f$ is a $\Complex-LUR$  point of $B_E$.
\item[{(ii)}] $\mu(f)$ is a  $\Complex-LUR$ point of $B_E$ and $\abs{f}\geq \mu(\infty,f)$.
\end{itemize}
\end{corollary}

Note that if $E$ is order continuous then $E= E_0$ (Remark \ref{rem:MLUR}), and the norm on $E$ is strongly symmetric \cite[Proposition 2.6]{CSweak}. Hence by Theorems \ref{thm:2clur} and \ref{thm:1clur} we can conclude the following.

\begin{corollary}\cite[Corollary 4.3]{czer-kam2010}
Let $E$ be  order continuous, and $\M$ have a $\sigma$-finite trace $\tau$. If $E$ is a $\Complex-LUR$ space then $\nonsp$ is a $\Complex-LUR$ space. If in addition $\M$ is  non-atomic and $\nonsp$ is $\Complex-LUR$ then $E$ is $\Complex-LUR$ as well.
\end{corollary}

\begin{theorem}\cite[Theorem 4.5]{czer-kam2010}
 Let $E$ be an order continuous  symmetric sequence space. Then $C_E$ is a $\Complex-LUR$ space if and only if $E$ is a $\Complex-LUR$ space.
 \end{theorem}

Let us discuss here the notions of complex strongly extreme ($\Complex-MLUR$) points  
and complex midpoint locally uniformly rotund ($\Complex-MLUR$) spaces.

It was demonstrated in \cite{czer-kam2010} that the notions of $\Complex-LUR$ and $\Complex-MLUR$ points, and hence the notions of $\Complex-LUR$ and $\Complex-MLUR$ spaces, are equivalent in any complex normed space. Consequently, in complex  normed spaces these complex properties are not distinguishable contrary to their corresponding "real" properties $LUR$ and $MLUR$ \cite{Lin}.

The modulus of complex strong extremality was defined in \cite{CCH} analogously as the modulus
of strong extremality in the real case, introduced by C. Finet in
\cite{Fin}.
Let $(X,\|\cdot\|)$ be a
 normed space over the field of complex numbers. For $x\in S_X$ and
$\epsilon>0$, the \emph{modulus of complex strong extremality} at
$x$ is the number
  \[
  \displaystyle \Delta(x,\epsilon)=\inf\left\{1-\abs{\lambda}:\hspace{4mm} \exists y,\|y\|>\epsilon,\quad \|\lambda x\pm y\|\leq1\text{  and  }\|\lambda ix\pm y\|\leq 1\right\}.
  \]
The element $x\in S_X$ is said to be a \emph{$\Complex-MLUR$} point in $B_X$, or \emph{complex strongly extreme} point of the unit ball $B_X$, if for any $\epsilon>0$, the modulus of complex extremality $\Delta(x,\epsilon)>0$. A  normed space $X$ is said to be \emph{complex midpoint locally uniformly rotund} or $\Complex-MLUR$ space, if every element from the unit sphere $S_X$ is a $\Complex-MLUR$ point.

The following  equivalent definition of $\Complex-MLUR$ points leads to the proof of equivalence of $\Complex-LUR$ and $\Complex-MLUR$ notions.
\begin{lemma}\cite[Lemma 5.1]{czer-kam2010}
\label{complexdef}
 An element $x\in S_X$ is a $\Complex-MLUR$ point of $B_X$ if and only if for any $\{x_n\}\subset X$, $\lambda=\pm 1,\pm i$, $\|x+\lambda x_n\|\to 1$ implies that $\|x_n\|\to 0$.
\end{lemma}
\begin{proof}
Suppose that $x\in S_X$ is a $\Complex-MLUR$ point, that is for all $\epsilon>0$, the modulus $\Delta(x,\epsilon)>0$.
Let $\|x\pm x_n\|\to 1$ and $\|x\pm ix_n\|\to 1$, where $\{x_n\}\subset X$. Set
\[
c_n=\max_{\lambda\in\{\pm1,\pm i\}}\|x+\lambda x_n\|.
\] Clearly, $c_n\to 1$. If for some $n$, $c_n\leq1$ then $\|x+\lambda x_n\|\leq 1$ for all $\lambda=\pm1,\pm i$, and consequently $x_n=0$. Indeed, suppose that $x_n\neq0$. Hence, there exists an $\epsilon>0$ such that $\|x_n\|>\epsilon$, $\|x\pm x_n\|\leq 1$ and $\|ix\pm x_n\|\leq 1$. But then $\Delta(x,\epsilon)=0$, which leads to a contradiction. Therefore without lost of generality, we can assume that $c_n>1$ for all $n\in\mathbb{N}$.  Clearly, for all $n\in\mathbb{N}$,
\[
\|c_n^{-1}x\pm c_n^{-1} x_n\|\leq 1\text{\quad and\quad } \|ic_n^{-1}x\pm c_n^{-1} x_n\|\leq 1.
\]
Denote $\lambda_n=c_n^{-1}$, $n\in\mathbb{N}$. Then for each $\lambda_n$ there exists an element $y_n=c_n^{-1}x_n$ such that $\|\lambda_nx\pm y_n\|\leq 1$ and $\|i\lambda_nx\pm y_n\|\leq1$. Hence $\|y_n\|\to 0$ and consequently $\|x_n\|\to 0$. If not, then there exist $\epsilon>0$ and a subsequence $y_{n_k}$ such that $\|y_{n_k}\|>\epsilon$, and since $\lambda_{n_k}\to 1$, 
\[
0\le \Delta(x,\epsilon)\le 1- |\lambda_{n_k}| \to 0, \ \ \ k\to\infty,
\]
which leads to  $\Delta(x,\epsilon) = 0$, a contradiction with the assumption.

To prove the  reverse implication, assume that $\Delta(x,\epsilon)=0$ for some $\epsilon>0$. Therefore there exists a sequence $\{\lambda_n\}\subset\Complex$ satisfying $\abs{\lambda_n}\uparrow 1$ and for each $n\in\mathbb{N}$, there is  $x_n\in\emph{B}_X$, $\|x_n\|\geq \epsilon$ such that $\|\lambda_n x\pm x_n\|\leq 1$ and $\|i\lambda_n x\pm x_n\|\leq 1$.
Therefore, for all $n\in\mathbb{N}$ we have
\[
\|x\pm \lambda_n^{-1}x_n\|\leq \abs{\lambda_n}^{-1}\quad\text{and}\quad \|x\pm i \lambda_n^{-1}x_n\|\leq\abs{\lambda_n}^{-1},
\]
and since $\abs{\lambda_n}\to1$, $\overline\lim_n\|x\pm \lambda_n^{-1}x_n\|\leq 1$ and $\overline\lim_n\|x\pm i\lambda_n^{-1}x_n\|\leq 1$.

By $2=2\|x\|\leq \|x+\lambda_n^{-1}x_n\|+\|x-\lambda_n^{-1}x_n\|$ it follows  that $\overline\lim_n\|x\pm \lambda_n^{-1}x_n\|= 1$. Moreover, $2-\|x-\lambda_n^{-1} x_n\|\leq \|x+\lambda_n^{-1} x_n\|$, and so $2-\underline\lim_n\|x-\lambda_n^{-1} x_n\|=\overline \lim_n (2-\|x-\lambda_n^{-1} x_n\|)\leq \overline\lim_n\|x+\lambda_n^{-1} x_n\|=1$. Hence $1\leq \underline\lim_n\|x-\lambda_n^{-1} x_n\|\leq \overline\lim_n\|x-\lambda_n^{-1} x_n\|=1$ and so $\lim_n\|x-\lambda_n^{-1} x_n\|=1$. Similarly one can show that $\lim_n\|x+\lambda_n^{-1} x_n\|=1$ and $\lim_n\|x\pm i\lambda_n^{-1} x_n\|=1$.
 Hence there exists a subsequence $\displaystyle \{\lambda_{n_k}^{-1}x_{n_k}\}$ such that 
 \[
 \lim_k\|\lambda_{n_k}^{-1}x_{n_k}\|\neq0, \ \ \ \lim_k\|x\pm \lambda_{n_k}^{-1}x_{n_k}\|=1, \ \  \lim_k\|x\pm i\lambda_{n_k}^{-1}x_{n_k}\|=1,
 \]
 which completes the proof.
\end{proof}
Now we can state the equivalence result of $\Complex-LUR$ and $\Complex-MLUR$ properties.
\begin{proposition}\cite[Proposition 5.2]{czer-kam2010}
Let $(X, \|\cdot\|)$ be a normed space and $x\in S_X$. The following conditions are equivalent.
\begin{itemize}
\item[{(i)}] An element $x\in S_X$ is a $\Complex$-LUR point of $B_X$.
\item[{(ii)}] For all $\{y_n\}\subset X$,  $\sup_{\lambda=\pm1,\pm i}\|x+\lambda y_n\|\to1$ implies $\|y_n\|\to0$.
\item[{(iii)}] For all $\{y_n\}\subset X$,  $\|x\pm y_n\|\to 1$ and $\|x\pm i y_n\|\to 1$ implies $\|y_n\|\to 0$.
\item[(iv)] An element $x\in S_X$ is a $\Complex-MLUR$ point of $B_X$.
\end{itemize}
\end{proposition}
\begin{proof}
Let $x\in S_X$. It is clear that (i) and (ii) are equivalent and (ii) implies (iii). By Lemma \ref{complexdef}, conditions (iii) and (iv) are also equivalent. It remains to show implication from (iii) to (ii).

Suppose that $\displaystyle \sup_{\lambda=\pm 1,\pm i}\|x+\lambda y_n\|\to 1$, $\{y_n\}\subset X$. Then $\overline{\lim}_n\|x\pm y_n\|\leq 1$ and $\overline{\lim}_n\|x\pm iy_n\|\leq 1$. Similarly as in the last paragraph of the proof of  Lemma \ref{complexdef} we can show that  for all $\lambda=\pm 1,\pm i$, we have $\lim_n\|x+\lambda y_n\|=1$. Hence by (iii), $\|y_n\|\to 0$.
\end{proof}

\begin{corollary} A normed space $X$ is $\Complex-LUR$ if and only if it is $\Complex-MLUR$.

\end{corollary}

\section{$p$-convexity and $q$-concavity} 
In \cite{AL1985}, J. Arazy and in \cite{DDS3}, P. Dodds, T. Dodds and F. Sukochev have characterized  $p$-convexity (concavity) and  lower- (upper) $p$-estimate of $C_E$ and  of $\nonsp$, respectively. Those studies have been  performed  in the case  when  $E$ is a quasi-normed symmetric space. Recall that the real valued  functional $\|\cdot\|$ on a complex vector space $X$ is a quasi-norm if it satisfy the following conditions: (1) $\|x\|=0$ if and only if $x=0$; (2) $\|\lambda x\| = |\lambda| \|x\|$ for all $x\in X$, $\lambda\in \Complex$; (3) there exists $C>0$ such that $\|x + y\| \le C(\|x\| + \|y\|)$ for all $x,y \in X$.  The space $X$ equipped with a quasi-norm $\|\cdot\|$ is called a quasi-normed space, and when it is complete then it is called a quasi-Banach space. 

A quasi-normed space $F=F(I)\subset L^0(I)$, where either  $I= [0,\alpha)$, $0< \alpha \le \infty$, or $I=\mathbb{N}$, with the quasi-norm $\normf{\cdot}$ satisfying the condition that $f\in F$ and $\normf{f}\leq\normf{g}$ whenever $0\leq f\leq g$, $f\in L^0(I)$ and $g\in F$, is a \emph{quasi-normed function, or sequence space}, respectively.  A quasi-normed function or sequence space $E\subset L^0$ is called \emph{a  quasi-normed symmetric space}  if it follows from $f\in L^0$, $g\in E$ and $\mu(f)\leq \mu(g)$ that $f\in E$ and $\norme{f}\leq\norme{g}$. If $E$ is complete then it is called a \emph{quasi-Banach symmetric space}. The notions of the Fatou property of $E$ or order continuity of $f\in E$ are defined analogously as in the case of Banach symmetric spaces.

   Given a quasi-normed  symmetric space $E$, the space $E\Mtau$ of measurable operators defined  analogously as for a normed space $E$, that is $E\Mtau=\{x\in S\Mtau:\quad \mu(x)\in E\}$  and 
$\|x\|_{\nonsp}=\norme{\mu(x)}$, is a quasi-normed space, and if $E$ is complete then $E\Mtau$ is also complete \cite{S}.  The space $E\Mtau$ is an ideal with respect to natural order. In fact if $0\le x\le y$, $x\in S(\mathcal{M},\tau)$, and $y\in E\Mtau$ then $x\in E\Mtau$ and $\|x\|_{E\Mtau}\le \|y\|_{E\Mtau}$.  However it is not a lattice in the sense that for given two operators $x$ and $y$ their minimum or maximum may not exist. Despite this the definitions of order convexity or concavity, and to some limited cases upper or lower estimates are extended to these spaces   in the analogous way.

Let $X\subset \mathcal{S}\Mtau$ be a quasi-normed space with quasi-norm $\|\cdot\|_X$. It is called symmetric if for any $x\in X$, $y\in \mathcal{S}\Mtau$ with $\mu(y)\le\mu(x)$ we have that $y\in X$ and $\|y\|_X \le \|x\|_X$. In particular if $E\subset L^0$ is a quasi-normed symmetric space, then $E\Mtau$ is  a quasi-normed symmetric  space of measurable operators. 
We also have the opposite relation, if  $\M$ is a non-atomic von Neumann algebra, then for every
symmetric space $(X, \|\cdot\|_X)\subset \mathcal{S}(\M, \tau)$ there exists a symmetric function space $(E, \|\cdot\|_E)$ on
$[0,\tauone)$ such that $X = \nonsp$ and $\|x\|_X=\|x\|_{\nonsp}$ for every $x\in X$ \cite{CKS1992, CSweak}.

 Let $X\subset \mathcal{S}\Mtau$ be a quasi-normed symmetric space.  Given $x_i\in X$, $i=1,2,\dots,n$,  $0<p<\infty$, the element $\left(\sum_{i=1}^n |x_i|^p \right)^{\frac{1}{p}}$ is well defined by functional calculus.

For operators $x_i$, the expression $\left(\sum_{i=1}^n\abs{x_i}^p\right)^{1/p}$ fails the monotonicity and convexity properties enjoyed by the analogous expressions in quasi-normed lattices. It is well known that $|\cdot|$ does not satisfy the triangle inequality for operators.    Neither $p\mapsto \tr {(a^p+b^p)^{1/p}}$ nor $p\mapsto (a^p+b^p)^{1/p}$, for two positive operators $a,b\in B(\ell_2)$, need to be monotone \cite{AL1985}. Despite of this the  quasi-norms $\left\|\left(\sum_{j=1}^n\abs{x_j}^p\right)^{1/p}\right\|_{X}$ behave in a much better way 
and can be studied via majorization inequalities between the sequences $\left(\sum_{j=1}^n\abs{x_j}^q\right)^{1/q}$ and $\left(\sum_{j=1}^n \mu(x_j)^p\right)^{1/p}$ for  $0< p,q<\infty$.

  Let $0< p,q< \infty$ and assume that 
$\left(\sum_{i=1}^n |x_i|^p \right)^{\frac{1}{p}}\in X$ if $x_i\in X$.  A  quasi-normed symmetric space $X\subset \mathcal{S}\Mtau$  is said to be $p$-\textit{convex}, $0
< p < \infty$, respectively $q$-\textit{concave}, $0<q<\infty$, if there
is a constant $M > 0$ such that
\begin{equation}
\label{eq1}
\left\|\left(\sum_{i=1}^n |x_i|^p \right)^{\frac{1}{p}}\right\|_X \leq M
\left(\sum_{i=1}^n \|x_i\|_X ^p \right)^{\frac{1}{p}},
\end{equation}
respectively,
\begin{equation}
\label{eq2}
\left(\sum_{i=1}^n \|x_i\|_X^q \right)^{\frac{1}{q}} \leq M
\left\|\left(\sum_{i=1}^n |x_i|^q \right)^{\frac{1}{q}}\right\|_X
\end{equation}
for every choice of vectors $x_1, \ldots, x_n \in X$. 
We set $M^{(p)}(X)$ to be the smallest constant $M$ in (\ref{eq1}), and we call it a $p$-convexity constant of $X$.
Similarly, $M_{(q)}(X)$, called $q$-concavity constant of $X$, will denote the smallest constant in (\ref{eq2}). 

Note that for $x_i \in E\Mtau$, $i=1,2,\dots,n$,  where $E$ is a quasi-normed symmetric space, we have $\left(\sum_{i=1}^n\abs{x_i}^p\right)^{1/p}\in E\Mtau$ \cite[Lemma 2.1]{DDS3}.
Indeed setting $|y| = \sum_{i=1}^n |x_i|^p$ we have  by Lemma \ref{lm:singfun} (5) and (4),
\[
 \mu\left(t,\left(\sum_{i=1}^n |x_i|^p\right)^{\frac1p}\right)=\left( \mu\left(t,\sum_{i=1}^n |x_i|^p\right)\right)^{\frac1p},\ \
 \mu\left(t,\sum_{i=1}^n |x_i|^p\right)\leq \sum_{i=1}^n \mu\left(\frac{t}{n},\sum_{i=1}^n |x_i|^p\right),
 \] 
and it follows 
  \begin{align*}
 \mu\left(t,\left(\sum_{i=1}^n |x_i|^p\right)^{\frac1p}\right)
 \leq \left(\sum_{i=1}^n \mu\left(\frac tn, |x_i|^p\right)\right)^{\frac1p}=\left(\sum_{i=1}^n \mu^p\left(\frac tn, |x_i|\right)\right)^{\frac1p}.
 \end{align*} 
Now since  the dilation operator  is bounded \cite[Lemma 1.4]{KR2009} on $E$, and  $\mu(|x_i|)\in E$, $ \mu\left(\frac tn, |x_i|\right)\in E$.  By functional calculus for $E$ \cite{LT2} we have that  $\left(\sum_{i=1}^n \mu^p\left(\frac tn, |x_i|\right)\right)^{\frac1p}\in E$, and by the above inequality,  $\mu\left(t,\left(\sum_{i=1}^n |x_i|^p\right)^{\frac1p}\right)\in E$, and so $\left(\sum_{i=1}^n |x_i|^p\right)^{\frac1p}\in \nonsp$.

It is easy to check that $L^p\Mtau$ is $p$-convex and $p$-concave with $M^{(p)}(L^p\Mtau) = M_{(p)}(L^p\Mtau) = 1$.

For any quasi-normed symmetric space $X\subset \mathcal{S}\Mtau$ and $0<p<\infty$ define $X^{(p)} = \{x\in \mathcal{S}\Mtau: |x|^{p} \in X\}$ equipped with  $\|x\|_{X^{(p)}} = \|\, |x|^{p}\|^{1/p}_X$.  It is  called $p$-{\it convexification} of $X$. If $X$ is a quasi-normed symmetric space then $X^{(p)}$ is also a quasi-normed symmetric space.  In Proposition 3.1 in \cite{DDS3} there is a  list of  properties of convexification.   Among others we have that $E^{(p)}\Mtau = (E\Mtau)^{(p)}$,  and  $M^{(pr)}(X^{(r)}) = M^{(p)}(X)^{1/r}$, $M_{(pr)}(X^{(r)}) = M_{(p)}(X)^{1/r}$. These relations help to characterize convexity and concavity properties allowing the reduction of "power" of the spaces.  

The main results on convexity properties are based on the  inequalities presented in Lemma \ref{lem:kalsuk} and Theorem \ref{th:convexity} below. 
Observe that in \cite[Theorem 2.5 (i)]{AL1985}  it has been proved the inequality $S(x+y)^\gamma \prec S(x)^\gamma + S(y)^\gamma$, $0<\gamma \le1$, which under the assumption that $E$ is separable  implies the analogue of  Theorem \ref{th:convexity} in $C_E$, \cite[Lemma 3.1 (i)]{AL1985}. In \cite{KS} the assumption of separability of $E$ was removed via Lemma \ref{lem:kalsuk}.

\begin{lemma} \cite[Theorem 8.10 (ii)]{KS}
\label{lem:kalsuk}
Let $E$ be a symmetric normed space and $\varphi: [0,\infty) \to [0,\infty)$ be a continuous increasing concave function.  Then for $x, y \in E\Mtau$,
\[
\| \varphi(|x + y|)\|_{E\Mtau} \le \|\varphi(|x|)\|_{E\Mtau}  + \|\varphi(|y|)\|_{E\Mtau}.
\]

\end{lemma}

\begin{theorem}\label{th:convexity}

   \cite[Proposition 3.6]{DDS3}, \cite[Lemma 3.1 (i)]{AL1985}
Let $E$ be a quasi-normed symmetric space with the Fatou property. Let $0<p\le q <\infty$. If $E$ is $p$-convex then
\[
\left\|\left(\sum_{i=1}^n |x_i|^q \right)^{\frac{1}{q}}\right\|_{E\Mtau} \leq M^{(p)}(E)
\left(\sum_{i=1}^n \|x_i\|_{E\Mtau}^p \right)^{\frac{1}{p}}
\]
for every $x_1,x_2,\dots,x_n\in E\Mtau$.

\end{theorem}
\begin{proof}
Since $E$ is $p$-convex then it admits an equivalent symmetric norm if $p \ge 1$ (respectively $p$-norm if $0 < p < 1$) with convexity constant $1$ \cite[Corollary 3.5]{DDS3}.  So we assume that $M^{(p)}(E) = 1$. Then $E^{(1/p)} $ is $1$-convex with constant $1$, so $\|\cdot\|_{E^{(1/p)}}$ is a symmetric norm. The function $\varphi(u)= u^{p/q}$ is concave, so we apply Lemma \ref{lem:kalsuk}   for $z_i = |x_i|^p \in E^{(1/p)}\Mtau$, where $x_1,x_2,\dots,x_n \in E\Mtau$. Thus
\[
\left\|\left(\sum_{i=1}^n |z_i|^{q/p}\right)^{p/q}\right\|_{E^{(1/p)}\Mtau} \le \sum_{i=1}^n \|z_i\|_{E^{(1/p)}\Mtau}. 
\]
Hence
\begin{align*}
\left\|\left(\sum_{i=1}^n |x_i|^{q}\right)^{1/q}\right\|^p_{E\Mtau} &= \left\|\left(\sum_{i=1}^n |x_i|^{q}\right)^{p/q}\right\|_{E^{(1/p)}\Mtau} \\
&\le \sum_{i=1}^n \|\, |x_i|^p\|_{E^{(1/p)}\Mtau} = \sum_{i=1}^n \|x_i\|^p_{E\Mtau}. 
\end{align*}

\end{proof}

\begin{theorem} \cite[Theorem 3.8] {DDS3}\label{th:conv}
Let $E$ be a  quasi-normed symmetric space with the Fatou property. If $E$ is $p$-convex, $0<p<\infty$, then $E\Mtau$ is $p$-convex with $M^{(p)}(E\Mtau) \le M^{(p)}(E)$. If $\M$ is non-atomic then $E\Mtau$ is $p$-convex if and only if $E$ is $p$-convex, and in this case  $M^{(p)}(E\Mtau) = M^{(p)}(E)$. 

\end{theorem}

\begin{proof}
The first statement is a direct consequence of Theorem \ref{th:convexity} for $p=q$. The second one is a consequence of the isometric embedding of $E$ into $E\Mtau$, Proposition \ref{isom2}. Indeed there exists a $*$-isomorphism  $V: S\left(\left[0,\tauone\right),m\right) \to S(\Mtau)$ such that $\mu(V(x)) = \mu(x)$, $x\in S\left(\left[0,\tauone\right),m\right)$. Since for any $f\in S\left(\left[0,\tauone\right),m\right)$, $|V(f)|^2=(V(f))^*V(f)=V(\overline{f})V(f)=V(|f|^2)=(V|f|)^2$, so $|V(f)|=V|f|$. Therefore  in view of Lemma \ref{lm:singfun} (5),   for any $f_1,f_2,\dots,f_n\in E$  we have 
\begin{align*}
\mu\left(V\left(\sum_{i=1}^n |f_i|^p \right)^{\frac{1}{p}}\right) &=\mu\left(\sum_{i=1}^n |f_i|^p \right)^{\frac{1}{p}}=\mu^{\frac{1}{p}}\left(\sum_{i=1}^n |f_i|^p \right)=\mu^{\frac{1}{p}}\left(V\left(\sum_{i=1}^n |f_i|^p \right)\right)\\
&=\mu^{\frac{1}{p}}\left(\sum_{i=1}^n V(|f_i|^p)\right)=\mu\left(\left(\sum_{i=1}^n V(|f_i|^p) \right)^{\frac{1}{p}} \right)=\mu\left(\left(\sum_{i=1}^n |V(f_i)|^p \right)^{\frac{1}{p}} \right).
\end{align*}
Thus
\begin{align*}
\left\|\left(\sum_{i=1}^n |f_i|^p \right)^{\frac{1}{p}}\right\|_{E}
&= \left\|V\left(\sum_{i=1}^n |f_i|^p \right)^{\frac{1}{p}}\right\|_{E\Mtau}
=\left\|\left(\sum_{i=1}^n |V(f_i)|^p \right)^{\frac{1}{p}}\right\|_{E\Mtau} \\
&\leq M^{(p)}(E\Mtau)
\left(\sum_{i=1}^n \|V(f_i)\|_{E\Mtau}^p \right)^{\frac{1}{p}}
\leq M^{(p)}(E\Mtau) \left(\sum_{i=1}^n \|f_i\|_{E}^p \right)^{\frac{1}{p}},
\end{align*}
which implies that $M^{(p)}(E) \ \leq M^{(p)}(E\Mtau)$.
\end{proof}

In order to study concavity properties we need some new notions and additional assumptions.   It will be assumed that $E$ is a normed space and $1<q<\infty$. This is caused by the method of the proof which is based on duality. 

For any $x,y\in S\Mtau$ we write $x\triangleleft y$ and say that $x$ is {\it supermajorized} by $y$, 
 if 
 \[
 \int_t^\infty \mu(x) \ge \int_t^\infty \mu(y)\ \ \ \text{for all} \ \ t\ge 0.
 \] 
 Clearly $x\triangleleft y$ if and only if $\mu(x)\triangleleft \mu(y)$.   
 If $\int_0^\infty \mu(x) = \int_0^\infty \mu(y) <\infty$, then $x\prec y$ if and only if $x\triangleleft y$.

 The next two lemmas are the main ingredients in the concavity results. The first one was proved in discrete case also in \cite{AL1985} as Theorem 2.5 (ii) for $\varphi(t) = t^\gamma$ for $1\le \gamma < \infty$. 
 
\begin{lemma}\cite[Proposition 4.1]{DDS3} If $\psi$ is an increasing convex function on $[0,\infty)$ with $\psi(0) = 0$, then for any $0\le x,y \in S\Mtau$, 
 \begin{equation}\label{eq:conv01}
 \psi(|x+y|) \triangleleft \psi(\mu(x)) + \psi(\mu(y)), \ \text{equivalently} \ \  \psi(\mu(x+y)) \triangleleft \psi(\mu(x)) + \psi(\mu(y)).
 \end{equation}
 \end{lemma}

 \begin{lemma}\label{lem:conv01}\cite[Lemma 4.2]{DDS3}
 If $E$ is a  Banach symmetric space with $M_{(q)}(E) = 1$ for some $1<q<\infty$ then $\|g\|_{E^{(1/q)}} \le \|f\|_{E^{(1/q)}}$ for any bounded functions $f,g \in L^0$ with supports of finite measure and such that $f\triangleleft g$.
 \end{lemma}

The set $F(\M,\tau)=\{x\in \M :\tau(s(x))<\infty\}$   is a two sided ideal in $\M$ and its closure in the measure topology is $S_0(\M,\tau)$. If $x\in F\Mtau$ then $\mu(x)$ is a bounded function with support of finite measure.

 Let $\psi(t) = t^{q/p}$, where $1\le p\le q$, and $x_i \in F\Mtau$, $i=1,2,\dots,n$.  Recall also that $\mu\left(\psi(|x|) \right) = \psi(\mu(|x|))$ for any $x\in S(\M,\tau)$. Then  by (\ref{eq:conv01}), 
 \[
 \psi\left(\mu\left(\sum_{i=1}^n |x_i|^p\right)\right) \triangleleft 
\mu\left(\sum_{i=1}^n \psi(\mu(|x_i|^p))\right),
\]
 and so
 \begin{equation}\label{eq:conv02}
 \mu^{q/p}\left(\sum_{i=1}^n |x_i|^p\right)\triangleleft 
 \mu\left(\sum_{i=1}^n \mu^{q/p}(|x_i|^p)\right)= \mu\left(\sum_{i=1}^n \mu^q(x_i)\right).
 \end{equation}

 Let $E$ be $q$-concave with $M_{(q)}(E) =1$. Then $E^{(1/q)}$ is $1$-concave with $M_{(1)}(E^{(1/q)}) = 1$. Therefore by (\ref{eq:conv01}),  (\ref{eq:conv02}), and Lemma \ref{lem:conv01}, 
 \begin{align*}
 \sum_{i=1}^n \|\mu(|x_i|^q)\|_{E^{(1/q)}}& =\sum_{i=1}^n \|\mu^q(x_i)\|_{E^{(1/q)}} \le \left\|\sum_{i=1}^n \mu^q(x_i)\right\|_{E^{(1/q)}} \\
 & \le 
 \left\| \mu^{q/p} \left(\sum_{i=1}^n |x_i|^p\right)\right\|_{E^{(1/q)}} =  \left\| \mu \left(\sum_{i=1}^n |x_i|^p\right)^{q/p}\right\|_{E^{(1/q)}}.
 \end{align*}
 It follows that
 \[
 \sum_{i=1}^n \|x_i\|^q_{E\Mtau} = \sum_{i=1}^n \||x_i|^q\|_{E^{(1/q)}\Mtau} \le \left\|\left(\sum_{i=1}^n |x_i|^p\right)^{q/p}\right\|_{E^{(1/q)}\Mtau} = \left\|\left(\sum_{i=1}^n |x_i|^p\right)^{1/p}\right\|^q_{E\Mtau}.
 \] 
 We just have proved the following result under the assumption that $x_i\in F\Mtau$ and  $M_{(q)}(E) =1$. It is a parallel version to Theorem  \ref{th:convexity} for $q$-concavity.
 
 \begin{theorem} \label{th:concavity} \cite[Proposition 4.6]{DDS3}, \cite[Lemma 3.1 (ii)]{AL1985}
Let $1\le p\le q$.  If $E$ is a $q$-concave  Banach symmetric space with the Fatou property then for every $x_i \in E\Mtau$, $i=1,2,\dots, n$, 
\[
\left(\sum_{i=1}^n \|x_i\|^q_{E\Mtau}\right)^{1/q} \le M_{(q)}(E) \left\|\left(\sum_{i=1}^n|x_i|^p\right)^{1/p}\right\|_{E\Mtau}.
\]
 \end{theorem}

\begin{theorem}\cite[Theorem 4.7]{DDS3}\label{th:conc}
Let $E$ be a  Banach symmetric space. If $E$ is $q$-concave for some $1 < q < \infty$, then $E\Mtau$ is $q$-concave with $M_{(q)}(E\Mtau) \le M_{(q)}(E)$. If $\M$ is non-atomic, then $E$ is $q$-concave if and only if $E\Mtau$ is $q$-concave, in which case $M_{(q)}(E\Mtau) = M_{(q)}(E)$.

\end{theorem}

The next result, which is a corollary from Theorems \ref{th:conv} and \ref{th:conc}, has been proved  in \cite{AL1985} (Theorem 1.3) under the assumption that $E$ is a Banach separable space. 

\begin{corollary}

\begin{itemize}
\item[{(i)}] Let $0<p<\infty$ and $E$ be a  quasi-normed symmetric sequence space with the Fatou property.  Then $E$ is $p$-convex if and only if $C_E$ is $p$-convex. Moreover, $M^{(p)}(E)=M^{(p)}(C_E)$.
\item[{(ii)}] Let $1<q<\infty$ and $E$ be a  Banach symmetric sequence space with the Fatou property. Then $E$ is $q$-concave  if and only if $C_E$ is $q$-concave. Moreover, $M_{(q)}(E)=M_{(q)}(C_E)$.
\end{itemize}
\end{corollary}

Given $x\in B(H)$, the \textit{right} and \textit{left support projections}  of $x$, denoted by $r(x)$ and $l(x)$  are the projection onto $\ker^\perp{x}$ and $\ker^\perp{x^*}$  respectively, that is $r(x)=s(x)$ and $l(x)=s(x^*)$.
The operators $x,y\in B(H)$ are said to \textit{have right} (respectively, \textit{left}) \textit{disjoint supports} if $r(x)r(y)=0$ (respectively, $l(x)l(y)=0$).
Furthermore if $x_1,x_2,\dots, x_n \in S\Mtau$ are left disjoint then
\begin{equation}\label{eq:est1}
|x_1 +\dots + x_n| = (|x_1|^2 + \dots + |x_n|^2)^{1/2}.
\end{equation}
Indeed, since $x_i^* x_j=x_i^* l(x_i) l(x_j) x_j=0$ we have that 
\begin{align*}
|x_1 +\dots + x_n|^2&=(x_1 +\dots + x_n)^*(x_1 +\dots + x_n)=x_1^*x_1+x_2^*x_2+\dots+x_n^*x_n\\
&=|x_1|^2+|x_2|^2+\dots+|x_n|^2.
\end{align*}

We observe that  the similar equality does not hold for any  power $p\ne 2$, which is different than in the commutative case. 

Given $0<p, q<\infty$, we say that  a  quasi-normed symmetric space $X\subset \mathcal{S}\Mtau$ satisfies an {\it upper $p$-estimate}, respectively {\it lower $q$-estimate}, if  there exists $M>0$ such that for any 
 left disjoint $x_1,x_2,\dots,x_n \in X$,
 \[
  \left\|\sum_{i=1}^n x_i \right\| _{X} \le M \left(\sum_{i=1}^n \|x_i\|_{X}^p \right)^{\frac{1}{p}},
  \]
respectively,
 \[
 \left(\sum_{i=1}^n \|x_i\|_{X}^q \right)^{\frac{1}{q}} \le  M \left\|\sum_{i=1}^n x_i \right\| _{X}.
  \]
Again the infimum of the constant $M$ in the above inequalities will be called the upper $p$-estimate constant, respectively the lower $p$-estimate constant. 

Observe that $r(x^*)=s(x^*)=l(x)$ and $l(x^*)=s(x)=r(x)$, and so $x, y$ are left disjoint if and only if $x^*, y^*$ are right disjoint. 
Since also  $\|x\|_{E\Mtau} = \|x^*\|_{E\Mtau}$ the left disjointness in the above definition can be equivalently replaced by the right disjointness.

The proof of the next theorem follows from equality (\ref{eq:est1}) and Theorems \ref{th:convexity} and \ref{th:concavity} where we put $q=2$ and $p=2$, respectively.

\begin{theorem} \cite[Proposition 5.1]{DDS3}\label{th:estimates}
Let $E$ be a  Banach symmetric space with the Fatou property. 
\begin{itemize}
\item[$(i)$] If $1\le p\le 2$ and $E$ is $p$-convex then $E\Mtau$ satisfies an upper $p$-estimate with constant $M^{(p)}(E)$. 
\item[$(ii)$] If $q \ge  2$ and $E$ is $q$-concave then $E\Mtau$ satisfies a lower $q$-estimate with constant $M_{(q)}(E)$. 
\end{itemize}

\end{theorem}

The next theorem follows from Theorem \ref{th:estimates} and the well known relations among upper  estimate and convexity (resp. lower estimate and concavity) in Banach lattices (see diagram on pages 100, 101 in \cite{LT2}). 

\begin{theorem} 
 \cite[Corollary 5.3]{DDS3} \label{cor:etsimates}
Let $E$ be a  Banach symmetric space with the Fatou property.
\begin{itemize}
\item[$(i)$] If $1< p\le 2$ and $E$ satisfies an upper $p$-estimate then $E\Mtau$ satisfies an upper $r$-estimate for all $1\le r<p$. 
\item[$(ii)$] If $q \ge  2$ and $E$ satisfies a lower $q$-estimate then $E\Mtau$ satisfies a lower $s$-estimate for all $s>q$.
\end{itemize}
\end{theorem}

\begin{corollary}\cite[Corollary 5.2]{DDS3}\label{cor:lpestimates}
If $1\le p < \infty$ then $L^p\Mtau$ is $p$-convex and $p$-concave with $M^{(p)}(L^p\Mtau) =M_{(p)}(L^p\Mtau)=1$.    
Consequently if $1\le p\le 2$ then $L^p\Mtau)$ satisfies an upper $p$-estimate with constant one, and if $q\ge 2$ then   $L^p\Mtau$ satisfies a lower $q$-estimate with constant one.

\end{corollary}

 We say that the von Neumann algebra $\M$ has property $P(n)$ for some $n\in \mathbb{N}$ if there exist $n$ projections $e_i\in P(\M) \cap  F\Mtau$  that are mutually orthogonal and pairwise equivalent, and $\M$ has  property $P(\infty)$ if $\M$ has property $P(n)$ for every $n\in \mathbb{N}$.  Recall that projections $p$ and $q$ are equivalent if there exists a partial isometry $u$ such that $p=u^*u$ and $q=uu^*$. 
 
Assume that $E\Mtau$ satisfies an upper $p$-estimate. If $\M$ has property $P(n)$ then  there exist $e_1, e_2,\dots,e_n\in P(\M) \cap F(\M, \tau)$ mutually orthogonal and equivalent. 
Let $u_i \in \M$ be partial isometries such that $u_i^*u_i = e_1$ and $u_iu_i^* = e_i$, $i=1,2,\dots,n$. Note first that $\langle u_i^*\xi,u_i^*\xi\rangle=\langle \xi, u_iu_i^*\xi\rangle=\langle \xi, e_i \xi\rangle= \langle e_i \xi, e_i \xi\rangle$ for  all $i=1,2,\dots, n$ and $\xi \in H$. Hence $u_i^*\xi =0$ for all $\xi \in e_i^\perp$ and $l(u_i)=s(u_i^*)\leq e_i$.
Hence for $i\ne j$, $l(u_i)l(u_j)=l(u_i)e_i e_j =0$ and $u_i$'s are left disjointly supported for $i=1,2, \dots, n$. Consequently,  $w_{i1}=u_iu_1^*$, $i=1,2,\dots,n$, are also left disjoint and $\abs{\sum_{i=1}^n w_{i1}}=\left(\sum_{i=1}^n |w_{i1}^2|\right)^{\frac12}$. Observe next that $|w_{i1}|^2=u_1u_i^*u_i u_1^*=u_1e_1u_1^*=u_1u_1^*u_1u_1^*=e_1e_1=e_1$, and so $|w_{i1}|=e_1$ for all $i=1,2,\dots,n$. Hence
\begin{align*}
\abs{\sum_{i=1}^n w_{i1}}=\left(\sum_{i=1}^n |w_{i1}|^2\right)^{\frac12}=\left(\sum_{i=1}^n e_1\right)^{\frac12}=n^{\frac12}e_1,
\end{align*}
and by upper $p$-estimate of $E\Mtau$, 
\begin{align*}
n^{\frac12}\|e_1\|_{\nonsp}=\left\|\sum_{i=1}^n w_{i1}\right\|_{\nonsp}\leq K\left(\sum_{i=1}^n \|w_{i1}\|^p_{\nonsp}\right)^{\frac1p}=Kn^{\frac1p}\|e_1\|_{\nonsp}.
\end{align*}
Consequently, if $E\Mtau$ satisfies an upper $p$-estimate and $\M$ has property $P(\infty)$ then $p \leq  2$. Similarly one can show that if $E\Mtau$ satisfies a lower $q$-estimate for some $1\leq  q <\infty$ then $q\geq  2$, under the  assumption that $\M$ has
property $P (\infty)$.

The above remarks show part (i) of the next theorem. Part (ii) is proved by using the embedding  of $E$ into $E\Mtau$, Proposition \ref{isom2}.

\begin{theorem}\label{th:indices}
Let $E$ be a  Banach symmetric space.

\rm{(i)} \cite[Proposition 6.1]{DDS3}
If $\M$ has property $P(\infty)$ and $E\Mtau$ satisfies an upper $p$ (respectively, lower $q$)-estimate for some $1\le p < \infty$ (respectively, $1\le q < \infty)$, then $p\le 2$ (respectively, $q\ge 2$).

\rm{(ii)} \cite[Proposition 6.2]{DDS3}
If $\M$ is non-atomic and if $E\Mtau$ satisfies an upper $p$ (respectively, lower $q$)-estimate for some $1\le p < \infty$ (respectively, $1\le q <\infty)$, then so does $E$.

\end{theorem}

For a symmetric space $X\subset E\Mtau$ define
\begin{align*}
s(X) &=\sup\{p: X \ \text{ satisfies an upper $p$-estimate}\},\\
\sigma(X) &=\inf\{q: X \ \text{ satisfies a lower $q$-estimate}\}.
\end{align*}

The consequence of Theorem \ref{th:indices} is the following result.

\begin{theorem} \cite[Proposition 6.3]{DDS3},\cite[Theorem 1.5]{AL1985}\label{th:indices2}
 If $E$  has the Fatou property and  $\M$ is non-atomic and has property $P(\infty)$, then
\[
s(E\Mtau) = \max\{2, s(E)\}, \ \ \  \sigma(E\Mtau)  = \min \{2, \sigma(E)\}.
\]
If $\M=B(H)$, and $E$ is a Banach symmetric sequence  space with the Fatou property, then the above equalities  hold true also for unitary ideal $C_E$.
\end{theorem}

A stronger version of the above result is presented in Corollary 6.9 in \cite{DDS3}, which is an extension of Corollary 4.3 in \cite{AL1985}.

\begin{problem} Prove Theorem \ref{th:conc} for a quasi-normed symmetric space $E$.

\end{problem}

\section{Uniform and local uniform convexity}

 The \emph{modulus of convexity} of a normed space  $(X,\|\cdot\|)$  is given by \[
\delta_X(\epsilon)=\inf\{1-\|x+y\|/2:\,x,y\in B_X\text{ and }\|x-y\|\geq\epsilon\},\ \ \   \ 0\leq \epsilon\leq 2. 
\]
It is said that the modulus of convexity is of power $q$ if there exists a constant $c>0$ such that $\delta(\epsilon)\geq c\epsilon^q$ for  $2\ge \epsilon > 0$. We call $X$ \emph{uniformly convex} if $\delta_X(\epsilon)>0$ for all $0\leq \epsilon\leq 2$ \cite{LT2}.  Equivalently, $X$ is uniformly convex if for any sequences $\{x_n\},\{y_n\}\subset B_X$ the condition $\|x_n+y_n\|\to 2$ implies that $\|x_n-y_n\|\to 0$  as $n\to\infty$. 
It is well known that $\ell_p$ and $L_p$ are uniformly convex for $p>1$. 
$X$  is said to be \emph{uniformly convexifiable} if it admits an equivalent uniformly convex norm.
A normed space $(X,\|\cdot\|)$ is called \emph{locally uniformly convex} if the conditions $x_n,x\in X$, $\|x_n\|\to \|x\|$, $\|x_n+x\|\to 2\|x\|$ imply $\|x_n-x\|\to 0$ as $n\to \infty$. 

J. Dixmier in  \cite{Dix} proved that $C_p$ is uniformly convex for $p\geq 2$. He also observed that moduli of convexity in $C_p$ and $\ell_p$ are equivalent.
C. McCarthy in \cite{Carth} extended J. Dixmier's results  showing that  $C_p$ is uniformly convex for $p>1$ and the moduli of convexity of $C_p$ and $\ell_p$ are in fact the same.
 N. Tomczak-Jaegermann in \cite{Tom1974} gave an alternative proof for estimating the modulus of convexity of $C_p$.  She proved Clarkson type inequalities for $C_p$ spaces, analogous to the classical  ones in $\ell_p$.     

We turn our attention next to the question whether the space $\left(\nonsp,\normcomm{\cdot}\right)$ is uniformly convex (respectively, locally uniformly convex), if $(E, \|\cdot\|_{E})$ is uniformly convex (respectively, locally uniformly convex).

J. Arazy  in \cite{A1} showed that $E$ is uniformly convexifiable if and only if $C_E$ is uniformly convexifiable. 
 N. Tomczak-Jaegermann in \cite{Tom1984} and Q. Xu in \cite{X1989} showed the following result for $C_E$ and $E\Mtau$, respectively.
 
 \begin{theorem}  \cite[Theorem 2]{Tom1984}, \cite[Th\'eor\`eme (ii)]{X1989}
 Let $1<p\leq 2\leq q<\infty$. Let $E$ be a  Banach symmetric sequence  or function space which is $p$-convex and $q$-concave with $M^{(p)}(E)=1=M_{(q)}(E)$. Then $C_E$ or $E\Mtau$ is uniformly convex with modulus of convexity of power type $q$ and uniformly smooth with modulus of smoothness of power type $p$.
\end{theorem}  

V. Chilin, A. Krygin, F. Sukochev in \cite{CKS1992} investigated local uniform convexity and uniform convexity in noncommutative symmetric spaces $\nonsp$. In order to prove the main result on uniform convexity, they used the fact that if $(E,\|\cdot\|_E)$ is uniformly convex then there exists a norm $\norms{\cdot}$ on $E$ equivalent to $\|\cdot\|_E$ such that $(E,\norms{\cdot})$ is  $p$-convex and $q$-concave for some $1<p\le 2 \le q < \infty$, with $p$-convexity and $q$-concavity constants equal to $1$ \cite[Proposition 1.d.8, vol.II]{LT2}. Consequently by \cite{X1989}, they obtained uniform convexity of $(\nonsp, \normscomm{\cdot})$, where $\normscomm{x}=\norms{\mu(x)}$, $x\in \nonsp$. This combined with the uniform convexity of $(E,\|\cdot\|_E)$ allowed them to show that $(\nonsp, \normcomm{\cdot})$ is also uniformly convex.

\begin{theorem}\cite[Theorem 2.1, Theorem 3.1]{CKS1992}
If  $E$ is  uniformly convex (respectively, locally uniformly convex) then $\nonsp$ is 
uniformly convex (respectively, locally uniformly convex).
\end{theorem}

\begin{theorem}\cite[Corollary 2.1, Corollary 3.1]{CKS1992}
Let $\M$ be non-atomic. If $\nonsp$ is  uniformly convex (respectively, locally uniformly convex)  then so is $E$.
\end{theorem}

By identifying $C_E$ with a symmetric space of measurable operators $G(B(H),\tr)$ (see Section \ref{sec:unitary}), the discrete versions of the above results was also obtained in \cite{CKS1992}. 
\begin{theorem}\cite[Theorem 2.2, Theorem 3.2]{CKS1992}
\label{thm:urce}
 Let $E$ be a symmetric sequence space. Then $C_E$ is (locally) uniformly convex if and only if $E$  is (locally) uniformly convex.
\end{theorem}
We wish to observe Theorem 2.2 and Theorem 3.2 in \cite{CKS1992} were stated under the assumption of order continuity of $E$. However, since $E$ is isometrically embedded in $C_E$, the (local) uniform convexity of $C_E$ passes to $E$ for arbitrary symmetric sequence space $E$. 
By Remark \ref{rem:MLUR}  every  $MLUR$  space is order continuous. Since we have the following implications \cite{Lin},
\[
UR\implies LUR\implies MLUR,
\]
every $UR$ and $LUR$ space is order continuous.
Hence if $E$ is (locally) uniformly convex then it is order continuous and by \cite[Theorem 2.2, Theorem 3.2]{CKS1992}, $C_E$ is (locally) uniformly convex.

 \section{Complex uniform convexity}
 \label{sec:cur}
 The following moduli of complex convexity of a complex quasi-normed space $(X,\|\cdot\|)$
were introduced in \cite{DGT1984}.  For $0 < p < \infty$ and $\epsilon \geq 0$,  we set  
\[
H_p^X(\epsilon)=\inf\left\{\left(\int_0^{2\pi} \|x+e^{i\theta}y\|^pd\theta\right)^{1/p}-1:\,\|x\|=1, \|y\|=\epsilon\right\},
\]
and
\[
H_\infty^X(\epsilon)=\inf\{\sup\{\|x+e^{i\theta}y\|:\,0\leq \theta \leq 2\pi\}-1:\,\|x\|=1, \|y\|=\epsilon\}.
\]
We say that the space $X$ is \emph{  complex uniformly convex} if $H_\infty^X(\epsilon)>0$ for all $\epsilon > 0$, and that $X$ is  \emph{uniformly $PL$-convex}   if $H_p^X(\epsilon)>0$ for all $\epsilon > 0$ and for some $0 < p <\infty$.  It was proved in \cite[Theorem 2.4]{DGT1984} that the previous definition is equivalent with $H_p^X(\epsilon)>0$ for all $0<p<\infty$. So we can say that $X$ is uniformly $PL$-convex when $H_1^X(\epsilon)>0$ for all $\epsilon>0$. Moreover, as shown in \cite{Dil} there exists a constant $C>0$ such that for all complex Banach spaces $X$ and all $0<\epsilon\leq 1$ we have
\[
C(H_\infty^X(\epsilon))^2\leq H_1^X(\epsilon)\leq H_\infty^X(\epsilon).
\]

Hence for complex Banach spaces, uniform complex convexity coincide with uniform $PL$-convexity. The same is not true in quasi-Banach lattices, where uniform complex convexity does not necessarily imply $PL$-convexity \cite{Ler}.  However, quasi-Banach lattices $X$ with $p$-convexity constant $M^{(p)}(X)=1$ for $0<p<\infty$ are complex uniformly convex if and only if they are uniformly $PL$-convex \cite[Theorem 3.4]{HanJu}.
Moreover, $X$ is said to be $r$-\emph{uniformly $PL$-convex} ($2\leq r<\infty$) whenever there is $K\geq 1$ such that $\left(\frac{\epsilon}{K}\right)^r\leq KH_1^X(\epsilon)$ for all $0<\epsilon <\frac1 K$. 
 
U. Haagerup observed that the dual of $C^*$-algebra is uniformly complex convex. His result with the proof is presented in \cite[Theorem 4.3]{DGT1984}.
 Since  the trace class $C_1$ is a dual  space of $C^*$-algebra $K(H)$ of compact operators on $H$, it is complex uniformly convex.
 Later K. Mattila in \cite[Lemma 3.1]{Mattila} gave an alternative proof of the complex uniform convexity of $C_1$.
 Similarly, since the noncommutative space $L_1\Mtau$ is a K\"{o}the dual of the von Neumann algebra $\M$ and thus it is an isometric subspace of $\M^*$, by U. Haagerup's result it is complex uniformly rotund. A direct proof of complex uniform convexity of $L_1\Mtau$ has been shown in \cite[Theorem 3.2]{czer-cur}. 

T.  Fack showed in \cite{Fack1987}  that if $\M$ is a factor (Lemma 12) or $H$ is separable (Theorem 4),  then $L_p(\M, \tau)$ is $q$-uniformly $PL$-convex for $q=\max(2,p)$.  
  
The research on how the properties of $E$ reflect on complex uniform convexity of $\nonsp$ started with the following result on $C_E$.

\begin{theorem}\cite[Theorem 1]{Tom1984}
\label{thm:tomczakplconv}
If $E$ is a symmetric Banach sequence space which is $q$-concave, $2\leq q<\infty$, with $M_{(q)}(E)=1$, then $C_E$ is  $q$-uniformly $PL$-convexifiable.
\end{theorem}

Q. Xu observed that if $E$ is a quasi-Banach lattice then $E$ is $q$-concave for some $q<\infty$ if and only if $E$ is uniformly complex convexifiable \cite[Corollary 3.3]{X1991}. By this, combined with Theorem \ref{thm:tomczakplconv} and the fact that  $C_E$ contains an isometric copy of $E$ (see Proposition \ref{prop:isomarazy}), we conclude the next result.
\begin{corollary}
\label{cor:unifconvexifce}
Let $E$ be a symmetric Banach sequence space. Then $E$ is  complex uniformly  convexifiable if and only if $C_E$ is  complex uniformly convexifiable.
\end{corollary}

Q. Xu in \cite{X1991} investigated complex uniform convexity of $\nonsp$. He assumed that $E$ is a symmetric function space with a weak Fatou property. We say that $E$ has the \emph{weak Fatou} property  if for $f_n, f\in E$ with $f_n\uparrow f$ a.e. it follows that $\|f_n\|_E\to \|f\|_E$.

\begin{theorem}\cite[Theorem 4.4]{X1991} Let $E$ be a symmetric Banach space  with the weak Fatou property and for
 some $1 < p\le q<\infty$, $M^{(p)}(E) = M_{(q)}(E) = 1$.  Then $\nonsp$ is a complex uniformly convex space.
\end{theorem}

Moreover, Q. Xu generalized Corollary \ref{cor:unifconvexifce} to noncommutative $\nonsp$ spaces.

  \begin{theorem}\cite[Corollary 4.6, Corollary 3.3]{X1991}\
  \label{thm:PL-conv}
  Let $E$ be a symmetric quasi-Banach function space with the weak Fatou property. Then the following statements are equivalent.
  \begin{itemize}
  \item[{(i)}] $E$ is $q$-concave for some $q<\infty$.
    \item[{(ii)}] $E$ is uniformly $PL$-convexifiable.
      \item[{(iii)}] $\nonsp$ is uniformly $PL$-convexifiable.
\end{itemize}   
  \end{theorem}
Recall that uniform $PL$-convexity and complex uniform  convexity  coincide for Banach spaces, but not for quasi-Banach lattices  $E$ \cite{Ler}, unless their convexity constants $M^{(p)}(E)=1$, $0<p<\infty$ \cite[Theorem 3.4]{HanJu}. Hence uniform $PL$-convexifiability can be replaced with complex uniform convexifiability, under assumption that $E$ is a symmetric Banach space.

 In \cite{czer-cur} the relations between complex uniform  convexity of $E$ and $\nonsp$ have been studied by  one of the authors of this survey.
The following result combines Theorems 2.6 and 2.7  in \cite{czer-cur}.

\begin{theorem}
\label{thm:czercur}
If $E$  is complex uniformly convex then $\nonsp^+$ is complex uniformly convex. If in addition $\M$ is non-atomic then complex uniform convexity of $\nonsp^+$ implies  complex uniform convexity of $E$.   
\end{theorem}
Therefore if  $\M$ is non-atomic, complex uniform convexity of $E$ is equivalent to complex uniform convexity of $\nonsp^+$. From the above it also follows that if $\nonsp$ is complex uniformly convex and $\M$ is non-atomic, then the subspace $\nonsp^+$ of $\nonsp$  is complex uniformly convex, and $E$ is  complex uniformly convex. 

 Moreover, under the assumption that $E$ is $p$-convex for some $p>1$, complex uniform convexity of $E$ implies complex uniform convexity of $\nonsp$ \cite[Theorem 2.6]{czer-cur}.  Hence the following holds. 
 
 \begin{theorem}
 \label{thm:czercur1}
 If $E$ is $p$-convex for some $p>1$ then $\nonsp$ is complex uniformly convex whenever $E$ is complex uniformly convex. If $\M$ is non-atomic and $\nonsp$ is complex uniformly convex then $E$ is complex uniformly convex.
 \end{theorem}

  The analogous results followed for the unitary matrix space $C_E$.
  \begin{theorem}\cite[Theorem 2.10]{czer-cur} \label{th:compunifconv2}
   Let $E$ be a symmetric Banach sequence space. Then $C_E^+$ is  complex uniformly convex  if and only if $E$ is complex uniformly convex. Moreover, if $E$ is $p$-convex for some $p>1$, then $C_E$ is complex uniformly convex if and only if $E$ is complex uniformly convex.
\end{theorem}

 Let a normed space $(X, \|\cdot\|)$ be  partially ordered by $\le$. Then $X$ is said
to be \textit{uniformly monotone}  whenever for any $\epsilon>0$ there exists $\delta(\epsilon)>0$ such that for any $0\leqslant x, y\in X$ we have $\|x+y\|>1+\delta(\epsilon)$, whenever $\|y\|\geq\epsilon$ and $\|x\|=1$. If in addition $x \wedge y=0$, then $X$  is said to be \textit{disjointly uniformly monotone}.

It is known that complex uniform  convexity of a Banach lattice is equivalent to its uniform monotonicity \cite[Theorem 3.4]{HanJu}. It was first discovered for Banach function space in \cite[Theorem 2]{HN}.

The next result relates complex uniform convexity of $\nonsp^+$ or $\nonsp$ with the uniform monotonicity.
\begin{corollary}\cite[Corollary 2.9]{czer-cur}
\label{cor:nonspcurum}  Let $\M$ be non-atomic. The space $\nonsp^+$ is complex uniformly convex if and only if  $\nonsp$ is uniformly monotone.  Moreover, if $E$ is $p$-convex for some $p>1$, then  $\nonsp$ is complex uniformly convex if and only if  $\nonsp$ is uniformly monotone.
\end{corollary}

We will see in Section \ref{sec:17} that complex convexity properties  of $E$ and $\nonsp$ are also  related to  Kadec-Klee properties. It is  summarized in Corollary \ref{cor:ukkcurum}.

\begin{problem}
As we mentioned above  $L_1\Mtau$ is complex uniformly convex.  However this  does not follow from Theorem \ref{thm:czercur1}, since $L_1$ is not $p$-convex for any $p>1$. 
Show Theorems \ref{thm:czercur1}, \ref{th:compunifconv2} and Corollary \ref{cor:nonspcurum} without assumption that $E$ is $p$-convex for some $p>1$.

\end{problem}

\section{Smoothness}
 
 For a normed space $(X,\|\cdot\|)$, an element  $x\in S_X$ is said to be  a \emph{smooth point} of $B_X$ if there exists a unique  functional $F\in S_{X^*}$ which supports $B_X$ at $x$, that is $F(x)=1$. We will say then that the functional $F$ \emph{supports} $x$.
A normed space $X$ is said to be \emph{smooth} (or \emph{G\^{a}teaux smooth}) if every $x$ from the unit sphere is a smooth point \cite{ Diestel,Godefroy}.

If $T$ is a linear isometry from a Banach space $X$ onto a Banach space $Y$, then $x\in S_X$ is a smooth point of $B_X$ if and only if $T(x)$ is a smooth point of $B_Y$. Moreover,  smooth points of a normed space remain smooth on its subspaces.

It is worth to observe that a unique functional $F\in X^*$ supporting the smooth point $x$ is an extreme point of $B_{X^*}$. Indeed, letting $F=(F_1+F_2)/2$, where $F_1,F_2\in B_{X^*}$, we have
$2=2F(x)=F_1(x)+F_2(x)$. Since $\abs{F_1(x)}, \abs{F_2(x)}\leqslant 1$ it follows that $F_1(x)=F_2(x)=1$. Using now the fact that $F$ is a unique functional supporting $x$, we get that $F_1=F_2=F$.

As an elementary example note that $x\in S_{\ell_1}$ is smooth if and only if $\supp (x)=\mathbb N$. So any element from the unit sphere with all coordinates different than zero is smooth.  It follows that its supporting functional is determined by a  unique normalized element $y=\{y_n\}\in \ell_\infty$ such that $y_n=1$ if $x_n>0$ and $y_n=-1$ if $x_n<0$.

The study of smooth points in noncommutative spaces started with J. Holub \cite{Holub}, who considered them in the trace class $C_1$.

\begin{theorem}\cite[Theorem 3.2]{Holub}
\label{thm:Holubsmoothc1}
 Let $x\in C_1$, $\|x\|_{C_1}=1$. Then $x$ is smooth if and only if $x$ or $x^*$ is one-to-one.
\end{theorem}

Later on, J. Arazy   characterized smooth points in $C_E$.

\begin{theorem} \cite[Theorem 2.3]{A1} 
\label{thm:Araysmoothce}
 Let $E$ be a separable symmetric sequence space and $x\in S_{C_E}$ Then $x$ is a smooth point of $B_{C_E}$ if and only if $S(x)$ is a smooth point of $B_E$. Consequently, in the commutative case, $x\in B_E$ is smooth if and only if $\mu(x)$ is smooth of the ball $B_E$.
\end{theorem}

The characterization of smooth points of $B_{\nonsp}$ was done in \cite{CKK2012}, for   order continuous symmetric function  spaces $E$.

\begin{theorem}\cite[Theorem 2.4]{CKK2012}
\label{thm:6}
Suppose that $E$ is order continuous. Let $x\in S_E$ and $\mu(x)$ be a smooth point of $B_E$, and $F(h)=\int_0^\infty hf$, $h\in E$, for some $f\in S_{E^{\times}}$,  be the functional supporting $\mu(x)$. If
\begin{itemize}
\item [{(i)}] $\mu(\infty,f)=0$, or
\item [{(ii)}] $s(x^*)=\one$,
\end{itemize}
  then $x$ is a smooth point of $B_{\nonsp}$.
\end{theorem}

 Recall that the trace $\tau$ on $L_1\Mtau$ is an additive positively homogeneous real valued functional, satisfying $\tau(x)=\int_0^\infty \mu(x)$  for all $x\in L_1\Mtau^+$.  Below we provide a list of basic properties  of $\tau$ on $L_1\Mtau$.
\begin{lemma} \label{lm:traceprop}The following properties hold for the extended trace  $\tau:L_1\Mtau\to\Complex$.
\begin{itemize}
\item[{(i)}] $\tau(x^*)=\overline{\tau(x)}$, for $x\in L_1\Mtau$.
\item[{(ii)}]  $|\tau(xy)|\leq \|y\|_{\M}\tau(|x|)$ for $x\in L_1\Mtau$ and $y\in \M$. In particular if $y=\one$ and $x\in L_1\Mtau$ then $|\tau(x)|\leq \tau(|x|)$. 
\item[{(iii)}] \cite[Proposition 3.4]{DDP4} $\tau(xy)=\tau(yx)$ if $xy,yx\in L_1\Mtau$.
\item[{(iv)}] \cite[Proposition 3.10]{DDP4} $\tau(|xy|)=\int_0^{\infty}\mu(xy)\leq \int_0^{\infty}\mu(x)\mu(y)$ for $x,y\in S\Mtau$. 
\end{itemize}
\end{lemma} 
\begin{proof}
The discussion of (ii) can be found at the beginning of section 3 in \cite{DDP4}. To show (i) observe first that $\tau(y)$ is real for any self-adjoint operator $y\in L_1\Mtau$, since $y$ can be written as a difference of its positive and negative parts. Now let $x\in L_1\Mtau$ and $\Rep(x),\, \Imp(x)$ be its real and imaginary parts, respectively. Then $x=\Rep(x)+i\Imp(x)$ and $x^*=\Rep(x)-i\Imp(x)$.  Hence $\tau(x^*)=\tau(\Rep(x)-i\Imp(x))=\tau(\Rep(x))-i\tau(\Imp(x))=\overline{\tau(\Rep(x))+i\tau(\Imp(x))}=\overline{\tau(\Rep(x)+i\Imp(x))}=\overline{\tau(x)}$.
\end{proof}

\begin{lemma}
\label{lm:smoothadjoint}
Let $E$ be order continuous, $x\in E\Mtau$, $y\in E^{\times}\Mtau$ with $\|x\|_{E\Mtau}=1$, and $\|y\|_{E^{\times}\Mtau}=1$. Then $y$ supports $x$ if and only if $y^*$ supports $x^*$. In particular, $x$ is a smooth point of $B_{\nonsp}$ if and only if $x^*$ is a smooth point of $B_{\nonsp}$. 
\end{lemma}
\begin{proof}
Since $(x^*)^*=x$ it is enough to show that if $x$ is a smooth point of $B_{\nonsp}$ then so is $x^*$. Let $x$ be a smooth point of $B_{\nonsp}$ and $\Phi_y(z)=\tau(zy)$, $z\in \nonsp$, $y\in S_{E^{\times}\Mtau}$, be the unique functional supporting $x$.  Suppose $\Phi_w(z)=\tau(zw)$, $z\in \nonsp$, $w\in S_{E^{\times}\Mtau}$, is a functional supporting $x^*$. By Lemma \ref{lm:traceprop} (i), $\tau(xw^*)=\overline{\tau(wx^*)}$ and by Lemma \ref{lm:traceprop} (iii), $\tau(wx^*)=\tau(x^*w)=\Phi_w(x^*)=1$. Hence $\Phi_{w^*}(x)=\tau(xw^*)=1$ and by the uniqueness  of $\Phi_y$ supporting $x$, we have that $w^*=y$ or $w=y^*$. Thus $x^*$ is a smooth point in $\nonsp$, where $\Phi_{y^*}(z)=\tau(zy^*)$, $z\in\nonsp $, is its unique supporting functional.
\end{proof}

By Lemma \ref{lm:smoothadjoint}, it is clear that the same conditions on $x$ and $x^*$ as well as on $y$ and $y^*$ need to be satisfied in the result below.
\begin{lemma}\cite[Lemma 2.5]{CKK2012}
\label{lm:smooth}
Let $E$ be  order continuous. If $x\in S_{\nonsp}$ is a smooth point of $B_{\nonsp}$ and the functional $\Phi_y(z)=\tau(zy)$, $z\in\nonsp$, $y\in E^{\times}\Mtau$, supports $x$, then either
\begin{itemize}
\item [{(i)}] $\mu(\infty,y)=0$, or
\item [{(ii)}] $s(x)=s(x^*)=s(y)=s(y^*)=\textup{\textbf{1}}$, $\abs{y}\ge \mu(\infty,y)\textup{\textbf{1}}$ and $\abs{y^*}\ge \mu(\infty,y)\textup{\textbf{1}}$.
\end{itemize}
\end{lemma}

\begin{theorem}\cite[Theorem 2.8]{CKK2012}
\label{thm:4}
Let $E$ be  order continuous and $\M$ be non-atomic. If  $x$ is a smooth point of $B_{\nonsp}$ then $\mu(x)$ is a smooth point of $B_E$.
\end{theorem}

The next theorem combines the results of Theorem \ref{thm:6}, Theorem \ref{thm:4} and Lemma \ref{lm:smooth}.
\begin{theorem}
\label{thm:smoothsummary}\cite[Theorem 2.9]{CKK2012}
Let $E$ be order continuous and $\M$ be non-atomic. Then $x$ is a smooth point of $B_{\nonsp}$ if and only if $\mu(x)$ is a smooth point of $B_E$  and either
\begin{itemize}
\item[{(i)}] $\mu(\infty,f)=0$, where $F(h)=\int_0^\infty h f$, $h\in E$, $f\in S_{E^{\times}}$, is the functional supporting $\mu(x)$ or,
\item[{(ii)}] $s(x^*)=\one$.
\end{itemize}
\end{theorem}

The following corollaries are direct consequences of the results  above.
\begin{corollary}\cite[Corollary 2.13]{CKK2012}
Let $\M$ be non-atomic and the space $E$ be  order continuous  such that $E^{\times}=(E^{\times})_0$.  Then $E$ is smooth if and only if $\nonsp$  is smooth.
\end{corollary}

Considering the commutative von Neumann algebra $\mathcal{M}=L_\infty[0,\alpha)$, $0<\alpha\leqslant \infty$, we obtain the corresponding result for the symmetric function spaces.
\begin{corollary}\cite[Corollary 2.10]{CKK2012}
Let $E$ be an order continuous symmetric function space on $[0,\alpha)$, $0<\alpha\leqslant \infty$. Then the function $x$ is a smooth point of $B_E$ if and only if its decreasing rearrangement $\mu(x)$ is a smooth point of $B_E$, and either
\begin{itemize}
\item[{(i)}] $\mu(\infty,f)=0$,  where $f\in S_{E^{\times}}$ induces the integral supporting functional of $\mu(x)$, or
\item[{(ii)}] $\supp(x)=[0,\alpha)$ a.e.
\end{itemize}
\end{corollary}

\begin{problem}
Find relations between  smooth points of the unit ball of $E\Mtau$ or $C_E$, and the unit ball of $E$,  without assumption that $E$ is order continuous. Consequently characterize smoothness of $E\Mtau$ and $C_E$ for any symmetric function or sequence space.

\end{problem}

\section{Strong smoothness}
 Given a normed space $(X,\|\cdot\|)$, let $x\in S_X$ be a smooth point of $B_X$ and $F$ be its supporting functional. If for any sequence $\{F_n\}\subset B_{X^{*}}$ the condition $F_n(x)\to 1$ implies $\|F_n-F\|_{X^{*}}\to 0$ as $n\to\infty$ then $x$ is called a {\it strongly smooth point} of $B_{X}$, and we say that $F$ \emph{strongly supports} $x$. A  normed space $X$ is said to be {\it Fr\'{e}chet smooth} if every $x$ from the unit sphere is a strongly smooth point. For these definitions and their applications we refer to \cite{Diestel, Godefroy}.

Recall  that $x\in S_X$ is a \emph{strongly extreme point} of $B_X$ whenever $\|x\pm x_n\|\to 1$, $\{x_n\}\subset X$, implies that $\|x_n\|\to 0$ \cite{Lin}.
It is easy to observe that the functional $F\in S_{X^*}$ which strongly supports  $x\in S_X$, is a strongly extreme point of $B_{X^*}$. Indeed, let $\|F\pm F_n\|_{X^*}\to 1$, for the sequence $\{F_n\}\subset X^*$. 
 By the inequality $ |1\pm F_n(x)|=|(F\pm  F_n)(x)| \leqslant \|F\pm F_n\|_{X^*}\|x\|_X$, it follows that $\overline{\lim}_n |1\pm F_n(x)|\le 1$, and so $\lim_n|1 \pm F_n(x)|=1$. Therefore $\lim_n F_n(x) =0$ and $\lim_n(F -F_n)(x) = 1$. By the assumption that $F$ strongly supports $x$, $\|F_n\|_{X^*} = \|F - (F- F_n)\|_{X^*}\to 0$, showing that $F$ is strongly extreme point of $B_{X^*}$.

Strongly smooth points  in the context of the spaces $E\Mtau$ or $C_E$ have been considered only in \cite{CKK2012}. The following results were obtained.

 \begin{proposition}\cite[Proposition 3.2]{CKK2012}
\label{prop:strsmoothoc}
Let $E$ be order continuous.
If $x\in S_{\nonsp}$ is strongly smooth and the operator $y\in S_{E^{\times}\Mtau}$ is such that the functional $\Phi_y(z)=\tau(zy), z\in \nonsp$, strongly supports $x$, then $y$ is an order continuous element of $E^{\times}\Mtau$.
\end{proposition}

\begin{theorem}\cite[Theorem 3.3]{CKK2012}
\label{thm:5}
Let $E$ be order continuous, the trace $\tau$ on $\M$ be $\sigma$-finite and $x\in S_{E\Mtau}$.  If  $\mu(x)\in S_E$ is a strongly smooth point of $B_E$ then $x$ is a strongly smooth point of $B_{\nonsp}$.
\end{theorem}

\begin{theorem}\cite[Theorem 3.7]{CKK2012}
\label{thm:7}
Let $E$ be  order continuous and $\M$ be non-atomic.  If $x\in S_{\nonsp}$ is a strongly smooth point of $B_{\nonsp}$ then $\mu(x)$ is a strongly smooth point of $B_E$.
\end{theorem}

Let us summarize the above results.

\begin{theorem}\cite[Theorem 3.8]{CKK2012} \label{th:strsmooth}
Let $E$ be  order continuous, $\M$ be non-atomic and the trace $\tau$  be $\sigma$-finite. Then $x\in S_{\nonsp}$ is a strongly smooth point of $B_{\nonsp}$ if and only if $\mu(x)$ is a strongly smooth point of $B_E$.
\end{theorem}

Considering the commutative von Neumann algebra $\M=L_{\infty}[0,\alpha)$, $0<\alpha\leqslant \infty$, we obtain the following consequence of the previous theorem.

\begin{corollary}\cite[Corollary 3.9]{CKK2012}
Let $E$ be an order continuous symmetric function space. Then the function $f$ is a strongly smooth point of $B_E$ if and only if its decreasing rearrangement $\mu(f)$ is a strongly smooth point of $B_E$.
\end{corollary}

The analogous result on strongly smooth points in $C_E$ can be proved using similar techniques as in the case of $\nonsp$. The proof of the next theorem is a good overview of various strategies employed in \cite{CKK2012}. By Lemma \ref{lm:traceprop} applied for $\M=B(H)$ with the canonical trace $\tr$ we have that $|\tr(x)|\leq \tr(|x|)$ for $x\in C_1$ and $\tr(|xy|)\leq \sum_{n=1}^\infty s_n(x)s_n(y)$ for $x,y\in B(H)$.

\begin{theorem}\cite[Theorem 3.11]{CKK2012}
\label{thm:strsmoothce}
Let $E$ be an order continuous symmetric sequence space  and let $x\in S_{C_E}$. Then the sequence of singular numbers $S(x)=\{s_n(x)\}$ is a strongly smooth point of $B_E$ if and only if $x$ is a strongly smooth point of $B_{C_E}$.
\end{theorem}

\begin{proof}
 Suppose first that $x$ is a strongly smooth point of the unit ball $B_{C_E}$. We provide the proof only for $x\geq 0$, but this can be extended to an arbitrary $x$ by \cite[Lemma 2.1]{CKK2012}. By Proposition \ref{prop:isomarazy} there exists a $*$-isomorphism $V:E\to C_E$ for which $V(S(x))=x$ and $S(V(a))=\mu(a)$ for any $a=\{a_n\}\in E$.
 
 Suppose that the functional $\Phi_y(z)=\tr(zy)$, $z\in C_E$, strongly supports $x$, where $y\in S_{C_{E^{\times}}}$.
We will show that the functional $F(a)=\sum_{i=1}^\infty a(i)s_i(y)$, $a=\{a(i)\}\in E$, strongly supports $S(x)$. 
 We have  
\[
1=\Phi_y(x)=|\tr(xy)|\le \tr(|xy|)\leq \sum_{i=1}^\infty s_i(x) s_i(y)=F(S(x))\leq \|S(x)\|_E\|S(y)\|_{E^\times}\leq  1.
\]
Hence $F(S(x))=1$. Moreover, since $*$-isomorphism $V$ preserves the order, $xV(S(y))=V(S(x)S(y))\geq 0$ and
\[
\Phi_{V(S(y))}(x)=\tau(xV(S(y)))=\|V(S(x)S(y))\|_{C_1}=\|S(x)S(y)\|_{\ell_1}=\sum_{i=1}^\infty s_i(x)s_i(y)=1.
\]
By the uniqueness of the functional $\Phi_y$ supporting $x$, it follows
that $y=V(S(y))$.  Suppose now that $G(S(x))=\sum_{i=1}^\infty s_i(x) c(i)=1$, where $G(a)=\sum_{i=1}^\infty a(i) c(i)$, $a=\{a(i)\}\in E$, for some $c=\{c(i)\}\in S_{E^{\times}}$. It is not difficult to see that $\sum_{i=1}^\infty s_i(x)|c(i)|=1$, and so also $\sum_{i=1}^\infty s_i(x)(|c(i)|+c(i))/2=1$.
 
By the previous argument applied to $xV(|c|)\geq 0$ and $xV((|c|+c)/2)\geq 0$ instead of $xV(S(y))$ we can show that $\Phi_{V(|c|)}(x)=\tau(xV(|c|))=1$ and $V(|c|)=y$, as well as  $\Phi_{V((|c|+c)/2))}(x)=\tau(xV((|c|+c)/2))=1$ and $V((|c|+c)/2)=y$. Hence $V(|c|)=V(S(y))=V((|c|+c)/2)$ and since $V$ is one-to-one $|c|=S(y)=(|c|+c)/2$. Consequently $c=|c|=S(y)$, proving that the functional $F$ is a unique functional supporting $S(x)$.

Suppose that $F_n(x)=\sum_{i=1}^\infty s_i(x)b_{n}(i)\to 1$, where $b_n=\{b_{n}(i)\}_{i=1}^\infty$ and $\{b_n\}_{n=1}^\infty\subset B_{E^{\times}}$. The goal is to show that then $\|F_n-F\|_{E^*}=\|b_{n}-S(y)\|_{E^{\times}}\to 0$. If is clear that $\sum_{i=1}^\infty s_i(x)\abs{b_{n}(i)}\to1$.
Since $V$ as a $*$-homomorphism is positive, it follows that
\[
xV(\abs{b_n})=V(S(x))V(\abs{b_n})=V(S(x)\abs{b_n})\ge 0, \ \ \ \ n\in\mathbb{N},
\] and
\begin{align*}
\Phi_{V(\abs{b_n})}(x)=\tr(xV(\abs{b_n}))=\|xV(\abs{b_n})\|_{C_1}=\|S(x)\abs{b_n}\|_{\ell_1}
=\sum_{i=1}^\infty s_i(x)\abs{b_{n}(i)} \to 1.
\end{align*}
Thus since $x$ is strongly smooth,
\begin{align*}
\|S(y)-\abs{b_n}\|_{E^{\times}}&=\|V(S(y))-V(\abs{b_n})\|_{C_{E^{\times}}}=\|y-V(\abs{b_n})\|_{C_{E^{\times}}}\\
&=\|\Phi_y-\Phi_{V(\abs{b_n})}\|_{E^*}\to 0.
\end{align*}
Now by $(\abs{b_n}+b_n)/2\ge 0$ and $\sum_{i=1}^\infty s_i(x)(\abs{b_{n}(i)}+b_{n}(i))/2\to1$, again it follows that $\|S(y)-(\abs{b_n}+b_n)/2\|_{E^{\times}}\to0$. Hence
\[
\|S(y)-b_n\|_{E^{\times}}\leqslant \|2S(y)-b_n-\abs{b_n}\|_{E^{\times}}+\|S(y)-\abs{b_n}\|_{E^{\times}}\to0,
\]
which shows that $S(x)$ is a strongly smooth point of $B_E$.

It is known and standard to check that there are no strongly smooth points in $\ell_1$. Therefore by the preceding argument,  $C_1$ has no strongly smooth points.

Suppose now that $E\neq\ell_1$, $x\in C_E$, and $S(x)$ is a strongly smooth point of $B_E$. We will show that $x$ is a strongly smooth point  of $B_{C_E}$. Let $F(a)=\sum_{i=1}^\infty a(i) b(i)$, $a=\{a(i)\}\in E$, for some $b=\{b(i)\}\in S_{E^{\times}}$, be a functional that strongly supports $S(x)$. By \cite[Theorem 2.3]{A1}  and its proof, if $
S(x)$ is a smooth point of $B_E$ then $x$ is a smooth point of $B_{C_E}$ and moreover,  the functional $\Phi_y(x)=\tr(xy)=\sum_{i=1}^\infty s_i(xy)$ supports $x$, for $y\in C_{E^{\times}}$, whose sequence $S(y)$ of singular numbers satisfies the condition $s_i(y)=b(i)$, $i\in \mathbb{N}$.

Suppose that $\Phi_{y_n}(x)\to 1$, for the sequence $\{y_n\}\subset B_{C_{E^{\times}}}$. Since by \cite[Theorem 4.2 (iii)]{Fack-Kos1986} $S(xy_n)\prec S(x)S(y_n)$  we have that $\sum_{i=1}^\infty s_i(xy_n)\leqslant \sum_{i=1}^\infty s_i(x)s_i(y_n)\leqslant 1$,  and $ \sum_{i=1}^\infty s_i(x)s_i(y_n)\to 1$ as $n\to \infty$. Applying the assumption that $S(x)$ is strongly smooth, it follows that  $\|S(y)-S(y_n)\|_{E^{\times}}\to 0$.
One can also show that $\|S(y)-S\left((y+y_n)/2\right)\|_{E^{\times}}\to 0$. By the assumption $E\ne \ell_1$ we have $E^{\times}\neq\ell_\infty$, and by \cite[Lemma 1.3]{czer-kam2010}, $y_n\stackrel\tr\to y$, that is $y_n$ converges to $y$ in measure.

Similarly as in the proof of Proposition \ref{prop:strsmoothoc}, it can be shown that if $F(a)=\sum_{i=1}^\infty a(i) s_i(y)$, $a=\{a(i)\}\in E$, is a functional that strongly supports $S(x)$, then $S(y)$ is order continuous in $E^{\times}$. By $E^{\times}\neq \ell_\infty$, the space $C_{E^{\times}}$ is well defined, and applying the analogous argument as in the proof of Proposition 2.3 in \cite{czer-kam2010}, we can show that $y$ is an order continuous element of $C_{E^{\times}}$. Finally,  \cite[Proposition 1.5]{czer-kam2010} implies  that $\|y-y_n\|_{C_{E^{\times}}}\to 0$. Consequently, $x$ is a strongly smooth point of $B_{C_E}$ and the proof is complete.
\end{proof}

\begin{example}
The space $C_1$ has no strongly smooth points (see the proof of Theorem\ref{thm:strsmoothce}).

\end{example}

\begin{problem}
Remove the assumption of order continuity of $E$ in Theorems \ref{th:strsmooth} and \ref{thm:strsmoothce}.

\end{problem}

\section{Exposed and strongly exposed points}
 
Given a normed space $(X,\|\cdot\|)$, an element $x\in S_X$ is called an \emph{exposed point} of $B_X$ if there exists a functional $F\in S_{X^*}$ which supports $B_X$ exactly at $x$, i.e. $F(x)=1$ and $F(y)\neq 1$ for every $y\in B_X\setminus\{x\}$. We then say that $F$ \emph{exposes} $B_X$ at $x$.

 Let $x\in S_X$ be an exposed point of $B_X$ and suppose that the functional $F$ exposes $B_X$ at $x$. If  $F(x_n)\to1$ implies $\|x-x_n\|\to 0$ for all sequences $\{x_n\}\subset B_X$, then $x$ is called a \emph{strongly exposed point} of $B_X$  and $F$ \emph{strongly exposes} $B_X$ at $x$.  It is well known that every (strongly) exposed point of $B_X$ is (strongly) extreme \cite{Pietsch}.

Indeed, suppose $x$ is exposed point of $B_X$, and $F\in S_{X^*}$ exposes $x$. Let $x=y_1/2+y_2/2$, for some $y_1,y_2\in B_X$. In view of $|F(y_1)|\leq 1$ and $|F(y_2)|\leq 1$,
\[
2=2F(x)=F(y_1+y_2)=F(y_1)+F(y_2)\leq 2,
\]
and so $F(y_1)=F(y_2)=1$. Since $F$ exposes $x$,  $y_1=y_2=x$ and $x$ is an extreme point of $B_X$.

Let $x$ be a strongly exposed point and $F\in S_{X^*}$ be a functional strongly exposing $x$. Suppose $\|x\pm y_n\|\to 1$ where $\{y_n\}\subset B_X$. Clearly,  $\overline{\lim}_nF(x\pm y_n)\leq 1$. Moreover,
\[
\underline{\lim}_n F(x+y_n)=\underline{\lim}_n (2F(x)-F(x-y_n))=2-\overline{\lim}_nF(x-y_n)\geq 1.
\] 
Thus $1\leq \underline{\lim}_n F(x+y_n)\leq \overline{\lim}_n F(x+y_n)\leq 1$, or $\lim_n F(x+y_n)=1$. Since $F$ strongly exposes $x$, $\|y_n\|=\|x-(x+y_n)\|\to 0$ proving that $x$ is a strongly extreme point of $B_X$.

Exposed points were first defined by  S. Straszewicz in 1935 in the case of finite-dimensional spaces. The concept of strongly exposed points was introduced by J. Lindenstrauss in 1963. There is a connection between strongly exposed points and the Radon-Nikod\'ym property.  R. Phelps showed in 1974 that a Banach space $X$ has the Radon-Nikod\'ym property if every non-empty closed, bounded convex subset is contained in a closed convex hull of its strongly exposed points. In a strictly convex Banach space  all points of its unit sphere are exposed, while in a locally uniformly convex Banach space all points of its unit sphere are strongly exposed. 
 More historical details, as well as references to the facts given above can be found in \cite{Pietsch}.

Exposed and strongly exposed points in noncommutative symmetric spaces were considered first By J. Arazy in \cite{A1} in the case of unitary matrix spaces $C_E$.

\begin{theorem} \cite[Theorem 4.1]{A1}\label{th:arazy}
 Let $E$ be a separable  symmetric sequence  space and $x\in S_{C_E}$. Then $x$ is an exposed (respectively, a strongly exposed) point of $B_{C_E}$ if and only if $S(x)$ is an exposed (respectively, a strongly exposed) point of $B_E$. 
\end{theorem}

Although Theorem 4.1 in \cite{A1} was stated only for $E\neq \ell_1$, it remains true in case of $E=\ell_1$.
As pointed out in \cite{A1}, the sets of extreme, exposed and strongly exposed points coincide in the spaces   $E=\ell_1$ or $C_1$.  These points have the following form
\begin{align*}
\ext {B_{\ell_1}}&=\{\lambda e_n:\,|\lambda|=1, n=1,2,\dots\},\\
\ext {B_{C_1}}&=\{\langle\cdot ,e\rangle f:\, e,f\in\ell_2, \|e\|=\|f\|=1\}.
\end{align*}
 Hence  by Theorem \ref{thm:arazyextreme} on extreme points in $C_E$,  Theorem \ref{th:arazy} is valid  for $E=\ell_1$ as well.

Exposed and strongly exposed points of $\nonsp$ were investigated in \cite{CKK2013}. 

\begin{theorem}\cite[Theorem 3.5, Theorem 4.2]{CKK2013}
\label{thm:1}
Let $E$ be order continuous and $x\in S_{E\Mtau}$. If $\mu(x)$ is an  exposed (respectively, a strongly exposed) point of $B_E$ then $x$ is an  exposed (respectively, a strongly exposed) point of $B_{\nonsp}$.
\end{theorem}

\begin{theorem}\cite[Theorem 3.11,Theorem 4.7]{CKK2013}
\label{thm:2}
Let $E$ be order continuous and $\M$ be non-atomic.
If $x\in S_{E\Mtau}$ is an  exposed (respectively, a strongly exposed) point of $B_{\nonsp}$ then $\mu(x)$ is an exposed (respectively, a strongly exposed) point of $B_E$.
\end{theorem}

Finally we state the main result of this section, which follows from Theorems \ref{thm:1} and \ref{thm:2}. 
\begin{theorem}
\label{cor3}
Let $E$ be order continuous, $\M$ be non-atomic, and $x\in S_{E\Mtau}$. Then $x$ is an exposed (respectively, a strongly exposed) point of $B_{E\Mtau}$ if and only if $\mu(x)$ is an  exposed (respectively, a strongly exposed) point of $B_E$.
\end{theorem}

The next result is an immediate consequence of the previous theorem, taking for $\mathcal{M}$ the commutative von Neumann algebra $L_\infty[0,\alpha)$. 

\begin{theorem} \label{th:noncom1}
Let $E$ be an order continuous symmetric function space.  Then the function $f$ is an exposed (respectively, a strongly exposed) point of $B_E$ if and only if its decreasing rearrangement $\mu(f)$ is an exposed (respectively, a strongly exposed) point of $B_E$.
\end{theorem}

\begin{problem}
Remove the assumption of order continuity of $E$ in Theorems \ref{th:arazy}, \ref{cor3} and \ref{th:noncom1}.

\end{problem}

\section{Kadec-Klee properties}

A Banach space $(X, \|\cdot\|)$ has the \emph{Kadec-Klee} $(KK)$ property if for any $x_n, x\in X$,  whenever $\|x_n\|\to \|x\|$ and $x_n\to x$ weakly then $\|x_n-x\|\to 0$ as $n\to \infty$.
In the literature this property appears  under three different names, Kadec-Klee, Radom-Riesz or $H$ property.   Early on J. Radon in 1913, and F. Riesz in 1929, proved that that property holds for $L_p$ spaces for $1\le p <\infty$. M. I. Kadets and V. L. Klee used some versions of this property to show that all infinite-dimensional separable Banach spaces are homeomorphic \cite{Pietsch}.

Let $(X, \|\cdot\|)$ be a Banach space and $\mathscr{T}$  be a linear topology on $X$ weaker than the norm topology.  We say that  $X$ has the \emph{Kadec-Klee property with respect to $\mathscr{T}$} (for short $X\in (KK(\mathscr{T}))$) if for every $x, x_n\in X$, $x_n\to x$ in $\mathscr{T}$ and $\|x_n\|\to \|x\|$ imply $\|x_n-x\|\to 0$ as $n\to \infty$.

Recall that the collection of sets for $\epsilon, \delta > 0$,
\[
V(\epsilon,\delta)=\{x\in S\Mtau:\,\tau(e^{\abs{x}}(\epsilon,\infty))\leq \delta\}=\{x\in S\Mtau:\,\mu(\delta,x)\leq \epsilon\}
\]
forms a base at zero for the measure topology $\mathscr{T}_m$ on $S\Mtau$. 
The measure topology on $S\Mtau$ can be localized in the following way \cite{DDDLS}. Let $\epsilon,\delta>0$
and  $e\in P(\M)$  with $\tau(e)<\infty$.  Then the family 
\[
V(\epsilon,\delta, e)=\{x\in S\Mtau:\,exe\in V(\epsilon,\delta)\}
\]
 forms a neighborhood base at 0 for a Hausdorff linear topology on $S\Mtau$. This topology is called the topology of
\textit{local convergence in measure} (denoted (\textit{lcm})).  The sequence $\{x_n\}\subset S\Mtau$ converges locally in measure to $x\in S\Mtau$ if $\{ex_ne\}$ converges to $exe$ for the measure topology on $S\Mtau$, for all $e\in P(\M)$ with $\tau(e)<\infty$.

 If $\mathcal{N}$ is a commutative von Neumann algebra, identified with $L_{\infty}[0,\alpha)$, $\alpha\leq \infty$ (see Section \ref{sec:symmfun} for details), then $V(\epsilon,\delta)$ can be identified with the set of functions $f\in L^0[0,\alpha)$ for which  $m\{t\in[0,\alpha):\, |f(t)|>\epsilon\}\leq \delta$.  Hence the measure
topology in $S\Mtau$ corresponds to the usual topology of convergence in
measure in $L^0[0,\alpha)$.  It is also not difficult to verify that given $e=N_{\chi_A}$, $\eta(e)=m(A)<\infty$,  we have that $N_f\in V(\epsilon, \delta, e)$ whenever $m\{t\in A:\, |f(t)|>\epsilon\}\le \delta$.  Hence $N_{f_n}\to 0$ in (lcm) is equivalent with $m\{t\in A:\, |f_n(t)|>\epsilon\}\to 0$ as $n\to \infty$ for all $\epsilon>0$ and all measurable sets $A$ with $m(A)<\infty$.  Hence in the commutative case the local
measure topology corresponds to the usual topology of local
convergence in measure in the space $L^0[0,\alpha)$. 

Recall that the \emph{weak operator topology} on $B(H)$  \cite{KR} is the weak
topology on $B(H)$ induced by the family of linear functionals $w_{\xi,\eta}:B(H)\to \mathbb{C}$ of the form
\[
w_{\xi,\eta}(x)=\langle x\xi, \eta\rangle,\quad \xi,\eta\in H,\, x\in B(H).
\]
Clearly the weak operator topology is weaker that the weak topology on $B(H)$.

 It is known that if  $\M=B(H)$ and $\tau$ is a canonical trace,  then for sequences bounded  in operator norm, convergence in $(lcm)$ is precisely convergence for the weak operator topology \cite{DDDLS}.

Now we are ready to present the results on Kadec-Klee property in $C_E$ and in $\nonsp$. We start with two results by J. Arazy from 1981.

\begin{theorem}\cite[Theorem I]{A3}
\label{thm:arazykk}
Let $E$ be a separable symmetric sequence space. Then $E$ has $KK$ property if and only if $C_E$ has $KK$ property.
\end{theorem}

\begin{theorem}\cite[Theorem II]{A3}
\label{thm:arazykk1}
Let $E$ be a separable symmetric sequence space. The following two statements are equivalent.
\begin{itemize}
\item[{(i)}] If $a=\{a(i)\} \in E$ and $\{a_n\}\subset E$, where $a_n=\{a_n(i)\}$, satisfy $\|a_n\|_E\to \|a\|_E$ and $a_n(i)\to a(i)$ for all $i\in \mathbb{N}$, then $\|a_n-a\|_E\to 0$.
\item[{(ii)}] If $x\in C_E$ and $\{x_n\}\subset C_E$ satisfy $\|x_n\|_{C_E}\to \|x\|_{C_E}$ and $x_n\to x$ in the weak operator topology, then $\|x_n-x\|_{C_E}\to 0$.
\end{itemize}
\end{theorem}
The next result relates $KK(lcm)$ property of $E$ and $C_E$. Note that the componentwise convergence of the sequence $\{a_n\}\subset E$ appearing in condition (i) of Theorem \ref{thm:arazykk1} is equivalent with the local convergence in measure on $E$. Moreover, the convergence in the weak operator topology of the sequence $\{x_n\}\subset C_E$ in condition (ii) of  Theorem \ref{thm:arazykk1} coincides with the topology of local convergence in measure.  It follows the corollary.

\begin{theorem}\cite[Theorem II]{A3}
\label{thm:KKC1}
Let $E$ be a separable symmetric sequence space. Then $E$ has $KK(lcm)$ property if and only if $C_E$ has $KK(lcm)$ property.
\end{theorem}

J. Arazy included a separate proof of \cite[Theorem II]{A3} for the important special case of the trace class $C_1$, which did not involve the elaborate blocking technique. 

The following results on $KK$ properties were established for $\nonsp$ spaces.  It has been shown in \cite[Proposition 1.1]{CDSS} that if a symmetric function space $E$  has either the Kadec-Klee property or the Kadec-Klee property for local convergence in measure, then $E$ is separable. It is worth noting that the latter statement does not remain true  if local convergence in measure is replaced by convergence in measure. This was  demonstrated on the example of Lorentz spaces in \cite[Corollary 1.3]{CDSS}. Therefore if the symmetric function space $E$ has $KK$ then it is separable and by \cite[Theorem 2.7]{CDS} $\nonsp$ has $KK$ property. On the other hand, if $\M$ is non-atomic then $E$ is isometrically embedded in $\nonsp$ by Corollary \ref{cor:isomglobal}, and so it inherits $KK$ property from $\nonsp$. Hence we have the following result.
\begin{theorem}\cite[Theorem 2.7]{CDS}
\label{thm:KK1}
If $E$ has $KK$ property then $\nonsp$ has $KK$ property. If $\M$ is non-atomic then $E$ has $KK$ property if and only if $\nonsp$ has $KK$ property.
\end{theorem}

 Using  similar arguments as in front of Theorem \ref{thm:KK1}, one can state Theorem 2.6 in \cite{DDS1} without assuming that $E$ is separable. Moreover, since every separable symmetric function space is strongly symmetric  \cite{BS,KPS} we have the next theorem.

\begin{theorem}\cite[Theorem 2.6]{DDS1}
If $E$ has $KK(lcm)$ property then $\nonsp$ has $KK(lcm)$ property. If $\M$ is non-atomic then $E$ has $KK(lcm)$ property if and only if $\nonsp$ has $KK(lcm)$ property.
\end{theorem}

The following criteria for norm convergence were established for $C_E$ and $\nonsp$.
Recall again that  in $C_E$ the convergence in weak operator topology is equivalent to the local convergence in measure.
\begin{theorem}\cite[Theorem 3.1]{A1}
Let $E$ be a separable symmetric sequence space. If $x,x_n\in C_E$ then the following are equivalent.
\begin{itemize}
\item[{(i)}] $\|x_n-x\|_{C_E}\to 0$.
\item[{(ii)}] $x_n\to x$ in weak operator topology and $\|S(x_n)-S(x)\|_{E}\to 0$.
\item[{(iii)}] $x_n\to x$ weakly and $\|S(x_n)-S(x)\|_{E}\to 0$.
\end{itemize}
\end{theorem}

\begin{corollary}\cite[Corollary 2.7]{DDS1}
Let $E$ be  order continuous. If $x,x_n\in \nonsp$ then the following are equivalent.
\begin{itemize}
\item[{(i)}] $\|x_n-x\|_{\nonsp}\to 0$.
\item[{(ii)}] $x_n\to x$ (lcm) and $\|\mu(x_n)-\mu(x)\|_{E}\to 0$.
\end{itemize} 

\end{corollary}

The following convergent result was proved in \cite{CSweak} for non-atomic von Neumann algegras, and extended to arbitrary von Neumann algebras in \cite{CKS1992}.

\begin{proposition}\cite[Proposition 3.2]{CSweak} \cite[Proposition 1.1]{CKS1992}
Let $E$ be  order continuous. If $x,x_n\in \nonsp$ then the following are equivalent.
\begin{itemize}
\item[{(i)}] $\|x_n-x\|_{\nonsp}\to 0$.
\item[{(ii)}] $x_n\xrightarrow{\tau} x$  and $\|\mu(x_n)-\mu(x)\|_{E}\to 0$.
\end{itemize} 
\end{proposition}

 The  space $(E,\|\cdot\|_E)$ is said to be \textit{locally uniformly monotone} if for every $\epsilon>0$ and every $0\leq x\in S_E$   there exists $\delta_E(x,\epsilon)>0$ such that $\|x+y\|_E\leq 1+\delta_E(x,\epsilon)$ for $y\in E$ implies that $\|y\|_E<\epsilon$. 
Equivalently, $E$ is locally uniformly monotone whenever for every $x,\{x_n\}\subset E$ if $0\leq x\leq x_n$ for all $n\in \mathbb{N}$, and $\|x_n\|_E\to\|x\|_E$ then $\|x_n-x\|_E\to 0$ as $n\to\infty$.

It turns out that $KK(lcm)$ property is equivalent to  local uniform monotonicity in both  commutative and noncommutative spaces, $E$ and $E\Mtau$, respectively.
\begin{theorem}\cite[Theorem 2.8]{DDS1}  \cite[Theorem 3.2]{CDSS} 
Le $E$ be order continuous and strongly symmetric. Consider the following properties.
\begin{itemize}
\item[{(i)}] $E$ has  $KK(lcm)$ property.
\item[{(ii)}] $\nonsp$ has $KK(lcm)$ property.
\item[{(iii)}] $E$ is locally uniformly monotone.
\item[{(iv)}] $\nonsp$ is locally uniformly monotone.
\item[{(v)}] If $x,x_n\in E$, $0\leq \mu(x)\leq \mu(x_n),\, \|x_n\|_{\nonsp}\to\|x\|_{\nonsp}$ then $\|\mu(x_n)-\mu(x)\|_E\to 0$.
\end{itemize}
The implications $\rm{(i)}\implies\rm{(ii)}$ and $\rm{(i)}\implies\rm{(iii)}\implies \rm{(iv)} $ are always true. If $\M$ is non-atomic,  $\rm{(iv)} \implies\rm{(v)}$ and $\rm{(ii)}\implies\rm{(i)}$. If $E$ is separable then $\rm{(v)}\implies\rm{(iii)}\implies \rm{(i)}$. Consequently, if $\M$ is non-atomic and $E$ is separable then $\rm{(i)}-\rm{(v)}$ are equivalent.
\end{theorem}

\section{Uniform Kadec-Klee property}\label{sec:17}
Let $(X, \|\cdot\|)$ be a Banach space and $\mathscr{T}$   a linear topology on $X$ weaker than the norm topology.  Then $X$ is said to have uniform Kadec-Klee property with respect  to $\mathscr{T}$, denoted by $(UKK)$($\mathscr{T}$), if for every $\epsilon>0$ there exists $\delta\in (0,1)$ such that whenever $x\in X$ and $\{x_n\}\subset B_X$, $x_n\to x(\mathscr{T})$ and $\inf_{n\neq m}\|x_n-x_m\|\geq \epsilon$, then it follows that $\|x\|<1-\delta$.  Equivalently, $X$ has $UKK(\mathscr{T})$ property whenever the $(UKK)(\mathscr{T})$-modulus of  $X$, 
$\delta^{\mathscr{T}}_X(\epsilon) > 0$ for every $\epsilon > 0$, where
\[
\delta^{\mathscr{T}}_X(\epsilon)=\inf\{1-\|x\|:\,x=\lim_n x_n \text{ in }\mathscr{T},\,\|x\|\leq 1, \|x_n\|\leq 1,\, \inf_{n\neq m}\|x_n-x_m\|\geq \epsilon\}.
\]

 The uniform Kadec-Klee property with respect to the local convergence in measure was studied by P. Dodds, T. Dodds and B. De Pagter in 1993 for $E\Mtau$, and  by Y. P. Hsu in 1995 for $C_E$.

\begin{theorem}\cite[Theorem 3.1]{DDS1}
If $E$ has  $UKK(lcm)$  property then $\nonsp$ has  $UKK(lcm)$ property. If $\M$ is non-atomic then  $E$ has  $UKK(lcm)$  property if and only if $\nonsp$ has  $UKK(lcm)$ property.
\end{theorem}
 
Y. P. Hsu gave    estimates for the $UKK(\mathscr{T})$-moduli of spaces $E$ and $C_E$.  Recall that pointwise convergence in  a symmetric sequence space $E$ coincides with the local convergence in measure. Moreover, the convergence in the weak operator topology on $C_E$ is also equivalent with the local convergence in measure. Hence Theorem 3.1 in \cite{Hsu} can be formulated as follows.

\begin{theorem}\cite[Theorem 3.1]{Hsu}
Let $E$ be a symmetric sequence space. If $E$ has $UKK(lcm)$ property then $C_E$ has $UKK(lcm)$ property.  Moreover if $\delta^m_E$ and $\delta^m_{C_E}$ denote the  corresponding $(UKK)(lcm)$ moduli for $E$ and $C_E$ respectively, then $\delta^m_{C_E}(\epsilon)\geq \frac12\delta^m_E\left(\frac{\epsilon^2}{128}\right)$ for $\epsilon > 0$.
\end{theorem}

If $E$ is a symmetric function space then  uniform monotonicity is equivalent to $E$ having $(UKK)(lcm)$ \cite{S2}. P. Dodds, T. Dodds and B. De Pagter  in \cite{DDS1}  showed that uniform monotonicity of $E$ transfers into $\nonsp$. If $\M$ is non-atomic, then using the embedding of $E$ into $\nonsp$, Corollary \ref{cor:isomglobal}, we also have  that if $\nonsp$ is uniformly monotone then so is $E$. This combined with \cite[Theorem 3.5]{HN}, Corollary \ref{cor:nonspcurum}, Theorem \ref{thm:czercur} and Theorem \ref{thm:czercur1} yields the following. For the definition of  the uniform monotonicity we refer to Section \ref{sec:cur}.

\begin{corollary}\cite[Corollary 2.11]{czer-cur}
\label{cor:ukkcurum}
 Consider the following properties.
\begin{itemize}
\item[(1)] $E$ has $UKK(lcm)$ property.
\item[(2)] $\nonsp$ has $UKK(lcm)$
property.
\item[(3)]  $E$ is uniformly monotone.
\item[(4)]  $\nonsp$ is uniformly monotone.
\item[(5)] $E$ is  uniformly monotone.
\item[(6)] $\nonsp^+$ is uniformly monotone.
\item[(7)] $\nonsp$ is complex uniformly convex.
\end{itemize}

We have  $\rm{(1)}\implies \rm{(2)}$, $\rm{(3)}\implies \rm{(4)}$,  $\rm{(5)}\implies\rm{(6)}$ and $\rm{(1)}\iff \rm{(3)}\iff \rm{(5)}$. If $\M$ is non-atomic then  $\rm{(2)}\implies \rm{(1)}$,  $\rm{(4)}\implies \rm{(3)}$ and $\rm{(6)}\implies \rm{(5)}$. Hence if $\M$ is non-atomic,  $\rm{(1)}-\rm{(6)}$ are equivalent. If  $E$ is $p$-convex for some $p>1$ then $\rm{(5)}\implies \rm{(7)}$. Thus if $\M$ is non-atomic and $E$ is $p$-convex for some $p>1$ then all conditions $\rm{(1)}-\rm{(7)}$ are equivalent. 
\end{corollary}

\section{Banach-Saks properties}

In geometry of Banach spaces an important role is played by (weak) Banach-Saks property and its stronger versions like Banach-Saks $p$-property $(BS_p)$ and property $(S_p)$.  It is said that a Banach space $(X, \|\cdot\|)$ satisfies the {\it Banach-Saks property} $(BS)$ if for every bounded sequence $\{x_n\}$ in X, there is a subsequence $\{y_j\}$ such that its Ces\`aro means converge, that is the sequence $\{ \frac 1m \sum_{j=1}^m y_j \}$ is convergent in norm. 

A Banach space $X$ is said to satisfy the {\it weak Banach-Saks property} $(wBS)$ if every weakly null sequence in $X$ has a subsequence such that its Ces\`aro means converge in norm, which implies that these means converge in norm to zero.
 It is well-known that a Banach space has the $BS$ property if and only if it is reflexive and it has the $wBS$ property \cite{Lin}.

W. B. Johnson introduced the following notion in \cite{Joh77}. Given $1<p\le \infty$, a Banach space $X$ has {\it Banach-Saks type p-property} ($pBS$) if every weakly null sequence $\{x_n\}$ in $X$ has a subsequence $\{y_j\}$ such that for some constant $C> 0$ and for all $m\in \mathbb{N}$,
\[
\left\|\sum_{j=1}^m y_j \right\| \le Cm^{1/p}.
\]
Here $m^{1/\infty}=1$ for all $m\in \mathbb{N}$. Clearly if $X$ has $pBS$ property then it has $rBS$ property for any $1<r<p$.

The stronger property ($S_p$) was introduced by H. Knaust and T. Odell in \cite{KO91}. It is said that $X$ has property $(S_p)$, $1<p\le\infty$, if every weakly null sequence $\{x_n\}$ in $X$ has a subsequence $\{y_j\}$ so that there is a constant $C>0$ such that for all $m\in \mathbb{N}$ and all real sequences $a=\{a(n)\}\in \ell_p$,
\[
\left\|\sum_{j=1}^m a(j)y_j \right\| \le C \|a\|_p,
\]
where $\|a\|_p$ is a norm in $\ell_p$.

It is clear that $S_p \implies pBS \implies wBS$ for all $1<p\le \infty$.  The Elton $c_0$-theorem \cite[Theorem III.3.5]{AGRO3} states that $\infty BS \iff S_\infty$. 
 In general, the two properties $pBS$ and $S_p$ are not equivalent if $1<p<\infty$ \cite{Kna92}, however S. Rakov \cite[Theorem 3]{Rak79}  showed that if $1<q<p<\infty$, then  $pBS$ implies $S_q$.

For a Banach space $X$ we define the  following set
\[\Gamma(X)=\{ p\in (1, \infty] : X \text{ satisfies } pBS \text{-property} \}.
\]
Banach-Saks properties and in particular the set $\Gamma(X)$ have been studied in general rearrangement invariant spaces as well as in specific symmetric spaces like Orlicz, Lorentz or Marcinkiwicz spaces (e.g. \cite{DSS2004, Rak79, rak1982}).
Recall that $\Gamma(\ell_p) = (1,p]$ for $1<p<\infty$, and $\Gamma(c_0) = \Gamma(\ell_1) = (1,\infty]$. The space $\ell_\infty$ does not have $(wBS)$ property, so $\Gamma(\ell_\infty) = \emptyset$. Any separable sequence Orlicz space $\ell_\phi$, or a separable part $h_\phi$ of a nonseparable Orlicz space $\ell_\phi$ has $wBS$ \cite{rak1982}. For a sequence Orlicz space $\ell_\phi$, whenever $\ell_\phi$ is reflexive, we have that $(1,\alpha^0_\phi) \subset \Gamma(\ell_\phi) \subset (1, \alpha^0_\phi]$, where $\alpha^0_\phi$ is the lower Matuszewska-Orlicz index around zero. Moreover $BS_p$ and $S_p$ are equivalent in $\ell_\phi$ for any $1<p\le \infty$ \cite{Kna92}.  In \cite{KL}  the similar result was also proved in Musielak-Orlicz sequence space $\ell_\Phi$. It was also shown there that $\ell_\Phi$ has the $wBS$ property if and only if it is separable,  and that the Schur  and $\infty BS$ properties coincide in $\ell_\Phi$.

The main results on Banach-Saks properties in spaces $E\Mtau$ or $C_E$ are contained in \cite{A1981, DDS2, L-PS}.

Some methods used   by P. Dodds, T. Dodds and F. Sukochev in noncommutative spaces \cite{DDS2} in 2007, are generalizations of the analogous methods in function spaces. 
The following variant of Kadec-Pe{\l}czy\'nski result holds in noncommutative spaces.

\begin{proposition}\label{prop:KP}\cite[Corollary 2.5]{DDS2}
Let $Y\subset E\Mtau$ be a closed subspace. Then either $\rm{(i)}$ the norm topology from $E\Mtau$ on $Y$ coincides with the measure topology, or $\rm{(ii)}$ there exist $\{y_n\}\subset Y$ with $\|y_n\|_{E\Mtau} \le 1$, and a two-sided disjointly supported sequence $\{d_n\} \subset E\Mtau$ such that $\|y_n - d_n \|_{E\Mtau} \to 0$.

\end{proposition}

It is also shown there that in $E\Mtau$ the subsequence splitting principle is satisfied, that is for each bounded sequence in $E\Mtau$ there is a subsequence which is approximated in norm by the sum of two sequences, from which one consists of equimeasurable elements and another one contains two-sided disjoint operators. This implies a 
noncommutative analogue of the Koml\'os theorem.

\begin{theorem}\label{th:komlos} \cite[Theorem 2.8]{DDS2}
Assume $E$ has the Fatou property and $\{x_n\} \subset E\Mtau$ is  bounded. Then there exists $y\in E\Mtau$ and a subsequence $\{y_n\} \subset \{x_n\}$ such that for any further subsequence $\{z_n\} \subset \{y_n\}$, $\lim_n \sum_{k=1}^n z_k/n = y$ in measure topology. 

\end{theorem}

\begin{theorem}\label{th:BS1}\cite[Theorem 2.13]{DDS2}
Let $E$ satisfy the Fatou property and let $\M$ be non-atomic.Then $E\Mtau$ has $wBS$ property if and only if each weakly null two sided disjoint sequence $\{x_n\}\subset E\Mtau$ contains a subsequence $\{y_n\}$ such that the Ces\`aro means of any $\{z_n\}\subset \{y_n\}$  tend to zero in norm.

\end{theorem}

The next result lifts $wBS$ property from $E$ to $E\Mtau$.

\begin{theorem}\label{th:BS2}\cite[Theorem 2.14]{DDS2}
Let $E$ satisfy the Fatou property. If $E$ has $wBS$ property then $E\Mtau$ has also this property. If in addition $\M$ is non-atomic then  the converse statement holds  true.
\end{theorem}


In the case of the unitary ideals $C_E$ a stronger version was proved by J. Arazy in 1981.  

\begin{theorem}\cite[Corollary 3.6]{A1981}\label{th:BS-C_E}
Let E be a symmetric separable sequence space. Then $C_E$ has the $BS$ property (respectively, the $wBS$ property) if and only if $E$ has the $BS$ property (respectively, the $wBS$ property).
\end{theorem}

The main result on $BS_p$ property requires additional assumptions on convexity and concavity of $E$. 

\begin{theorem}\cite[Proposition 3.2]{DDS2} \label{th:pBS1}
 If $E$ is $p$-convex and $q$-concave for some $1 < p \le 2 \le q < \infty$, then $E\Mtau$ has the $p$-Banach-Saks property.
 \end{theorem}
 
\begin{proof}  Under our assumptions on $p,q$ and $E$, there exists on $E$ an equivalent symmetric norm with moduli of $p$-convexity and $q$-concavity both equal to $1$ (compare \cite[Proposition 1.d.8]{LT2}). Thus we may assume that $E$ has these both moduli equal to $1$.  It then follows from \cite{X1989} that $E\Mtau$ has type $p$.  We complete the proof by  the result in  \cite{rak1982} stating  that if a Banach space is of type $p$, $1 < p \le 2$, then it has the $p$-Banach-Saks property.
\end{proof}

Since the space $L^p$, $1\le p < \infty$, is $p$-convex and $p$-concave, it follows from Theorem \ref{th:pBS1} that the noncommutative space $L^p\Mtau$ has $\min\{p,2\}BS$ property.

 The next results are strictly related to Banach-Saks properties. The first one is the extension of classical  Schlenk  theorem.
 
 The noncommutative version of equiintegrability is defined as follows. For a bounded set $K \subset E\Mtau$ we say that $K$ is  $E$-\emph{equiintegrable} if $sup_{x\in K}\{\|e_n x e_n\|_{E\Mtau}\}\to 0$
for every system $\{e_n\}\subset\M$ of projections with $e_n \downarrow  0$.

\begin{theorem}
\cite[Corollary 3.7]{DDS2}\label{th:Schlenk} 
Suppose that $E$ has the Fatou property and that $\M$ is non-atomic. If $\{x_n\} \subset E\Mtau$ is weakly null and 
$E$-equiintegrable, then $\{x_n\}$ contains a Banach-Saks subsequence $\{y_n\}$, that is $\lim_m m^{-1} \|\sum_{j=1}^m z_j\|_{E\Mtau} =0$ for every subsequence $\{z_j\} \subset \{y_n\}$.
\end{theorem}

\begin{theorem}\cite[Theorem 3.9]{DDS2}\label{th:pBS2}
 Suppose that $\M$ is non-atomic. If $E$ is $p$-convex and $q$-concave for some $1 < p < 2 \le q < \infty$, then each weakly null, $E$-equiintegrable sequence $\{x_n\}$ in $E\Mtau$ contains a strong $p$-Banach-Saks subsequence $\{y_n\}$, that is $\lim_m m^{-1/p}\|\sum_{j=1}^m z_j\|_{E\Mtau} =0$ for any subsequence $\{z_j\}\subset \{y_n\}$.
\end{theorem}

Given a closed subspace $X$ of $E\Mtau$, a sequence $\{x_n\} \subset X$ is called almost disjointly supported 
if there exists a two sided disjointly supported sequence $\{y_n\}\subset E\Mtau$ such that $\|x_n - y_n\|_{E\Mtau} \to 0$.

\begin{proposition}
\cite[Proposition 3.12]{DDS2}
 Suppose that $E$ is $p$-convex and $q$-concave for some $1 < p < 2 \le q < \infty$,  $\M$ is non-atomic and  $X \subset E\Mtau$ is a closed linear subspace. If $X$ does not have the strong $p$-Banach-Saks property, then $X$ contains a seminormalised almost disjointly supported sequence which converges to zero in measure.
\end{proposition}

The Boyd indices of a symmetric space $E$  on $[0,\alpha)$, $0<\alpha\le \infty$ are nontrivial, i.e., $1 < p_E \le q_E < \infty$ if and only if $E$ is an interpolation space between $L_r[0,\alpha)$ and $L_q[0,\alpha)$  for some $1<r<q<\infty$ \cite{LT2}. For such spaces $E$ more specialized Banach-Saks properties have been studied in \cite{L-PS}. As an example let us state the  result below.
We say that $E$ has the \emph{disjoint $pBS$ property}  if every weakly null disjointly supported sequence in $E$ has a $p$-Banach-Saks subsequence. Clearly if $E$ satisfies an upper $p$-estimate, then E has disjoint $pBS$ property.

\begin{theorem}\cite[Theorem 12]{L-PS}
 Let $E$ be a symmetric separable space which  is interpolation  between $L_r[0,\tauone)$ and $L_q[0,\tauone)$, $1<r<q<\infty$. Let $\beta = \min\{r, 2\}$. Assume also that $E$ has the disjoint $\beta BS$ property and that $\M$ is non-atomic. Then $E\Mtau$ has the $\beta BS$ property as soon as either $\M = R$ is hyperfinite or $E$ is $D^*$−convex.
\end{theorem}

\begin{problem} $\rm{(i)}$ Characterize  property $S_p$ for  spaces $E\Mtau$ and $C_E$. 

$\rm{(ii)}$ Find the relationship bewteen the intervals $\Gamma(E\Mtau)$ or $\Gamma(C_E)$, and the interval $\Gamma(E)$ . 
\end{problem}

\section{Radon-Nikod\'ym property}

Let $X$ be a Banach space and $K\subset X$ be closed, bounded and convex. Then $K$ is said to have the \emph{Radon-Nikod\'ym property} ($RNP$)  if for any finite measure space $(\Omega, \Sigma,\mu)$ and any $X$-valued measure $m$ on $\Sigma$ that is absolutely continuous with respect to $\mu$, $m(A)/\mu(A) \in K$ for all $A\in \Sigma$ with $\mu(A) > 0$ implies that there is an $f \in L_1(\mu,X)$ such that for all $A\in \Sigma$, $m(A) = \int_A f \,d\mu$. We say that $X$ has $RNP$ whenever every closed, bounded and convex  subset of $X$ has the $RNP$. It is well known that the spaces $L_1[0,1]$ and $c_0$ do not have the $RNP$, and therefore any Banach space that contains a subspace isomorphic to either $L_1$ or $c_0$ do not possess the $RNP$. On the other hand  every reflexive space or a space which is dual and separable  has the $RNP$. We refer to the book by Pei-Kee Lin \cite{Lin} for details on the Radon-Nikod\'ym property.

Q. Xu proved the following result in 1992.

\begin{theorem} \cite{X1992}
Let  $E$ have the $RNP$. Then $E\Mtau$ has the $RNP$. Similarly if $E$ is a symmetric sequence space with the $RNP$ then the unitary ideal $C_E$ has the $RNP$. 
\end{theorem}
\begin{proof}
We give only a sketch of the proof.  If $E$, a function or sequence space, has the $RNP$ then $E$ cannot contain an order isomorphic copy of $c_0$, and it follows from Proposition \ref{prop:czero} that $E$ is order continuous (equivalently separable), and it satisfies the Fatou property. From the Fatou property we have that $E^{\times\times}=E$. Observe that for any symmetric sequence space $E$ the subspace $E_a$ is always non-trivial since any unit  vector $\phi_n$ belongs to $E_a$.  Recall that $\phi_n=\{\phi_n(i)\}$ with $\phi_n(i) = 0$ if $i\ne n$ and $\phi_n(n) =1$, $n\in\mathbb{N}$. Therefore if $E$ is a symmetric sequence space then $F=(E^\times)_a$  is a non-trivial symmetric sequence space  and such that $F^* = [(E^\times)_a]^* = (E^\times)^\times =E$. Now since $F$ is separable, Theorem 12.2 in  \cite{GK} implies that $(C_F)^* = C_{F^*}$.  Thus we get
\[
C_E = (C_F)^*,
\]
which means that $C_E$ is a dual space.  If in addition we assume that the Hilbert space $H$ is separable then $C_E$ must be separable \cite[Proposition 1, Theorem 2]{Medz},  and so it must satisfy the $RNP$. In the case when $H$ is not separable the proof for $C_E$ or $E\Mtau$  goes along the similar line but it is more involved.  In particular the spaces $E\Mtau$ or $C_E$ do not need to be separable  even though $E$ is separable. 
\end{proof}

In view of Corollary \ref{cor:isomglobal} that states when $E$ is isomorphically embedded either in $E\Mtau$ or $C_E$,  we get the converse  of the above result.

\begin{theorem} If $\M$ is non-atomic and $E\Mtau$ has the $RNP$ then $E$ has also this property. For any symmetric sequence space $E$, if $C_E$ has the $RNP$ then $E$ has this property too.  
\end{theorem}

Let $\{e_{i,j}\}$ denote the sequence of standard unit matrices for $i,j\in\mathbb{N}$, that is $e_{i,j}(k,l) = \delta_{i,k} \cdot \delta_{j,l}$,  where $\delta_{i,k}=0$ if $i\ne k$ and 
$\delta_{i,i} =1$ for $i,k\in\mathbb{N}$.
 The  $n$'th shell-subspace is defined as $S_n = \text{span} \{e_{i,j}: \max\{i,j\} = n\}$, $n\in\mathbb{N}$. The sequence $\{S_n\}$ is called the shell decomposition, and it is a monotone Schauder decomposition for $C_E$. For details on Schauder bases in Banach spaces we refer to \cite{guerre, LT1}. We finish  with a list of equivalent conditions for $RNP$ of $C_E$ in the case when $E$ is separable, due to J. Arazy. 

\begin{theorem} \cite[Proposition 3.7]{A1981} The following eight properties are equivalent for every symmetric separable sequence space $E$.
\begin{itemize}
\item[$(1)$] $E$ does not contain a subspace isomorphic to $c_0$. 
\item[$(2)$] $E$ has the $RNP$.
\item[$(3)$] The unit vector basis of $E$ is boundedly complete. 
\item[$(4)$] $E$ is a dual space.
\item[$(5)$] $C_E$ does not contain a subspace isomorphic to $c_0$.
\item[$(6)$] $C_E$ has the $RNP$.
\item[$(7)$] The shell decomposition of $C_E$ is boundedly complete. 
\item[$(8)$] $C_E$ is a dual space.
\end{itemize}
\end{theorem}

\section{Stability in the sense of Krivine-Maurey}

J. L. Krivine and B. Maurey introduced the notion of stable Banach spaces in    \cite{KM1981} in 1981. They proved that any stable space contains an almost isometric subspace of $\ell_p$, $1\le p < \infty$. They also proved that any subspace of $\ell_p$ or $L_p[0,\alpha)$, $0<\alpha\le \infty$, is stable. It is well known that any finite-dimensional or Hilbert space is stable, while $c_0$ is not stable \cite{guerre}. The Orlicz-Bochner space $L_\varphi(X)$ over the probability measure space is stable whenever $\varphi$ satisfies condition $\Delta_2$ , that is the Orlicz space $L_\varphi$ is separable,  and a Banach space $X$ is separable and stable \cite[Theorem 16]{G}. The Bochner-Lorentz spaces $L_{p,q}(X)$, $1\le p,q<\infty$, are stable if the Banach space $X$ is stable \cite{Ray1981}. A generalization of this result to the Lorentz space $\Lambda_{p,w}$ for $1\le p <\infty$ and the weight $w$ decreasing,  has been done by Yves Raynaud in his doctoral thesis. 

Since $c_0$ is not stable, a symmetric space $E$ cannot be stable whenever it contains an isomorphic copy of $c_0$. Thus any stable space $E$ must be separable by Proposition \ref{prop:infinity}. 

 Let $E$ be an order continuous symmetric Banach sequence space. Then $\{\phi_n\}$ forms a  symmetric basis in $E$. In \cite{A1983},  J. Arazy studied basic sequences in unitary matrix spaces $C_E$. In Theorem 2.4 and Corollary 2.8 in \cite{A1983} it was proved that every basic sequence in $C_E$ has a subsequence equivalent to a basic sequence in $\ell_2 \oplus E$. This result and its several variants is a powerful method in reducing the studies of properties that depend on asymptotic behavior of sequences in $C_E$ to analogous properties in $E$. Stability in the sense of Krivine-Maurey is one of such properties.

\begin{definition} \label{def:stab} A Banach space $(X, \|\cdot\|)$ is said to be {\it stable} (in the sense of Krivine-Maurey) if for every pair $\{x_n\}$ and $\{y_n\}$ of bounded sequences in $X$ and for every pair of ultrafilters $\mathcal{U}$ and $\mathcal{V}$ on the set of natural numbers $\mathbb{N}$, one has
\begin{equation}
\lim_{m,\mathcal{V}}\, \left(\lim_{n,\mathcal{U}} \|x_n + y_m\|\right) = \lim_{n,\mathcal{U}}\, \left(\lim_{m,\mathcal{V}} \|x_n + y_m\|\right).
\end{equation}
\end{definition}

The next result by J. L. Krivine and B. Maurey  states the equivalent condition for stability which does not use ultrafilters.

\begin{proposition}\cite{KM1981} \label{prop:stab}
A Banach space $(X,\|\cdot\|)$ is stable if and only if for every pair $\{x_n\}$, $\{y_n\}$ of bounded sequences in $X$,
\begin{equation}
\inf_{n>m} \|x_n + x_m\| \le \sup_{n<m} \|x_n + x_m\|.
\end{equation}
\end{proposition}

\begin{proposition}\cite[Theorem 1]{GL}\label{prop:wsc}
Every stable Banach space is weakly sequentially complete.
\end{proposition}

The next corollary follows from Proposition 3.7 in \cite{A1981} and the fact that $c_0$ is not weakly sequentially complete. 

\begin{corollary}\label{cor:shell}
 Let $E$ be a symmetric separable sequence space. If $E$ is stable then $C_E$ does not contain a subspace isomorphic to $c_0$, and the shell decomposition of $C_E$ is boundedly complete.
\end{corollary}

The main result lifting the property of stability from $E$ to $C_E$ was obtained independently by J. Arazy \cite{A1983} and Y. Raynaud \cite{Ray1982}.

\begin{theorem}
Let $E$ be a symmetric separable  sequence  space. Then $E$ is stable if and only if $C_E$ is stable. 
\end{theorem}

Here we use the fact that $E$ is an isometric subspace of $C_E$ (see Proposition \ref{prop:isomarazy}) and clearly stability is inherited by subspaces, so if $C_E$ is stable then $E$ is stable. The non-trivial proof is in the opposite direction. The proof is based on Propositions \ref{prop:stab}, \ref{prop:wsc}, Corollary \ref{cor:shell} and two technical lemmas.  

In 1997 Marcolino Nhany  found a necessary and sufficient condition for stabilty of noncommutative  $L_p(\mathcal{M}, \tau)$ spaces.

\begin{theorem}\cite[La Th\'eor\`eme principale]{MN} The following properties are equivalent.
\begin{itemize}
\item[$(1)$] The von Neumann algebra $\mathcal{M}$ is of type I.
\item[$(2)$]    $L_p(\mathcal{M}, \tau)$ is stable for all $1\le p < \infty$. 
\item[$(3)$] There exists $1\le p < \infty$, $p\ne 2$, such that $L_p(\mathcal{M}, \tau)$ is stable.
\end{itemize}
\end{theorem}

\begin{remark}
The semifinite von Neumann algebra $\mathcal{M}$ is always of type $I$ or type $II$. For precise definition of types of von Neumann algebra we refer to \cite{Takesaki}.
\end{remark}

\begin{problem}

Assume $E$ is stable and $\mathcal{M}$ is of type $I$. Is the space $E\Mtau$ stable?

\end{problem}

\bibliographystyle{amsplain}

\begin{thebibliography}{10} 

\bibitem{AGRO3}
S. A. Argyros, G. Godefroy and H. P. Rosenthal, \emph{Descriptive set theory and Banach spaces}, Handbook of the geometry of Banach spaces, Vol. 2, North-Holland, Amsterdam, 2003, pp. 1007--1069.


\bibitem{AAP1982}
A. Akemann, J. Anderson and G.K. Pederson, \emph{Triangle inequalities in operator algebras}, Linear and Multilinear Algebra \textbf{11}~(1982), 167--178.

\bibitem{AB}
C. D. Aliprantis and O. Burkinshaw, \emph{Positive Operators},  Academic Press Inc., Orlando, FL,  1985.

\bibitem{A1}
J. Arazy, \emph{On the geometry of the unit ball of unitary matrix
  spaces}, Integral Equations Operator Theory \textbf{4}~(1981), no.~2, 151--171.
  
  \bibitem{A1981}
J. Arazy, \emph{Basic sequences, embeddings, and the uniqueness of the symmetric structure on unitary matrix spaces}, J. Funct. Anal. \textbf{40}~(1981), 301--340.


\bibitem{A3}
J. Arazy, \emph{More on convergence in unitary matrix spaces},  Proc. Amer. Math. Soc. \textbf{83}~(1981), no. 1, 44--48.

\bibitem{A1983}
J. Arazy, \emph{On stability of unitary matrix spaces}, Proc. of Amer. Math. Soc. \textbf{87}~(1983), 317--321.





\bibitem{AL1985} J. Arazy and P. K. Lin,  \textit{On $p$-convexity and $q$-concavity of unitary matrix spaces}, Integral Equations Operator Theory \textbf{8}~(1985), no. 3, 295--313.


\bibitem{BFLM} P. Bandyopadhyay, V. P. Fonf, B. L. Lin and M. Mart{\'{\i}}n, \emph{Structure of nested sequences of balls in {B}anach spaces}, Houston J. Math. \textbf{29}~(2003), no.~1, 173--193.

\bibitem{BS}
C. Bennett and R. Sharpley, \emph{Interpolation of Operators}, Pure and
  Applied Mathematics, vol. \textbf{129}, Academic Press Inc., Boston, MA, 1988.
  
\bibitem{Carth} C. A.  McCarthy,   $C_p$, \emph{Israel J. Math.} \textbf{5}~ (1967), 249--271.

\bibitem{Chen}
S. Chen, \emph{Geometry of {O}rlicz Spaces},  {Dissertationes Math. (Rozprawy Mat.)} \textbf{356}, (1996).

\bibitem{CCH}
L. Chen,  Y. Cui and H.  Hudzik, \emph{Criteria for complex strongly extreme points of Musielak-Orlicz function spaces},  Nonlinear Anal. \textbf{70}~(2009), no. 6, 2270--2276. 

  \bibitem{CDS}
V. I. Chilin, P.~G. Dodds and F. A.  Sukochev, \emph{The Kadec-Klee property in symmetric spaces of measurable operators}, Israel J. Math. \textbf{97}~(1997), 203--219.

\bibitem{CDSS}
V. I. Chilin, P.~G. Dodds, A. A. Sedaev and F. A.  Sukochev, \emph{Characterisations of Kadec-Klee properties in symmetric spaces of measurable functions}, Trans. Amer. Math. Soc. \textbf{348} (1996), no. 12, 4895--4918. 

\bibitem{CKS1992ext}
V.~I. Chilin, A.~V. Krygin and F.~A. Sukochev, \emph{Extreme
  points of convex fully symmetric sets of measurable operators}, Integral
  Equations Operator Theory \textbf{15}~(1992), no.~2, 186--226.

\bibitem{CKS1992}
V.~I. Chilin, A.~V. Krygin and F.~A. Sukochev, \emph{Local uniform and uniform convexity of noncommutative symmetric
  spaces of measurable operators}, Math. Proc. Cambridge Philos. Soc.
  \textbf{111}~(1992), no.~2, 355--368.

  \bibitem{CSweak} V. I. Chilin and F. A. Sukochev, \textit{Weak convergence in non-commutative symmetric
spaces.}, J. Operator Theory \textbf{31}~(1994), 35--65.


  
\bibitem{doctth}  M. M. Czerwi\'nska, \emph{Geometric properties of symmetric spaces of measurable operators}, (Order No. 3476359, The University of Memphis), ProQuest Dissertations and Theses, \textbf{131}, 2011;  http://search.proquest.com/docview/893659843.

\bibitem{czer-cur}  
M.~M. Czerwi\'nska, \emph{Complex uniform rotundity in symmetric spaces of measurable operators}, J. Math. Anal. Appl. \textbf{395}~(2012), no. 2, 501--508. 


\bibitem{czer-kam2010}
M. M. Czerwi\'nska and A. Kami\'nska \emph{Complex rotundity properties and midpoint local uniform rotundity in symmetric spaces of measurable operators}, Stud. Math. \textbf{201}~(2010), no. 3, 253--285.


\bibitem{czer-kam2015} M.~M. Czerwi\'nska and A. Kami\'nska, \emph{$k$-extreme points in symmetric spaces of measurable operators}, Integral Equations Operator Theory \textbf{82} (2015), no. 2, 189--222.

\bibitem{CKK2012}
M. M. Czerwi\'nska, A. Kami\'nska and D. Kubiak, \emph{Smooth and strongly smooth points in symmetric spaces of measurable operators}, Positivity \textbf{16}~(2012), no. 1, 29--51.

\bibitem{CKK2013} M.~M. Czerwi\'nska, A. Kami\'nska and D. Kubiak, \emph{Exposed and strongly exposed points in symmetric spaces of measurable operators}, Houston Journal of Mathematics \textbf{39}~(2013), no. 3, 823--852.

\bibitem{DGT1984}
William~J. Davis, D.~J.~H. Garling and Nicole Tomczak-Jaegermann, \emph{The
  complex convexity of quasinormed linear spaces}, J. Funct. Anal. \textbf{55}~(1984), no.~1, 110--150.


\bibitem{Godefroy}
R. Deville, G. Godefroy and V. Zizler, \emph{Smoothness and Renormings in Banach Spaces}, Longman Group UK Limited 1993.
  
  \bibitem{Diestel}
J. Diestel, \emph{Geometry of Banach Spaces - selected topics}, Lectures Notes in Math. \textbf{485}, Springer-Verlag, Berlin, Heidelberg, New York, 1975.

 \bibitem{Dil}  S. J. Dilworth, \emph{Complex convexity and the geometry of Banach spaces}, Math. Proc. Cambridge Philos. Soc. \textbf{99} (1986), no. 3, 495--506.


 \bibitem{Dix} J. Dixmier, \textit{ Formes lin\'eaires sur un anneau d'op\'erateurs}, (French) Bull. Soc. Math. France \textbf{81}~(1953), 9--39.
 
 \bibitem{DD1993}
P.~G. Dodds and T.~ K. Dodds, \emph{Some aspects of the theory of symmetric operator spaces},
First International Conference in Abstract Algebra (Kruger Park, 1993), Quaestiones Math. \textbf{18}~(1995), no. 1-3, 47--89.


  \bibitem{DDDLS}
P.~G. Dodds, T.~ K. Dodds, P.~N. Dowling, C.~J.Lennard and F. A.
Sukochev, \emph{A uniform Kadec-Klee property for symmetric operator spaces},
Math. Proc. Camb. Phil. Soc. \textbf{118}~(1995), 487--502.

\bibitem{DDP2011}
P.~G. Dodds and B.~De~Pagter, \emph{Properties $(u)$ and $(V^*)$ of Pelczynski in spaces of $\tau$-measurable operators}, Positivity \textbf{15}~(2011), 571--594.


\bibitem{DDP2014}
P.~G. Dodds and B.~De~Pagter, \emph{Normed K\"̈othe spaces: A non-commutative view-point}, Indag. Math \textbf{25}~(2014), 206--249.


\bibitem{DDPMark}
P.~G. Dodds, T.~ K. Dodds and B.~De~Pagter, \emph{A general Markus inequality}, Proc. Centre Math. Anal. Austral. Nat. Univ. \textbf{24} (1989), 47--57.

\bibitem{DDPnoncomm}
P.~G. Dodds, T.~ K. Dodds and B.~De~Pagter, \emph{Noncommutative
  {B}anach function spaces}, Math. Z. \textbf{201}~(1989), no.~4, 583--597.
  
 \bibitem{DDP4}
P.~G. Dodds, T.~K. Dodds and B. De~Pagter, \emph{Noncommutative {K}\"{o}the duality}, Trans. Amer. Math. Soc. \textbf{339} (1993),
  no.~2, 717--750.

  \bibitem{DDS1}
P.~G. Dodds, T.~ K. Dodds and F. A.  Sukochev, \emph{Lifting of Kadec-Klee properties to symmetric spaces of measurable operators},  Proc. Amer. Math. Soc. \textbf{125}~(1997), no. 5,  1457--1467.

 \bibitem{DDS2}
P.~G. Dodds, T.~ K. Dodds and F. A.  Sukochev, \emph{Banach-Saks properties in symmetric spaces of measurable operators}, Studia Math. \textbf{178}~(2007), no. 2, 125--166.

\bibitem{DDS3}
P.~G. Dodds, T.~ K. Dodds and F. A.  Sukochev, \emph{On $p$-convexity and $q$-concavity in non-commutative symmetric spaces}, Integral Equations Operator Theory \textbf{78}~(2014), no. 1, 91--114.

 
 \bibitem{noncomm}
P.~ G. Dodds, B. De~Pagter and F.~A. Sukochev, \emph{Theory of
  Noncommutative Integration}, unpublished monograph; to appear.
  
\bibitem{DSS2004}  P. G. Dodds, E. M. Semenov and F. A. Sukochev, \emph{The Banach-Saks property in rearrangement invariant spaces}, Studia Math. \textbf{162}~(2004), no. 3, 263--294.
  
 \bibitem{P}
B. De~Pagter, \emph{Non-commutative Banach function spaces}, Positivity,
 Trends Math., Birkh\"auser, Basel, 2007, pp.~197--227.

  
\bibitem{Fack1987} T. Fack, \textit{Type and cotype inequalities for noncommutative $L_p$-spaces.} J. Operator Theory \textbf{17}~(1987), no. 2, 255--279. 

\bibitem{Fack-Kos1986} T. Fack and H. Kosaki, \textit{Generalized $s$-numbers of $\tau$-measurable operators}, Pacific
J. Math. \textbf{123}~(1986), no. 2, 269--300.

\bibitem{Fin}
C. Finet, \emph{Uniform convexity properties of norms on a
  super-reflexive {B}anach space}, Israel J. Math. \textbf{53} (1986), no.~1,
  81--92.
  
\bibitem{G}
D. J. H. Garling, \emph{Stable Banach spaces, random measures and Orlicz function spaces}, Probability measures on groups (Oberwolfach, 1981), pp. 121--175, Lecture Notes in Math., \textbf{928}, Springer, Berlin-New York, 1982. 


\bibitem{Gar-Tom} D. J. H. Garling and N. Tomczak-Jaegermann, \emph{ The cotype and uniform convexity of unitary ideals}, Israel J. Math. \textbf{45}~(1983), no. 2-3, 175--197.

\bibitem{Glob}
J.~Globevnik, \emph{On complex strict and uniform convexity}, Proc. Amer. Math. Soc. \textbf{47}~(1975), 175--178.
  


\bibitem{GK}
I. C.~Gohberg and M. G.~Krein, \emph{Introduction to the Theory of Linear Nonselfadjoint Operators}, American Mathematical Society, Providence 1969.

\bibitem{Gro}
A.~Grothendieck, \emph{R\'earrangements de fonctions et in\'egalit\'es de convexit\'e dans les alg\'ebres de von Neumann munies d'une trace}, S\'em. Bourbaki,   Vol. 3, Exp. No. 113, 127--139, Soc. Math. France, Paris, 1995.

\bibitem{guerre}
S. Guerre-Delabri{\`e}re, \emph{Classical Sequences in Banach Spaces},
 Monographs and Textbooks in Pure and Applied Mathematics, \textbf{166}, Marcel Dekker, Inc., New York, 1992. 
 
 \bibitem{GL}
 S. Guerre and J. T. Laprest\'e, \emph{Espaces de Banach stables}, Israel J. Math. \textbf{39}~(1981), 273--295. 
 
 \bibitem{Hag}
U. Haagerup, \emph{Non-commutative integration theory}, Lecture given at the Symposium in Pure Mathematics of the American Mathematical Society, Queens University, Kingston, Ontario (1980).


\bibitem{Holub} J. R. Holub, \textit{On the metric geometry of ideals of operators on Hilbert space}, Math. Ann. \textbf{201}~(1973), 157--163.

\bibitem{HN}
H. Hudzik and A. Narloch, \emph{Relationships between monotonicity and
  complex rotundity properties with some consequences}, Math. Scand.
  \textbf{96}~(2005), no.~2, 289--306.
 
 
\bibitem{Hsu} Yu-Ping Hsu,\emph{The lifting of the UKK property from $E$ to $C\sb E$}, Proc. Amer. Math. Soc. \textbf{123}~(1995), no. 9, 2695--2703.

\bibitem{Joh77}
W.~B. Johnson, \emph{On quotients of {$L_{p}$} which are quotients of
  {$l_{p}$}}, Compositio Math. \textbf{34} (1977), no.~1, 69--89.
  
  
\bibitem{JP}
M. Junge and J. Parcet, \emph{Rosenthal's theorem for subspaces
of noncommutative $L\sb p$}, Duke Math. J. \textbf{141}~(2008), no.
1, 75--122.


    \bibitem{KR}
R.~V. Kadison and J.~R. Ringrose, \emph{Fundamentals of the Theory of
  Operator Algebras. {V}ol. {I}}, Graduate Studies in Mathematics, vol.~15,
  American Mathematical Society, Providence, RI, 1997.
  
  \bibitem{KS}
N.~J. Kalton and F.~A. Sukochev, \emph{Symmetric norms and spaces of
  operators}, J. Reine Angew. Math. \textbf{621}~(2008), 81--121.
  
  \bibitem{KL}
  A. Kami\'nska and H. J. Lee, \emph{Banach-Saks properties of Musielak-Orlicz and Nakano sequence lattices}, Proc. Amer. Math. Soc. \textbf{142} (2014), no. 2, 547--558.
  
  \bibitem{KLL}
  A.  Kami\'nska, H. J. Lee and G. Lewicki, \emph{Extreme and smooth points in Lorentz and Marcinkiewicz spaces with applications to contractive projections}, Rocky Mountain J. Math. \textbf{39} (2009),  no. 5, 1533--1572.
  
  \bibitem{KP}
A. Kami{\'n}ska and A. Parrish, \emph{Note on extreme points in {M}arcinkiewicz function spaces}, Banach J. Math. Anal., \textbf{4} (2010), no.~1, 1--12.

  \bibitem{KR2009}
  A. Kami\'nska and Y. Raynaud, \emph{Isomorphic copies in the lattice $E$ and its symmetrization $E^{(*)}$ with applications to Orlicz-Lorentz spaces}, J. Funct. Analysis \textbf{257}~(2009), no. 1, 271--331.
  

  
  \bibitem{KA}
L. V.  Kantorovich and G. P.  Akilov, \emph{Functional Analysis},
Pergamon Press and "Nauka", Second Edition, 1982.


\bibitem{Kna92}
H.~Knaust, \emph{Orlicz sequence spaces of {B}anach-{S}aks type}, Arch.
  Math. (Basel) \textbf{59} (1992), no.~6, 562--565.
  
  \bibitem{KO91}
H.~Knaust and E.~Odell, \emph{Weakly null sequences with upper
  {$l_p$}-estimates}, Functional analysis ({A}ustin, {TX}, 1987/1989), Lecture
  Notes in Math., vol. 1470, Springer, Berlin, 1991, pp.~85--107.

  
\bibitem{KPS}
S.~G. Kre{\u\i}n, Yu.~{\=I}. Petun{\={\i}}n and E.~M. Sem{\"e}nov,
  \emph{Interpolation of Linear Operators}, Translations of Mathematical
  Monographs, \textbf{54}, American Mathematical Society, Providence, R.I., 1982.
    
\bibitem{KM1981}
J. L. Krivine and B. Maurey, \emph{Espaces des Banach stables}, Israel J. Math. \textbf{39}~(1981), no. 4, 273--295.

\bibitem{HanJu} H.J. Lee, \emph{Complex convexity and monotonicity in quasi-Banach lattices}, Israel J. Math. \textbf{159} (2007), 57--91.

\bibitem{Ler}C. Ler\'anoz, \emph{A remark on complex convexity}, Canad. Math. Bull. \textbf{31} (1988), no. 3, 322--324.
  
\bibitem{Lin} P. K. Lin, \textit{K\"othe-Bochner Function Spaces}, Birkh\"auser Boston Inc., Boston, MA, 2004.

\bibitem{LT1}
J.~Lindenstrauss and L.~Tzafriri, \emph{Classical
Banach Spaces I}, Springer-Verlag, 1977.

\bibitem{LT2}
J.~Lindenstrauss and L.~Tzafriri, \emph{Classical
Banach Spaces II}, Springer-Verlag, 1979.

\bibitem{LSZ2013} S. Lord, F. Sukochev and D. Zanin, \textit{Singular Traces: Theory and Applications}, De Gruyter Studies in Mathematics, \textbf{46}, De Gruyter, Berlin, 2013. 


\bibitem{L-PS} F. Lust-Piquard and F. A. Sukochev, \emph{$p$-Banach-Saks properties in symmetric operator spaces}, Illinois J. Math. \textbf{51}~(2007), no. 4, 1207--1229.


\bibitem{Mattila}
K. Mattila \emph{Complex strict and uniform convexity and hyponormal operators},
 Math. Proc. Cambridge Philos. Soc. \textbf{96} (1984),  no.~3, 483--493.

 \bibitem{chmc}
 C. McCarthy, \emph{$c_p$}, Israel J. Math. \textbf{5}~(1967), 249-271.
 
 \bibitem{Megg} R. Megginson, \emph{An Introduction to Banach Space Theory}, Graduate Texts in Mathematics, \textbf{183}, Springer-Verlag, New York, 1998.
 
\bibitem{Medz}
A. Medzitov, \emph{On separability of non-commutative symmetric spaces}, Matematicheskaya Analiz i Algebra, TashGU, Tashkent, (1986), 38--43 (in Russian). 

 \bibitem{Nelson}
E.~Nelson, \emph{Notes on non-commutative integration}, J. Funct. Anal. \textbf{15}~(1974), 103--116.

\bibitem{MN}
Jos\'e Luis Marcolino Nhany, \emph{La stabilit\'e des espaces 
$L^p$ non-commutatifs}, Math. Scand. \textbf{81}~(1997), no. 2, 212--218.

\bibitem{VN}
J. von Neumann, \emph{Collected works}, Vol. IV: Continuous geometry and other topics, General editor: A. H.
Taub, Pergamon Press, Oxford 1962.
  
   \bibitem{Ov1}
V.~ I.~Ov{\v{c}}innikov, \emph{Symmetric spaces of measurable operators}, Dokl. Akad. Nauk SSSR, \textbf{191} (1970), 769--771.

\bibitem{Ov2}
V. ~I.~Ov{\v{c}}innikov, \emph{Symmetric spaces of measurable operators}, Vorone\v z. Gos. Univ. Trudy Nau\v cn.-Issled. Inst. Mat. VGU, \textbf{3} (1971), 88--107. 

  \bibitem{Pietsch} A. Pietsch, \textit{History of Banach Spaces and Linear Operators}, Birkh\"auser Boston Inc., Boston, MA, 2007.
  
  \bibitem{PX}
G.  Pisier and Q. Xu, \emph{Non-commutative $L\sp p$-spaces},
 Handbook of the Geometry of Banach Spaces, Vol. 2, 1459--1517, North-Holland, Amsterdam, 2003.

  
  \bibitem{Rak79}
  S. A. Rakov, \emph{The Banach-Saks property for a Banach space}, Mat. Zametki \textbf{26}~(1979), no. 6, 823--834, 972 (in Russian). 
  
  \bibitem{rak1982}
  S. A. Rakov, \emph{Banach-Saks exponent of certain Banach spaces of sequences}, Mat. Zametki \textbf{32}~(1982), 613--625 (in Russian); English transl.: Math. Notes \textbf{32}~(1982), 791--797.

 
 \bibitem{Ryff} J. V. Ryff, \emph{Extreme points of some convex subsets of {$L^{1}(0,\,1)$}}, Proc. Amer. Math. Soc. \textbf{18}~(1967), 1026--1034.
 
 \bibitem{Ray1981}
Y. Raynaud, \emph{ Deux nouveaux exemples d'espaces de Banach stables},  C. R. Acad. Sci. Paris Sér. I Math. \textbf{292}~(1981), no. 15, 715--717.
 
 \bibitem{Ray1982}
 Y. Raynaud, \emph{Stabilit\'e des espaces $C_E$}, S\'eminaire de G\'eom\'etrie des Espaces de Banach, Univ. Paris 7 (1982--83).
 

\bibitem{Sch}
 R. Schatten, \emph{Norm Ideals of Completely Continuous Operators}, Springer-Verlag, 1960.
 
\bibitem{Seg}
I. Segal, \emph{A non-commutative extension of abstract integration}, Ann. Math \textbf{57}~(1953), 401--457.

 
 \bibitem{Sim} B. Simon, \emph{Trace Ideals and Their Applications}, second edition, Mathematical Surveys and Monographs \textbf{120}, American Mathematical Society, Providence, RI, 2005. 
 
 \bibitem{St}
W.~F. Stinespring, \emph{Integration theorems for gages and duality for
  unimodular groups}, Trans. Amer. Math. Soc. \textbf{90}~(1959), 15--56.
  
\bibitem{S1988} 
  F. ~A. Sukochev, \emph{Symmetric spaces of measurable operators on finite von Neumann algebras}, Ph.D. thesis, Tashkent State University, 1988.


\bibitem{S1992}
  F.~A. Sukochev, \emph{The MLUR property in symmetric (KB)-spaces (Russian)}, Mat. Zametki \textbf{52}~(1992), no. 6, 149--151; translation in Math. Notes \textbf{52}~(1992), no. 5-6, 1280--1282.

 \bibitem{S2} F. A.  Sukochev, \emph{On the uniform Kadec-Klee property with respect to convergence in measure}, J. Austral. Math. Soc. Ser. A \textbf{59} (1995), no. 3, 343--352. 
 
 \bibitem{S1}
F.  A. Sukochev, \emph{Linear-topological classification of separable $L\sb p$-spaces associated with von Neumann algebras of type I}, Israel J. Math. \textbf{115} (2000), 137--156.


   
\bibitem{S}
F.  A. Sukochev, \emph{Completeness of quasi-normed symmetric operator spaces}, Indag. Math. (N.S.) \textbf{25}~(2014), no. 2, 376--388.

 
 
 \bibitem{MR979385}
F.~A. Sukochev and V.~I. Chilin, \emph{The triangle inequality for operators
  that are measurable with respect to {H}ardy-{L}ittlewood order}, Izv. Akad.
  Nauk UzSSR Ser. Fiz.-Mat. Nauk (1988), no.~4, 44--50.
  
\bibitem{SC1990}
F.~A. Sukochev and V.~I. Chilin, \emph{ Convergence in measure in regular noncommutative symmetric spaces (Russian)},  Izv. Vyssh. Uchebn. Zaved. Mat. (1990), no. 9, 63--70; translation in Soviet Math. (Iz. VUZ) \textbf{34}~(1990), no. 9, 78--87.

\bibitem{SCH1990}
F.~A. Sukochev and V.~I. Chilin, \emph{Symmetric spaces over semifinite von Neumann
algebras}, Dokl. Akad. Nauk SSSR \textbf{313}~(1990), 811--815.

 
   \bibitem{Takesaki}
M. Takesaki, \emph{Theory of Operator Algebras. {I}}, Springer-Verlag,
  New York, 1979.
  
  \bibitem{TW}
E. Thorp and R. Whitley, \emph{The strong maximum modulus theorem for
  analytic functions into a {B}anach space}, Proc. Amer. Math. Soc. \textbf{18}~(1967), 640--646.
  
\bibitem{Tom1974} N. Tomczak-Jaegermann, \textit{The moduli of smoothness and convexity and the Rademacher averages of trace classes $S_p$ ($1\leq p<\infty$) }, Studia Math. \textbf{50}~(1974), 163--182. 

\bibitem{Tom1984} N. Tomczak-Jaegermann, \textit{Uniform convexity of unitary ideal}, Israel J. Math. \textbf{48}~(1984), no. 2-3, 249--254.

\bibitem{WT}
T. Wang and Y. Teng, \emph{Complex locally uniform rotundity of   {M}usielak-{O}rlicz spaces}, Sci. China Ser. A \textbf{43} (2000), no.~2,  113--121.

\bibitem{X1989} Q. Xu,  \textit{Convexit\'e uniforme des espaces sym\'etriques d'op\'erateurs mesurables}(French) [Uniform convexity of symmetric spaces of measurable operators], C. R. Acad. Sci. Paris Sér. I Math. \textbf{309}~(1989), no. 5, 251--254.


\bibitem{X1991} Q. Xu, \textit{Analytic functions with values in lattices and symmetric spaces of measurable operators}, Math. Proc. Cambridge Philos. Soc. \textbf{109}~(1991), 541--563.

\bibitem{X1992}
Q. Xu, \emph{Radon-Nikod\'ym property in symmetric spaces of measurable operators}, Proc. Amer. Math. Soc. \textbf{115} (1992), no. 2, 329--335.

\bibitem{Y1}
F.~ J. Yeadon, \emph{Convergence of measurable operators}, Math. Proc. Cambridge Philos. Soc. \textbf{74} (1973), 257--268.

\bibitem{Y2}
F.~ J. Yeadon, \emph{Non-commutative $L^p$-spaces}, Math. Proc. Cambridge Philos. Soc. \textbf{77} (1975), 91--102.

\bibitem{Y3}
F.~J. Yeadon, \emph{Ergodic theorems for semifinite von Neumann algebras: II}, Math. Proc. Cambridge
Philos. Soc. \textbf{88} (1980), no. 1, 135--147.

\bibitem{Z}
A. C. Zaanen, \emph{Integration}, North-Holland Publishing Co., Amsterdam, 1967.


\bibitem{ZYD}
Liu Zheng and Zhuang Ya Dong, \emph{$k$-Rotund complex normed linear spaces}, J. Math. Anal. Appl.  \textbf{146} (1990), no. 2, 540--545.
\end{thebibliography}

\end{document}